\documentclass[11pt,a4paper,leqno]{amsart} 

\usepackage{common} 

\title{Zygmund dilations: bilinear analysis and commutator estimates}
 \author{Emil Airta}
 \author{Kangwei Li}
 \author{Henri Martikainen}
 
 \address[E.A.]{Department of Mathematics and Statistics, University of Jyv\"askyl\"a, P.O. Box 35 (MaD),
 FI-40014 University of Jyv\"akyl\"a, Finland}
 \email{emil.airta@gmail.com}
 \address[K.L.]{Center for Applied Mathematics, Tianjin University, Weijin Road 92, 300072 Tianjin, China}
 \email{kli@tju.edu.cn}
 
 \address[H.M.]{Department of Mathematics and Statistics, Washington University in St. Louis, 1 Brookings Drive, St. Louis, MO 63130, USA}
 \email{henri@wustl.edu}
 \thanks{K. L. was supported by National Key R\&D Program of China (No. 2021YFA1002500), and the National Natural Science Foundation of China through project numbers 12222114 and 12001400.}
 
 \thanks{E.A. was supported by Academy of Finland through Grant No. 321896 ``Incidences on Fractals'' (PI = Orponen) and No. 314829 ``Frontiers of singular integrals'' (PI = Hyt\"onen).}
 \thanks{H.M. was  supported by the National Science Foundation under Grant No. 2247234 (PI = H. Martikainen). H.M. was, in addition,
supported by the Simons Foundation through MPS-TSM-00002361 (travel support for
mathematicians).}

\makeatletter
\@namedef{subjclassname@2010}{%
  \textup{2010} Mathematics Subject Classification}
\makeatother


\begin{document}

\allowdisplaybreaks

\begin{abstract}
  We develop both bilinear theory and commutator estimates
  in the context of entangled dilations, specifically Zygmund dilations $(x_1, x_2, x_3) \mapsto (\delta_1 x_1, \delta_2 x_2, \delta_1 \delta_2 x_3)$ in $\R^3$.
  We construct bilinear versions of recent dyadic multiresolution methods for Zygmund dilations and apply them
  to prove a paraproduct free $T1$ theorem for bilinear singular integrals invariant under Zygmund dilations.
  Independently,
  we prove linear commutator estimates even when the underlying singular integrals do not satisfy 
  weighted estimates with Zygmund weights. This requires new paraproduct estimates.
\end{abstract}

\maketitle

\section{Introduction}
``Entangled'' systems of dilations, see Nagel-Wainger \cite{NW}, in the $m$-parameter product space $\R^d = \prod_{i=1}^m \R^{d_i}$ have the general form
\begin{equation*}\label{eq:NWdil}
  (x_1, \ldots, x_m) \mapsto (\delta_1^{\lambda_{11}}\cdots\delta_k^{\lambda_{1k}} x_1, \ldots, \delta_1^{\lambda_{m1}}\cdots\delta_k^{\lambda_{mk}} x_m), \qquad \delta_1, \ldots, \delta_k > 0,
\end{equation*}
and appear naturally throughout analysis.
For instance, in $\R^3$ the {\em Zygmund dilations}
$(x_1, x_2, x_3) \mapsto (\delta_1 x_1, \delta_2 x_2, \delta_1 \delta_2 x_3)$
are compatible with the group law of the Heisenberg group, see e.g. M\"uller--Ricci--Stein \cite{MRS}.
Even these simplest entangled dilations are not completely understood, especially 
when it comes to the associated Calder\'on--Zygmund type singular integral operators (SIOs).

Until recently, multiresolution methods were still missing in the Zygmund dilations setting, as pointed out in \cite{DLOPW-ZYGMUND}.
This was a big restriction on how to go about developing singular integral theory.
However,
the last two authors together with
T. Hyt\"onen
and E. Vuorinen recently developed this missing Zygmund multiresolution analysis
in \textbf{\cite{HLMV-ZYG}}. Such dyadic representation theorems and related multiresolution techniques had been highly influential
in recent advances on SIOs and their applications (see e.g. \cites{HPW, Hy, MA_ADV, NTV}),
but developing them in the entangled situation required new ideas.
These tools then yielded very delicate weighted norm inequalities $L^p(w) \to L^p(w)$ for general non-convolution form Zygmund singular integrals in the optimal generality of Zygmund weights (introduced by Fefferman--Pipher \cite{FEPI}) 
\begin{equation*}\label{eq:zygweights}
    [w]_{A_{p,Z}} :=\sup_{I\in\mathcal{R}_Z}\Big(\frac{1}{|I|}\int_I w(x)\ud x\Big)\Big(\frac{1}{|I|} \int_I w^{-1/(p-1)}(x) \ud x\Big)^{p-1}<\infty, \qquad 1 < p < \infty,
\end{equation*}
where the supremum is over {\em Zygmund rectangles} $I = I_1\times I_2\times I_3$, $\ell(I_3) = \ell(I_1)\ell(I_2)$.

In fact, there is a precise threshold: if the kernel decay in terms of the deviation of $z\in\R^3$ from the ``Zygmund manifold'' $\abs{z_1z_2}=\abs{z_3}$
is not fast enough, singular integrals invariant under Zygmund dilations fail to be bounded with Zygmund weights. We constructed counterexamples
and showed the delicate positive result in the optimal range using the new multiresolution analysis. 
Previous results include \cites{RS,FEPI,DLOPW-ZYGMUND,HLLT}.

This rather striking threshold for weighted estimates means that
it is, in particular, unclear in what generality natural estimates for commutators
$[b, T] = bT - T(b\,\cdot\,)$ hold. 
Of course, ever since the classical one-parameter result of Coifman--Rochberg--Weiss \cite{CRW}, stating that
$\|[b,T]\|_{L^p \to L^p} \sim \|b\|_{\BMO}$, commutator estimates have been a large and fundamental part of the theory of SIOs and their applications.
Commutator estimates in the Zygmund dilation setting were previously considered in 
\cite{DLOPW-ZYGMUND} using the so-called Cauchy integral trick.
That method requires weighted bounds with Zygmund weights -- this is because 
it uses the fact that natural Zygmund adapted $\BMO$ functions generate Zygmund weights.
But we now know \cite{HLMV-ZYG}
that such weighted bounds are quite delicate -- and it turns out that the commutator bounds are 
true even in the regime where weighted estimates fail. We prove the following.
\begin{thm}\label{thm:main1}
  Let $b \in L^1_{loc}$ and $T$ be a linear paraproduct free Calder\'on-Zygmund operator adapted to Zygmund dilations
  as in \cite{HLMV-ZYG}. Let $\theta \in (0,1]$ be the kernel exponent measuring the decay in terms of the Zygmund ratio 
  \[D_\theta (x) := \Bigg(\frac{|x_1x_2|}{|x_3|} + \frac{|x_3|}{|x_1x_2|} \Bigg)^{-\theta}.\] Then for all such $\theta$ we have
  \[\|[b,T]\|_{L^p \to L^p} \lesssim \|b\|_{\bmo_Z}, \qquad 1 < p < \infty. \]
\end{thm}
As weighted estimates only hold with $\theta =1$,
this requires a proof
based on the multiresolution decomposition \cite{HLMV-ZYG} and a new family of ``Zygmund paraproducts''. 
Studying paraproducts is also interesting from the technical viewpoint that, generally,
proofs of $T1$ theorems display a structural decomposition of SIOs into their cancellative parts and paraproducts.
The new Zygmund theory in \cite{HLMV-ZYG} is designed for the fully cancellative case leaving out paraproducts and $\BMO$ considerations,
so this is the first paper, as far as we know, where paraproducts are considered in the Zygmund situation. 
They are tricky objects in the entangled situation.
However, while this is also a step
forward towards a full $T1$ theorem in the Zygmund setting,
the commutator theory that we develop does not require so-called partial paraproducts, and so the paraproduct
tools developed here are not yet sufficient to prove a $T1$ theorem in the non-cancellative case. We also mention that during our proof we include some results of independent interest, mainly, a new, extremely short proof
of the $A_{\infty}$ extrapolation theorem \cite{CUMP}.

Moving to a different direction, we push the Zygmund multiresolution methods \cite{HLMV-ZYG} to the multilinear setting and study 
bilinear SIOs invariant under Zygmund dilations.
A classical model of an $n$-linear SIO $T$ in $\R^d$ is obtained by setting
\begin{equation*}\label{eq:multilinHEUR}
T(f_1,\ldots, f_n)(x) = U(f_1 \otimes \cdots \otimes f_n)(x,\ldots,x), \qquad x \in \R^d,\, f_i \colon \R^d \to \C,
\end{equation*}
where $U$ is a linear SIO in $\R^{nd}$. See e.g. Grafakos--Torres \cite{GT} for the basic theory.
Estimates for classical multilinear SIOs play a fundamental role in pure and applied analysis -- 
for example, $L^p$ estimates for the homogeneous fractional derivative
$D^{\alpha} f=\mathcal F^{-1}(|\xi|^{\alpha} \widehat f(\xi))$
of a product of two or more functions, the fractional Leibniz rules,
are used in the area of dispersive equations, see e.g. Kato--Ponce \cite{KP} and Grafakos--Oh \cite{GO}.
We do not otherwise attempt to summarize the massive body of literature here and simply mention that the closest
existing result is perhaps \cite{LMV1}, which develops multiresolution methods in the non-entangled
multilinear bi-parameter case.

In this paper we prove the following ``paraproduct free'' $T1$ theorem for bilinear Zygmund SIOs.
\begin{thm}\label{thm:bilinearczz}
  Let $T$ be a bilinear paraproduct free Calder\'on-Zygmund operator adapted to Zygmund dilations as in Definition \ref{defn:czz}.
  Let $1 < p_1,p_2< \infty$ and $\frac{1}{2} < p < \infty$ with $\frac{1}{p} := \frac{1}{p_1} + \frac{1}{p_2}.$ Then we have
  $$
  \|T(f_1,f_2)\|_{L^p} \lesssim \|f_1\|_{L^{p_1}} \|f_2\|_{L^{p_2}}.
  $$
\end{thm}
Notice that we can conclude the full bilinear range, including the quasi-Banach range, just from the
paraproduct free $T1$ type assumptions. Also relevant is the fact that e.g. the appearing weak boundedness 
condition only involves Zygmund rectangles -- that is, the $T1$ assumptions of Definition \ref{defn:czz} are Zygmund adapted
and in this respect weaker than the corresponding tri-parameter assumptions.

It would also be very interesting to develop 
weighted theory with suitable kernel assumptions like in the linear case \cite{HLMV-ZYG}. That is, 
to generalize our recent paper \cite{LMV-GEN} from the standard multi-parameter setting 
to this entangled Zygmund setting. Recall that 
it would be key to deal with ``genuine'' multilinear weights, i.e., only impose a \emph{joint} $A_p$ condition
on the associated tuple of weights $\vec w = (w_1, \ldots, w_n)$. 
While such multilinear weighted estimates had been known for one-parameter SIOs for over 10 years
by the influential paper \cite{LOPTT}, the multi-parameter version was only recently solved in \cite{LMV-GEN}.
The entangled situation is very difficult, though, and we do not achieve such estimates in this paper.
Indeed, we are splitting our operators in a way that is sufficient for the unbounded estimates, but
not for the weighted estimates. In fact, already the unweighted estimates are surprisingly delicate
and the only way we found to achieve them was with using this additional decomposition and even some
sparse domination tools.

Here is an outline of the paper. In Section \ref{sec:mulrelanalysis} we develop the
fundamental Zygmund adapted multiresolution methods in the bilinear setting.
Section \ref{sec:ZSIO} introduces the singular integrals and the corresponding
testing conditions, and Section \ref{sec:coefficientesti} uses the kernel estimates 
to bound the various coefficients arising in the multiresolution analysis.
Section \ref{sec:structural} contains a further decomposition of our dyadic model 
operators -- this is then required in Section \ref{sec:boundedness}, where the $L^p$ estimates
of these model operators are proved. Section \ref{sec:boundedness} concludes with the proof of
Theorem \ref{thm:bilinearczz}. Section \ref{sec:commutator} contains the proof of the linear commutator estimates,
Theorem \ref{thm:main1}, and the corresponding theory of product and little $\BMO$ commutators 
in the Zygmund setting. Appendix \ref{app:multipliers}
considers bilinear variants of the multipliers studied by Fefferman-Pipher \cite{FEPI} --
this is motivation for the abstract definitions of Section \ref{sec:ZSIO}.

\section{Bilinear Zygmund multiresolution analysis}\label{sec:mulrelanalysis}
\subsection{Dyadic intervals, Zygmund rectangles and basic randomization}
Given a dyadic grid $\calD$, $I \in \calD$ and $k \in \Z$, $k \ge 0$, we use the following notation:
\begin{enumerate}
\item $\ell(I)$ is the side length of $I$.
\item $I^{(k)} \in \calD$ is the $k$th parent of $I$, i.e., $I \subset I^{(k)}$ and $\ell(I^{(k)}) = 2^k \ell(I)$.
\item $\ch(I)$ is the collection of the children of $I$, i.e., $\ch(I) = \{J \in \calD \colon J^{(1)} = I\}$.
\item $E_I f=\langle f \rangle_I 1_I$ is the averaging operator, where $\langle f \rangle_I = \fint_{I} f = \frac{1}{|I|} \int _I f$.
\item $\Delta_I f$ is the martingale difference $\Delta_I f= \sum_{J \in \ch (I)} E_{J} f - E_{I} f$.
\item $\Delta_{I,k} f$ or $\Delta_{I}^k f$ is the martingale difference block
$$
\Delta_{I,k} f= \Delta_I^k f = \sum_{\substack{J \in \calD \\ J^{(k)}=I}} \Delta_{J} f.
$$
\end{enumerate}

We will have use for randomization soon. While often the grids are fixed and we suppress the dependence on the random parameters,
it will be important to understand the definitions underneath. So we go ahead and introduce the related notation and standard results now.
Let $\calD_0$ be the standard dyadic grid in $\R$. For $\omega \in \{0,1\}^{\Z}$, $\omega = (\omega_i)_{i \in \Z}$, 
we define the shifted lattice
$$
\calD(\omega) := \Big\{L + \omega := L + \sum_{i\colon 2^{-i} < \ell(L)} 2^{-i}\omega_i \colon L \in \calD_0\Big\}.
$$
Let $\bbP_{\omega}$ be the product probability measure on $\{0,1\}^{\Z}$. 
We recall the following notion of a good interval from \cite{GH}.
We say that $G \in \calD(\omega, k)$, $k \ge 2$,
if $G \in \calD(\omega)$ and
\begin{equation}\label{eq:DefkGood}
d(G, \partial G^{(k)}) \ge \frac{\ell(G^{(k)})}{4} = 2^{k-2} \ell(G).
\end{equation}
Notice that for all $L \in \calD_0$ and $k \ge 2$ we have
\begin{equation}\label{eq:gprob}
  \bbP_{\omega}( \{ \omega\colon L + \omega \in \calD(\omega, k) \})  = \frac{1}{2}.
\end{equation}
The key implication (of practical use later) of $G \in \calD(\omega, k)$ is that for $n \in \Z$ with $|n| \le 2^{k - 2}$ we have
\begin{equation}\label{eq:kparent}
(G \dotplus n)^{(k)} = G^{(k)}, \qquad G \dotplus n := G + n\ell(G).
\end{equation}

In fact, we will not need much more of randomization -- it only remains to move the notation 
to our actual setting of $\R^3 = \R \times \R^2$. We define for
$$
\sigma = (\sigma^1, \sigma^2, \sigma^3) \in \{0,1\}^{\Z} \times \{0,1\}^{\Z} \times \{0,1\}^{\Z}
$$
that
\begin{align*}
  \calD(\sigma) &:= \calD(\sigma^1) \times \calD(\sigma^2) \times \calD(\sigma^3).
\end{align*}
Let $$
\bbP_{\sigma} := \bbP_{\sigma^1} \times \bbP_{\sigma^2} \times \bbP_{\sigma^3}.
$$
For $k = (k^1, k^2, k^3)$, $k^1, k^2, k^3 \ge 2$, we define
$$
\calD(\sigma, k) = \calD(\sigma^1, k^1) \times \calD(\sigma^2, k^2) \times \calD(\sigma^3, k^3).
$$
We also e.g. write
$$
\calD(\sigma, (k^1, 0, k^3)) = \calD(\sigma^1, k^1) \times \calD(\sigma^2) \times \calD(\sigma^3, k^3),
$$
that is, a $0$ will designate that we do not have goodness in that parameter.

As for most of the argument $\sigma$ is fixed, it makes sense to mainly suppress it from the notation and
abbreviate, whenever possible, that
$$
\calD^m = \calD(\sigma^m), \qquad \calD(\sigma^m, k^m)=\calD^m(k^m), \qquad m = 1,2,3.
$$ Then also
$$
\calD = \calD(\sigma) = \prod_{m=1}^3 \calD^m,\qquad \calD(k)=\prod_{m=1}^3 \calD^m(k^m) .
$$ 
We define the Zygmund rectangles $\calD_Z \subset \calD$
by setting
\begin{equation}\label{zyglat}
\calD_Z = \Big\{I = \prod_{m=1}^3 I^m \in \calD \colon \ell(I^1) \ell(I^2) = \ell(I^3)\Big\}.
\end{equation}
Obviously, $\calD_Z(k)$ is defined similarly as above but also requires $\prod_{m=1}^3 I^m \in \calD(k)$.

\subsection{Zygmund martingale differences}
Given $I = \prod_{m=1}^3 I^m$ we define the Zygmund martingale difference operator
$$
\Delta_{I,Z} f := \Delta_{I^1} \Delta_{I^2 \times I^3} f.
$$
\begin{rem}
We highlight that the martingale difference $\Delta_{I^2 \times I^3}$ is the \textbf{one-parameter} (and not the bi-parameter) martingale difference on the rectangle $I^2 \times I^3$:
$$
\Delta_{I^2 \times I^3} = \Delta_{I^2} \Delta_{I^3}  + E_{I^2} \Delta_{I^3} + \Delta_{I^2} E_{I^3} \ne \Delta_{I^2} \Delta_{I^3}.
$$
Moreover, the above operators
really act on the full product space but only on the given parameters -- for instance, $\Delta_{I^1} f(x_1, x_2, x_3)
= \Delta_{I^1}^1 f(x_1, x_2, x_3) =  (\Delta_{I^1} f(\cdot, x_2, x_3))(x_1)$.
\end{rem}
We recall the following facts from \cite{HLMV-ZYG}. 
For a dyadic $\lambda > 0$ define the dilated lattices
$$
\calD_{\lambda}^{2,3} = \{I^{2,3}\in \calD^{2,3} := \calD^2 \times \calD^3 \colon \ell(I^3) = \lambda \ell(I^2)\}.
$$
The basic Zygmund expansion goes as follows:
\begin{equation}\label{zygexp}
\begin{split}
f = \sum_{I^1 \in \calD^1} \Delta_{I^1} f = \sum_{I^1 \in \calD^1} \sum_{I^{2,3} \in \calD^{2,3}_{\ell(I^1)}} 
\Delta_{I^1} \Delta_{I^{2,3}}  f
= \sum_{I \in \calD_Z} \Delta_{I, Z} f.
\end{split}
\end{equation}
However, the way we split our operators will not be this simple.

The following basic results hold for the martingale differences.
For $I, J \in \calD_Z$ we have
$$
  \Delta_{I,Z} \Delta_{J,Z} f = \left\{ \begin{array}{ll}
  \Delta_{I,Z} & \textup{if } I = J, \\
  0 & \textup{if } I \ne J.
  \end{array} \right.
$$
Notice also that the Zygmund martingale differences satisfy
$$
\int_{\R} \Delta_{I,Z}f \ud x_1 = 0 \qquad \textup{and} \qquad \int_{\R^2} \Delta_{I,Z}f \ud x_2 \ud x_3 = 0.
$$
Moreover, we have
$$
\int (\Delta_{I, Z} f) g = \int f \Delta_{I,Z} g.
$$

\subsection{Haar functions}
For an interval $J \subset \R$ we denote by $J_{l}$ and $J_{r}$ the left and right
halves of $J$, respectively. We define
$$
h_{J}^0 = |J|^{-1/2}1_{J} \qquad \textup{and} \qquad h_{J}^1 = h_{J} = |J|^{-1/2}(1_{J_{l}} - 1_{J_{r}}).
$$
The reader should carefully notice that $h_I^0$ is the non-cancellative Haar function for us and that
in some other papers a different convention is used.

As we mostly work on $\R^3 = \R \times \R^2$ we require some Haar functions on $\R^2$ as well.
For $I^2 \times I^3 \subset \R^2$ and
$\eta = (\eta_2, \eta_3) \in \{0,1\}^2$ define 
$$
h_{I^2 \times I^3}^{\eta} = h_{I^2}^{\eta_2} \otimes h_{I^3}^{\eta_3}.
$$
Similarly, as $h_{I^1}$ denotes a cancellative Haar function on $\R$,
we let $h_{I^2 \times I^3}$ denote a cancellative \textbf{one-parameter} Haar function
on $I^2 \times I^3$. This means that
$$
h_{I^2 \times I^3} = h_{I^2 \times I^3}^{\eta}
$$
for some $\eta = (\eta_2, \eta_3) \in \{0,1\}^2 \setminus \{(0,0)\}$. We only 
use a $0$ to denote a non-cancellative Haar function: $h_{I^2 \times I^3}^0 = 
h_{I^2 \times I^3}^{(0,0)}$.

We suppress this $\eta$ dependence
in all that follows in the sense that a finite $\eta$ summation is not written. For example,
given $I = I^1 \times I^2 \times I^3 \in \calD_Z \subset \prod_{m=1}^3 \calD^m$ decompose
$$
\Delta_{I, Z}f = \Delta_{I^1} \Delta_{I^2 \times I^3} f = \langle f, h_{I^1} \otimes h_{I^2 \times I^3} \rangle h_{I^1} \otimes h_{I^2 \times I^3} =: 
\langle f, h_{I,Z} \rangle h_{I,Z}.
$$

\subsection{Bilinear Zygmund shifts}\label{subsec:shifts}
In preparation for defining the shifts, we define the following notation.
Let $I_1, I_2, I_3$ be rectangles, $I_j = I^1_j \times I^2_j \times I^3_j = I^1_j \times I^{2,3}_j$, and $f_1, f_2, f_3$ be functions
defined on $\R^3$.
For $j_1, j_2 \in \{1,2,3\}$ define
\begin{align*}
A^{j_1,j_2}_{I_1,I_2,I_3} = A^{j_1,j_2}_{I_1,I_2,I_3}(f_1,f_2,f_3)  := \prod_{j = 1}^3 \ave{f_j, v_{I_j}},
\end{align*}
where
\begin{align*}
  v_{I_j} &= \wt h_{I_j^1} \otimes \wt h_{I_j^{2,3}}; \\
  \wt h_{I_{j_1}^1} &= h_{I_{j_1}^1} \qquad \textup{and} \qquad \wt h_{I_{j}^1} = h_{I_{j}^1}^0,\, j\neq j_1; \\
  \wt h_{I_{j_2}^{2,3}} &= h_{I_{j_2}^{2,3}} \qquad \textup{and} \qquad \wt h_{I_{j}^{2,3}} = h_{I_{j}^{2,3}}^0,\, j\neq j_2.
\end{align*}
For a dyadic $\lambda > 0$ define
$$
\calD_\lambda = \{K = K^1\times K^2\times K^3 \in \calD \colon \lambda\ell(K^1)\ell(K^2) = \ell(K^3) \}.
$$
Moreover, for a rectangle $I = I^1 \times I^2 \times I^3$ and $k = (k^1, k^2,k^3)$ define
$$
I^{(k)} = I_1^{(k^1)} \times I_2^{(k^2)} \times I_3^{(k^3)}.
$$
\begin{defn}
  Let $k = (k^1, k^2,k^3)$, $k^i \in \{0,1,2,\ldots\}$, be fixed. A bilinear Zygmund shift $Q = Q_k$
  of complexity $k$ has the form
  \begin{align*}
    &\ave{Q_{k}(f_1,f_2),f_3} \\
    &= \sum_{K \in \calD_{2^{-k^1-k^2+k^3}}}\sum_{\substack{I_1,I_2,I_3 \in \calD_{Z}\\ I_j^{(k)} = K}} a_{K,(I_j)} \Big[A^{j_1,j_2}_{I_1,I_2,I_3} - A^{j_1,j_2}_{I_{j_1}^1 \times I_1^{2,3},I_{j_1}^{1} \times I_2^{2,3}, I_{j_1}^{1} \times I_3^{2,3}}\\
    &\hspace{10em} -A^{j_1,j_2}_{I_{1}^1 \times I_{j_2}^{2,3},I_{2}^{1} \times I_{j_2}^{2,3}, I_{3}^{1} \times I_{j_2}^{2,3}} +  A^{j_1,j_2}_{I_{j_1}^1 \times I_{j_2}^{2,3},I_{j_1}^{1} \times I_{j_2}^{2,3}, I_{j_1}^{1} \times I_{j_2}^{2,3}} \Big]
  \end{align*}
  for some $j_1, j_2 \in \{1,2,3\}$. The coefficients $a_{K,(I_j)}$ satisfy
  $$
  |a_{K,(I_j)}| \le \frac{|I_1|^{1/2}|I_2|^{1/2}|I_3|^{1/2}}{|K|^2} = \frac{|I_1|^{3/2}}{|K|^2}.
  $$
\end{defn}
Now, the game is to represent bilinear singular integrals using the operators $Q_k$ and also -- independently -- 
bound the operators $Q_k$ suitably. We start with the representation part and deal with bounding the operators later.
We have not defined our singular integrals carefully yet, however,
a lot of the required decomposition can be formally carried out for an arbitrary operator $T$. The singular integral 
part is later required to get sufficient decay for the appearing scalar coefficients and to handle the paraproducts.

\subsection{Zygmund decomposition of $\langle T(f_1, f_2), f_3\rangle$}\label{sec:ZygDeco}
For now, we focus on the multiresolution part and start formally decomposing a general bilinear operator.
We begin by writing $\langle T(f_1, f_2), f_3\rangle$ as 
\begin{align*}
  \sum_{I^1_1, I^1_2, I^1_3 \in \calD^1}& \langle T(\Delta_{I^1_1} f_1, \Delta_{I^1_2} f_2), \Delta_{I^1_3} f_3\rangle \\
  &= \sum_{\substack{I^1_1, I^1_2, I^1_3 \in \calD^1 \\ \ell(I^1_1), \ell(I^1_2) > \ell(I^1_3)  }} \langle T(\Delta_{I^1_1} f_1, \Delta_{I^1_2} f_2), \Delta_{I^1_3} f_3\rangle \\
  &+ \sum_{\substack{I^1_1, I^1_2, I^1_3 \in \calD^1 \\ \ell(I^1_1), \ell(I^1_3) > \ell(I^1_2)  }} \langle T(\Delta_{I^1_1} f_1, \Delta_{I^1_2} f_2), \Delta_{I^1_3} f_3\rangle \\
  &+ \sum_{\substack{I^1_1, I^1_2, I^1_3 \in \calD^1 \\ \ell(I^1_2), \ell(I^1_3) > \ell(I^1_1)  }} \langle T(\Delta_{I^1_1} f_1, \Delta_{I^1_2} f_2), \Delta_{I^1_3} f_3\rangle \\
  &+ \sum_{\substack{I^1_1, I^1_2, I^1_3 \in \calD^1 \\ \ell(I^1_1) > \ell(I^1_2) = \ell(I^1_3)  }} \langle T(\Delta_{I^1_1} f_1, \Delta_{I^1_2} f_2), \Delta_{I^1_3} f_3\rangle \\
  &+ \sum_{\substack{I^1_1, I^1_2, I^1_3 \in \calD^1 \\ \ell(I^1_2) > \ell(I^1_1) = \ell(I^1_3)  }} \langle T(\Delta_{I^1_1} f_1, \Delta_{I^1_2} f_2), \Delta_{I^1_3} f_3\rangle \\
  &+ \sum_{\substack{I^1_1, I^1_2, I^1_3 \in \calD^1 \\ \ell(I^1_3) > \ell(I^1_1) = \ell(I^1_2)  }} \langle T(\Delta_{I^1_1} f_1, \Delta_{I^1_2} f_2), \Delta_{I^1_3} f_3\rangle \\
  &+ \sum_{\substack{I^1_1, I^1_2, I^1_3 \in \calD^1 \\ \ell(I^1_1) = \ell(I^1_2) = \ell(I^1_3)  }} \langle T(\Delta_{I^1_1} f_1, \Delta_{I^1_2} f_2), \Delta_{I^1_3} f_3\rangle.
\end{align*}
We collapse the first six sums, which are not already diagonal sums, into diagonal sums
$$
\sum_{\substack{I^1_1, I^1_2, I^1_3 \in \calD^1 \\ \ell(I^1_1) = \ell(I^1_2) = \ell(I^1_3)  }}.
$$
This has the effect that whenever we have an inequality $\ell(I^1_i) > \ell(I^1_j)$, the martingale difference operator
$\Delta_{I^1_i}$ corresponding with the larger cube is changed to the averaging operator $E_{I^1_i}$. Thus, in the first three
sums we now have two averaging operators, and in the next three we have one averaging operator. The more averaging operators we have, the less cancellation we have,
and thus the main challenge are the first three sums with the least cancellation. We mainly focus on the first three sums for this reason.

In addition, the first three sums are symmetric, so we may focus on only one of them, and choose to look at
$$
\sum_{\substack{I^1_1, I^1_2, I^1_3 \in \calD^1 \\ \ell(I^1_1), \ell(I^1_2) > \ell(I^1_3)  }} \langle T(\Delta_{I^1_1} f_1, \Delta_{I^1_2} f_2), \Delta_{I^1_3} f_3\rangle
= \sum_{\substack{I^1_1, I^1_2, I^1_3 \in \calD^1 \\ \ell(I^1_1) = \ell(I^1_2) = \ell(I^1_3)  }}
\langle T(E_{I^1_1} f_1, E_{I^1_2} f_2), \Delta_{I^1_3} f_3\rangle.
$$
Now, we fix $I^1_1, I^1_2, I^1_3 \in \calD^1$ with $\ell(I^1_1) = \ell(I^1_2) = \ell(I^1_3)$ and
repeat the argument for $\langle T(E_{I^1_1} f_1, E_{I^1_2} f_2), \Delta_{I^1_3} f_3\rangle$
using the lattice $\calD^{2,3}_{\ell(I^1)}$, where recall that 
for a dyadic $\lambda > 0$ we have
\begin{equation*}
\calD_{\lambda}^{2,3} = \{I^2 \times I^3 \in \calD^{2,3} := \calD^2 \times \calD^3 \colon \ell(I^3) = \lambda \ell(I^2)\}.
\end{equation*}

This produces seven terms, and we again focus on
$$
\sum_{\substack{I^2_1 \times I^3_1, I^2_2 \times I^3_2, I^2_3 \times I^3_3 \in \calD^{2,3}_{\ell(I^1)} \\ \ell(I^2_1) = \ell(I^2_2) = \ell(I^2_3)  }}
\langle T(E_{I^1_1} E_{I^2_1 \times I^3_1} f_1, E_{I^1_2} E_{I^2_2 \times I^3_2} f_2), \Delta_{I^1_3} \Delta_{I^2_3 \times I^3_3} f_3\rangle.
$$
Altogether, our focus, for now, is on the key term
\begin{equation}\label{eq:key}
  \sum_{\substack{I_1, I_2, I_3 \in \calD_Z \\ \ell(I_1) = \ell(I_2) = \ell(I_3) }}
  \langle T(E_{I_1} f_1, E_{I_2} f_2), \Delta_{I_3, Z} f_3 \rangle,
\end{equation}
where $\ell(I_1) = \ell(I_2) = \ell(I_3)$ means that
$$
\ell(I_1^m) = \ell(I_2^m) = \ell(I_3^m), \qquad m = 1,2,3.
$$
This was completely generic -- we now go a step further to the direction of Zygmund shifts and start introducing 
Haar functions into the mix.

\subsection{Further decomposition of \eqref{eq:key}}\label{sec:ZygDecoRefined}
Write
$$
\langle T(E_{I_1} f_1, E_{I_2} f_2), \Delta_{I_3, Z} f_3 \rangle = 
\langle T(h_{I_1}^0, h_{I_2}^0), h_{I_3, Z} \rangle \langle f_1, h_{I_1}^0 \rangle
\langle f_2, h_{I_2}^0 \rangle \langle f_3, h_{I_3, Z} \rangle.
$$
Now, we perform a rather complicated decomposition of the product $\langle f_1, h_{I_1}^0 \rangle \langle f_2, h_{I_2}^0 \rangle$.
To this end, start by writing
\begin{align*}
  &\langle f_1, h_{I_1}^0 \rangle \langle f_2, h_{I_2}^0 \rangle \\ 
  &= \Big[\langle f_1, h_{I_1}^0 \rangle \langle f_2, h_{I_2}^0 \rangle - \langle f_1, h_{I^1_3}^0 h_{I^{2,3}_1}^0 \rangle \langle f_2, h_{I^1_3}^0 h_{I^{2,3}_2}^0 \rangle\Big]
  + \langle f_1, h_{I^1_3}^0 h_{I^{2,3}_1}^0 \rangle \langle f_2, h_{I^1_3}^0 h_{I^{2,3}_2}^0 \rangle \\
  &=: A_1 + A_2.
\end{align*}
We then further decompose $A_1$ as follows
\begin{align*}
  A_1 = \Big[\langle f_1, h_{I_1}^0 \rangle \langle f_2, h_{I_2}^0 \rangle &- \langle f_1, h_{I^1_3}^0 h_{I^{2,3}_1}^0 \rangle \langle f_2, h_{I^1_3}^0 h_{I^{2,3}_2}^0 \rangle \\
  &- \langle f_1, h_{I^1_1}^0 h_{I^{2,3}_3}^0 \rangle \langle f_2, h_{I^1_2}^0 h_{I^{2,3}_3}^0 \rangle 
  + \langle f_1, h_{I^1_3}^0 h_{I^{2,3}_3}^0 \rangle \langle f_2, h_{I^1_3}^0 h_{I^{2,3}_3}^0 \rangle\Big] \\
  &+ \Big\{ \langle f_1, h_{I^1_1}^0 h_{I^{2,3}_3}^0 \rangle \langle f_2, h_{I^1_2}^0 h_{I^{2,3}_3}^0 \rangle - \langle f_1, h_{I^1_3}^0 h_{I^{2,3}_3}^0 \rangle \langle f_2, h_{I^1_3}^0 h_{I^{2,3}_3}^0 \rangle \Big\}.
\end{align*}

When we later specialize to singular integrals, we will in particular make the following assumption.
We say that $T$ is a paraproduct free operator, if for all cancellative Haar functions $h_{I^1}$ and $h_{I^{2,3}}$ we have
\begin{align*}
 & \ave{T(1\otimes 1_{J_1^{2,3}},1\otimes 1_{J_2^{2,3}}),h_{I^1} \otimes 1_{J_3^{2,3}}} = \ave{T_1^{*,j}(1\otimes 1_{J_1^{2,3}},1\otimes 1_{J_2^{2,3}}),h_{I^1} \otimes 1_{J_3^{2,3}}} \\
 &\quad =\ave{T(1_{I_1^1}\otimes 1, 1_{I_2^1}\otimes 1), 1_{I_3^1}\otimes h_{I^{2,3}}}=\ave{T_{2,3}^{*, j}(1_{I_1^1}\otimes 1, 1_{I_2^1}\otimes 1), 1_{I_3^1}\otimes h_{I^{2,3}}}=0
\end{align*}
for all the adjoints $j \in \{1,2\}$. 
With this assumption in the full summation \eqref{eq:key} everything else vanishes except
\begin{align*}
  \sum_{\substack{I_1, I_2, I_3 \in \calD_Z \\ \ell(I_1) = \ell(I_2) = \ell(I_3) }}
  \langle T(h_{I_1}^0, h_{I_2}^0),& h_{I_3, Z} \rangle  \Big[\langle f_1, h_{I_1}^0 \rangle \langle f_2, h_{I_2}^0 \rangle - \langle f_1, h_{I^1_3 \times I^{2,3}_1}^0 \rangle \langle f_2, h_{I^1_3 \times I^{2,3}_2}^0 \rangle \\
  &- \langle f_1, h_{I^1_1 \times I^{2,3}_3}^0 \rangle \langle f_2, h_{I^1_2 \times I^{2,3}_3}^0 \rangle 
  + \langle f_1, h_{I_3}^0 \rangle \langle f_2, h_{I_3}^0 \rangle\Big] \langle f_3, h_{I_3, Z} \rangle.
\end{align*}
So we eliminated the paraproducts by assumption, and now we have to manipulate this 
remaining term to a suitable form involving shifts.

In the above sum we will relabel $I_3 = I = I^1 \times I^2 \times I^3 = I^1 \times I^{2,3}$. Then, for
$n_1 = (n_1^1, n_1^2, n_1^3) = (n_1^1, n_1^{2,3})$
we write
$$
I_1 = I \dotplus n_1 
= (I^1 + n_1^1\ell(I^1)) \times (I^2 + n_1^2\ell(I^2)) \times (I^3 + n_1^3\ell(I^3))
=  (I^1 \dotplus n_1^1) \times (I^{2,3} \dotplus n_1^{2,3}).
$$
We write $I_2$ similarly as $I_2 = I \dotplus n_2$. Notice that if $n_1^1 = n_2^1 = 0$, then
the term inside the summation vanishes. Similarly, if $n_1^{2,3} = n_2^{2,3} = (0, 0)$, the term 
inside the summation vanishes. So we need to study
$$
\sum_{\substack{ n_1, n_2 \in \Z^3 \\ \max (|n_1^1|, |n_2^1|) \ne 0 \\ \max (|n_1^2|, |n_2^2|) \ne 0 \textup{ or } \max (|n_1^3|, |n_2^3|) \ne 0  }} 
\sum_{I \in \calD_Z} c_{I, n_1, n_2},
$$
where
\begin{align*}
  &c_{I, n_1, n_2}  \\
  &=   \langle T(h_{I \dotplus n_1}^0, h_{I \dotplus n_2}^0),h_{I, Z} \rangle  \Big[\langle f_1, h_{I \dotplus n_1}^0 \rangle \langle f_2, h_{I \dotplus n_2}^0 \rangle - \langle f_1, h_{I^1 \times (I^{2,3} \dotplus n_1^{2,3}) }^0 \rangle \langle f_2, h_{I^1 \times (I^{2,3} \dotplus n_2^{2,3})}^0 \rangle \\
  &- \langle f_1, h_{(I^1 \dotplus n_1^1) \times I^{2,3}}^0 \rangle \langle f_2, h_{(I^1 \dotplus n_2^1) \times I^{2,3}}^0 \rangle 
  + \langle f_1, h_{I}^0 \rangle \langle f_2, h_{I}^0 \rangle\Big] \langle f_3, h_{I, Z} \rangle.
\end{align*}

We write
\begin{align*}
&\sum_{\substack{ n_1, n_2 \in \Z^3 \\ \max\limits_{j=1,2} |n_j^1|\ne 0 \\ \max\limits_{j=1,2} |n_j^2| \ne 0 \textup{ or } \max\limits_{j=1,2} |n_j^3| \ne 0  }} 
\sum_{I \in \calD_Z} c_{I, n_1, n_2}\\
 &= \sum_{k^1, k^2, k^3 = 2}^\infty \sum_{\substack{ n_1, n_2 \in \Z^3 \\ \max\limits_{j  =1,2}|n_j^m| \in (2^{k^m-3}, 2^{k^m - 2}]\\ m=1,2,3}}
\sum_{I \in \calD_Z} c_{I, n_1, n_2} \\&\qquad+  \sum_{k^1, k^2 = 2}^\infty \sum_{\substack{ n_1, n_2 \in \Z^3 \\  \max\limits_{j  =1,2}|n_j^m| \in (2^{k^m-3}, 2^{k^m - 2}]\\ m=1,2 \\ n_1^3=n_2^3 = 0}} 
\sum_{I \in \calD_Z} c_{I, n_1, n_2}  \\
&\qquad+ \Sigma_{sym},\\
\end{align*}
where $\Sigma_{sym}$ is symmetric to the second term and has $n^2_1 = n^2_2 = 0$.

Recall how everything implicitly depends on the random parameter $\sigma$, so that we can average over it.
By independence, like in \cite{HLMV-ZYG}, we have by \eqref{eq:gprob} that
\begin{equation}\label{eq:maingood}
  \begin{split}
    & \E_{\sigma} \sum_{k^1, k^2, k^3 = 2}^\infty \sum_{\substack{ n_1, n_2 \in \Z^3 \\ \max\limits_{j  =1,2}|n_j^m| \in (2^{k^m-3}, 2^{k^m - 2}]\\ m=1,2,3}}
  \sum_{I \in \calD_Z} c_{I, n_1, n_2} \\
  &= 8 \E_{\sigma} \sum_{k^1, k^2, k^3 = 2}^\infty \sum_{\substack{ n_1, n_2 \in \Z^3 \\ \max\limits_{j  =1,2}|n_j^m| \in (2^{k^m-3}, 2^{k^m - 2}]\\ m=1,2,3}}
  \sum_{I \in \calD_{Z}(k)} c_{I, n_1, n_2}, \qquad k = (k^1, k^2, k^3).
  \end{split}
\end{equation}

For the other two terms, where $n_j^2=0$ or $n_j^3 = 0$, we perform the above but do not add goodness to the second
and third parameters, respectively. 
For example, we have 
\begin{align*}
&\E_{\sigma} \sum_{k^1, k^2 = 2}^\infty \sum_{\substack{ n_1, n_2 \in \Z^3 \\  \max\limits_{j  =1,2}|n_j^m| \in (2^{k^m-3}, 2^{k^m - 2}]\\ m=1,2 \\ n_1^3=n_2^3 = 0}} \sum_{I \in \calD_{Z}} c_{I, n_1, n_2}\\
 &= 4\E_{\sigma} \sum_{k^1, k^2 = 2}^\infty \sum_{\substack{ n_1, n_2 \in \Z^3 \\  \max\limits_{j  =1,2}|n_j^m| \in (2^{k^m-3}, 2^{k^m - 2}]\\ m=1,2 \\ n_1^3=n_2^3 = 0}} \sum_{I \in \calD_{Z}(k^1,k^2, 0)} c_{I, n_1, n_2}.
\end{align*}

Continuing with \eqref{eq:maingood}, we write it  as
\begin{align*}
C 8 \E_{\sigma}& \sum_{k^1, k^2, k^3 = 2}^\infty (|k| +1)^{2} \varphi(k)\\
  & \sum_{K \in \calD_\lambda}\sum_{\substack{I \in \calD_{Z}(k)\\ I^{(k)} = K}} \sum_{\substack{ n_1, n_2 \in \Z^3 \\ \max_{j  =1,2}|n_j^m| \in (2^{k^m-3}, 2^{k^m - 2}]\\ m=1,2,3}}
 \frac{c_{I, n_1, n_2}}{C(|k| +1)^2 \varphi(k)} ,
\end{align*}
where
$$
\calD_\lambda = \{K = K^1\times K^2\times K^3 \in \calD \colon \lambda\ell(K^1)\ell(K^2) = \ell(K^3) \},
\qquad \lambda = 2^{k^3 -k^1 - k^2 },
$$ and $C$ is some suitably large constant depending on $T$.
Recall that by \eqref{eq:kparent} we also have
$$
(I \dotplus n_1)^{(k)} = (I \dotplus n_2)^{(k)} = I^{(k)} = K.
$$
We have arrived to a point where we cannot go further without talking about singular integrals. Indeed,
we need kernel estimates to control the coefficients. But on a structural level (with the paraproduct free assumption), 
we have obtained a reasonable representation of the main term \eqref{eq:key} in terms of sums of bilinear Zygmund shifts.

\section{Bilinear Zygmund singular integrals}\label{sec:ZSIO}
We begin by defining the required kernel estimates and cancellation conditions
for bilinear singular integrals $T$ invariant under Zygmund dilations. For motivation
for the form of the kernel estimates, see Appendix \ref{app:multipliers} for kernel
bounds of bilinear multipliers. This viewpoint makes the kernel estimates natural --
on the other hand, they are also of the right form so that we will be able to bound the 
coefficients from the multiresolution decomposition and obtain reasonable decay.

\subsection{Full kernel representation}
Our bilinear singular integral $T$ invariant under Zygmund dilations is related to a full kernel $K$ in the following way. 
The kernel $K$ is  a function
$$
K \colon (\R^3 \times \R^3 \times \R^3) \setminus \Delta \to \C,
$$
where
$$
\Delta=\{(x,y,z) \in \R^3 \times \R^3 \times \R^3 \colon  x_i = y_i = z_i \text{ for at least one } i = 1,2,3\}. 
$$
We look at the action of $T$ on rectangles like $I^1 \times I^2 \times I^3=: I^1 \times I^{2,3}$ in $\R^3 = \R \times \R \times \R = \R \times \R^2$.
So let $I_i = I^1_i \times I^2_i \times I^3_i$ be rectangles, $i = 1,2,3$.
Assume that there exists  $i_1, i_2,j_1,j_2\in \{1,2,3\}$ so that $I_{i_1}^1$ and  $I_{i_2}^1$ are disjoint and also $I_{j_1}^{2,3}$ and $I_{j_2}^{2,3}$ are disjoint.
Then we have the full kernel representation
$$
\ave {T(1_{I_1}, 1_{I_2}),1_{I_3}}
= \iiint K(x,y,z) 1_{I_1}(x) 1_{I_2}(y)1_{I_3}(z) \ud x\ud y \ud z.
$$
The kernel $K$ satisfies the following estimates.

First, we define the decay factor
$$
D_\theta(x,y) = \Bigg(\frac{\prod_{i=1}^2  (|x_i| + |y_i|)}{|x_3| + |y_3|}+\frac{|x_3| + |y_3|}{\prod_{i=1}^2  (|x_i| + |y_i|)}\Bigg)^{-\theta},\qquad \theta\in (0,2],
$$ 
and the tri-parameter bilinear size factor
$$
S(x , y) = \prod_{i=1}^3\frac{1}{\big(|x_i| + |y_i|\big)^2}.
$$
We demand the following size estimate
\begin{equation}
|K(x,y,z)|
\lesssim D_\theta(x - z,y - z) S(x-z,y-z).
\end{equation}


Let now $c=(c_1,c_2,c_3)$ be such that $|c_i- x_i| \le\max(|x_i - z_i|, |y_i - z_i|)/2$ for $i=1,2,3$. We assume that $K$ satisfies the mixed size and H\"older estimates
\begin{equation}
\begin{split}
|K((c_1,&x_2,x_3),y,z)-K(x,y,z)|\\
&\lesssim  \Big(\frac{|c_1-x_1|}{|x_1 - z_1| + |y_1 - z_1|}\Big)^{\alpha_1}D_\theta(x - z,y- z) S(x-z,y-z),
\end{split}
\end{equation}
and
\begin{equation}
\begin{split}
|K&((x_1,c_2,c_3),y,z)-K(x,y,z)|\\
&\lesssim  \Big(\frac{|c_2 -x_2|}{|x_2 - z_2| + |y_2 - z_2|} + \frac{|c_3-x_3|}{|x_3 - z_3| + |y_3 - z_3|} \Big)^{\alpha_{23}} D_\theta(x - z, y - z)S(x - z, y- z),
\end{split}
\end{equation}
where $\alpha_1, \alpha_{23} \in (0,1]$.
Finally, we assume that  $K$ satisfies the H\"older estimate
\begin{equation}
\begin{split}
&|K(c,y,z)-K((c_1,x_2,x_3),y,z)-K((x_1,c_2,c_3),y,z)+K(x,y,z)| \\
&\lesssim  \Big(\frac{|c_1-x_1|}{|x_1-z_1|+|y_1-z_1|}\Big)^{\alpha_1}\Big(\frac{|c_2 -x_2|}{|x_2 - z_2| + |y_2 - z_2|} + \frac{|c_3-x_3|}{|x_3 - z_3| + |y_3 - z_3|} \Big)^{\alpha_{23}} \\
&\quad\times D_\theta(x - z, y - z)S(x - z, y- z).
\end{split}
\end{equation}
We also demand the symmetrical mixed size and H\"older estimates and H\"older estimates. 
For $j=1,2$, define the adjoint kernels $K^{*, j}$, $K^{*, j}_1$ and $K^{*, j}_{2,3}$ via the natural formulas, e.g., 
\begin{align*}
K^{*, 1}(x, y,z)=K(z,y,x), \quad K^{*, 2}_1(x,y,z)= K(x, (z_1,y_2,y_3), (y_1,z_2,z_3)).
\end{align*}
We assume that each adjoint kernel satisfies the same estimates as the kernel $K$.

\subsection{Partial kernel representations}\label{subsec:partial} 
Let $\wt \theta \in (0,1].$
For every interval $I^1$ we assume that there exists a kernel 
$$
K_{I^1} \colon (\R^2 \times \R^2 \times \R^2) \setminus \{(x_{2,3},y_{2,3},z_{2,3})\colon x_{i}= y_{i} = z_i \text{ for } i= 2 \text{ or } i=3\} \to \C,
$$
so that if $I_{j_1}^{2,3}$ and $I_{j_2}^{2,3}$ are disjoint for some $j_1,j_2\in \{1,2,3\},$  then
\begin{align*}
\langle T(1_{I^1} &\otimes 1_{I_1^{2,3}}, 1_{I^1} \otimes 1_{I_2^{2,3}}), 1_{I^1} \otimes 1_{I_3^{2,3}}\rangle
\\&= \iiint K_{I^1}(x_{2,3},y_{2,3},z_{2,3}) 1_{I_1^{2,3}}(x_{2,3})1_{I_2^{2,3}}(y_{2,3})1_{I_3^{2,3}}(z_{2,3}) \ud x_{2,3} \ud y_{2,3} \ud z_{2,3}.
\end{align*}

We demand the following estimates for the kernel $K_{I^1}\colon$ The size estimate
\begin{align*}
&|K_{I^1}(x_{2,3},y_{2,3},z_{2,3}) | \\
&\lesssim\Big(\frac{ |I^1|(|x_2 - z_2| + |y_2 - z_2|)}{|x_3 - z_3| + |y_3 - z_3|}+\frac{|x_3 - z_3| + |y_3 - z_3|}{ |I^1|(|x_2 - z_2| + |y_2 - z_2|)}\Big)^{ -\wt \theta}   \frac{ |I^1| }{\prod_{i=2}^3 \big(|x_i - z_i| + |y_i - z_i|\big)^2}
 \end{align*}
and the continuity estimate
\begin{equation*}
\begin{split}
&|K_{I^1}(c_{2,3},y_{2,3},z_{2,3})-K_{I^1}(x_{2,3},y_{2,3},z_{2,3})|  \\
&\lesssim 
\Big(\frac{|c_2 -x_2|}{|x_2 - z_2| + |y_2 - z_2|}+ \frac{|c_3-x_3|}{|x_3 - z_3| + |y_3 - z_3|} \Big)^{\alpha_{23}}\\
&\times \Big(\frac{ |I^1|(|x_2 - z_2| + |y_2 - z_2|)}{|x_3 - z_3| + |y_3 - z_3|}+\frac{|x_3 - z_3| + |y_3 - z_3|}{ |I^1|(|x_2 - z_2| + |y_2 - z_2|)}\Big)^{-\wt \theta} \frac{|I^1|}{\prod_{i=2}^3\big(|x_i - z_i| + |y_i - z_i|\big)^2}
\end{split}
\end{equation*}
whenever $c_{2,3}=(c_2,c_3)$ is such that $|c_i- x_i| \le\max (|x_i - z_i|, |y_i-z_i|)/2$ for $i=2,3$. We also assume the symmetrical continuity estimates.

We assume similar one-parameter conditions for the other partial kernel representation.
That is, for every rectangle $I^{2,3}$, there exists a standard bilinear Calder\'on-Zygmund kernel $K_{I^{2,3}}$ so that
if $I^1_{j_1}$ and $I^1_{j_2}$ are disjoint for some $j_1,j_2 \in \{1,2,3\},$ then 
\begin{align*}
  \langle T(1_{I_1^1} &\otimes 1_{I^{2,3}}, 1_{I_2^1} \otimes 1_{I^{2,3}}), 1_{I_3^1} \otimes 1_{I^{2,3}}\rangle
  \\&= \iiint K_{I^{2,3}}(x_1,y_1,z_1) 1_{I_1^{1}}(x_1)1_{I_2^{1}}(y_1)1_{I_3^{1}}(z_1) \ud x_1 \ud y_1 \ud z_1.
\end{align*}
The kernel $K_{I^{2,3}}$ satisfies the standard estimates
$$
|K_{I^{2,3}}(x_1,y_1,z_1)| \leq C_{K_{I^{2,3}}} \frac{1}{(|x_1-z_1| + |y_1 - z_1|)^2},
$$
\[
|K_{I^{2,3}}(x_1,y_1,z_1)- K_{I^{2,3}}(c_{1},y_1,z_1)| \leq C_{K_{I^{2,3}}} \frac{|x_1-c_1|^{\alpha_1}}{(|x_1-z_1| + |y_1 - z_1|)^{2 + \alpha_1}}
\]
whenever $|x_1-c_1| \le \max(|x_1-z_1|,|y_1 - z_1|)/2,$
and the symmetric continuity estimates. 
The smallest possible constant $C_{K_{I^{2,3}}}$ in these inequalities is denoted by $\|K_{I^{2,3}}\|_{\text{CZ}_{\alpha_1}}.$ We then assume that
$$
\|K_{I^{2,3}}\|_{\text{CZ}_{\alpha_1}} \lesssim |I^{2,3}|.
$$ 

\subsection{Cancellation assumptions: paraproduct free operators} 
We say that $T$ is a paraproduct free operator, if for all cancellative Haar functions $h_{I^1}$ and $h_{I^{2,3}}$ we have
\begin{align*}
 & \ave{T(1\otimes 1_{J_1^{2,3}},1\otimes 1_{J_2^{2,3}}),h_{I^1} \otimes 1_{J_3^{2,3}}} = \ave{T_1^{*,j}(1\otimes 1_{J_1^{2,3}},1\otimes 1_{J_2^{2,3}}),h_{I^1} \otimes 1_{J_3^{2,3}}} \\
 &\quad =\ave{T(1_{I_1^1}\otimes 1, 1_{I_2^1}\otimes 1), 1_{I_3^1}\otimes h_{I^{2,3}}}=\ave{T_{2,3}^{*, j}(1_{I_1^1}\otimes 1, 1_{I_2^1}\otimes 1), 1_{I_3^1}\otimes h_{I^{2,3}}}=0
\end{align*}
for all the adjoints $j \in \{1,2\}$. 
We always assume that all bilinear Zygmund operators in this article satisfy this cancellation condition.
The intention of this condition is to guarantee that our operator is representable using cancellative shifts
only.

\subsection{Weak boundedness property}\label{sub:weak}
We say that $T$ satisfies the weak boundedness property if
$$
|\ave{T(1_I,1_I),1_I}| \lesssim |I|
$$
for all Zygmund rectangles $I = I^1\times I^2 \times I^3.$ 

\begin{defn}\label{defn:czz}
\begin{sloppypar}
We say that a bilinear operator $T$ is a paraproduct free Calder\'on-Zygmund operator adapted to Zygmund dilations (CZZ operator)
if $T$ has the full kernel representation, 
the partial kernel representations, is paraproduct free and satisfies the weak boundedness property.
\end{sloppypar}
\end{defn}


\section{Estimates for the shift coefficients}\label{sec:coefficientesti}
We now move to consider the shift coefficients that appeared in the decomposition of $T$ in Section \ref{sec:ZygDecoRefined}.
When $T$ is a CZZ operator, we can estimate them.
Without loss of generality, we estimate
$$
\ave{T(h_{I\dot+n_1}^0, h_{I\dot+n_2}^{0}), h_{I,Z}}
$$
for $I \in \calD_Z$ and different values of $n_1, n_2 \in \Z^3$, and without loss of generality we assume $\theta=\tilde \theta<1$.
The coefficients related to the other terms of the decomposition (other than the main term \eqref{eq:key}) may have a different set of Haar functions,
but they are treated similarly.

We show that 
\begin{equation}\label{eq:coefficientesti}
|\ave{T(h_{I\dot+n_1}^0, h_{I\dot+n_2}^{0}), h_{I,Z}}| \lesssim (|k|+1)^2\varphi(k) \frac{|I|^\frac{3}{2}}{|K|^2},
\end{equation}
where 
 $$
 \varphi(k) := 2^{-k^1\alpha_1 - k^2\min\{\alpha_{23},\theta\}- \max\{k^3-k^1-k^2,0\}\theta}.
 $$
For terms of this particular form, we would not actually need to analyze some of the diagonal cases
(see Section \ref{sec:ZygDecoRefined}). However, these diagonal terms would appear in some other forms, so it makes sense to deal
with them here (even though in the real situation the Haar functions might be permuted differently, this does not matter,
and the calculations we present apply). It is very helpful to study the linear case \cite{HLMV-ZYG}, since the kernel estimates
are relatively involved and we will not repeat every detail when they are similar. 

Let $m^i := \max_{j = 1,2}  |n_j^i|.$
The analysis of the coefficients splits into combinations of  \[\begin{cases}m^{1}  \in (2^{k^1-3},2^{k^1-2}],\qquad  k^1=3,4,\ldots, &\text{(Separated)}\\
m^{1}  = 1, &\text{(Adjacent)}\\
m^{1}  = 0, &\text{(Identical)}\end{cases}\]
and
\[
\begin{cases}  m^i \in (2^{k^i-3},2^{k^i-2}],\qquad i=2,3,\, k^i=3,4,\ldots, &\text{(Separated)}\\
 m^2 < 2 \text{ and }  m^3 \in (2^{k^3-3},2^{k^3-2}],\qquad  k^3=3,4,\ldots, &\text{(Separated)}\\
 m^2 \in (2^{k^2-3},2^{k^2-2}] \text{ and }  m^3<2 \qquad  k^2=3,4,\ldots,&\text{(Separated)} \\
 m^2 = 1 \text{ and } m^3\le 1 &\text{(Adjacent)}\\
  m^2 = 0 \text{ and } m^3 =1 &\text{(Adjacent)}\\
  m^2 = 0 = m^3. &\text{(Identical)}\end{cases}
\]

It is enough to consider $m^i = n_1^i$ since the case $m^i =  n_2^i$ is symmetrical.
We will not go through explicitly every combination -- rather, we choose some illustrative examples.

\subsubsection*{Separated/Separated}
We begin with the case $|n_1^{i}|\ge 2$ for all $i =1,2,3.$ 
Hence, \[|x_i - z_i| \ge |n_1^i| \ell(I^i)  \ge 2^{k^i-3} \ell(I^i)$$ and $$|x_i - z_i| \le |n_1^i| \ell(I^i)+ 2\ell(I^i) \le 2^{k^i-1}\ell(I^i)\] for $i = 1,2,3.$ Moreover,  $|x_i - z_i| \ge |y_i -z_i| /2\ge 0$ for $i = 1,2,3.$ 
Thus, we have the estimate
\begin{align*} 
\Big(&\frac{\prod_{i=1}^2 (|x_i - z_i| + |y_i - z_i|)}{(\abs{x_3 - z_3} + \abs{y_3 - z_3})}+\frac{|x_3 - z_3| + |y_3 - z_3|}{\prod_{i=1}^2(|x_i - z_i| + |y_i - z_i|)}\Big)^{-\theta} \\
&\sim \Big(\frac{\prod_{i=1}^2 |x_i - z_i|}{\abs{x_3 - z_3}}+\frac{\abs{x_3 - z_3}}{ \prod_{i=1}^2 |x_i - z_i|}\Big)^{-\theta}
\\
&\sim \Big(\frac{\prod_{i=1}^2 2^{k^i} \ell(I^i)}{2^{k^3} \ell(I^3)}+\frac{2^{k^3} \ell(I^3)}{ \prod_{i=1}^2  2^{k^i} \ell(I^i)}\Big)^{-\theta} = (2^{k^1+k^2-k^3}+ 2^{k^3-k^1-k^2} )^{-\theta}.
\end{align*}

Let \(c_{I^i}\) denote the center of the interval \(I^i\). Furthermore, notation like 
\(c_I\) then refers to the corresponding  tuple \((c_{I^{1}}, c_{I^2}, c_{I^3}).\)
Using the cancellation of the Haar function we then have
\begin{align*}
&\Big|\iiint K(x,y,z) h_{I\dot+n_1}^0(x) h_{I\dot+n_2}^{0}(y) h_{I,Z}(z) \ud x \ud y \ud z\Big|\\
&= \Big|\iiint \Big(K(x,y,z) - K(x,y,(c_{I^1}, z_{2,3})) -K(x,y,(z_1, c_{I^{2,3}}))  + K(x,y,c_{I})\Big)\\ 
&\qquad \times h_{I\dot+n_1}^0(x) h_{I\dot+n_2}^{0}(y) h_{I,Z}(z) \ud x \ud y \ud z\Big|\\
&\lesssim \iiint 2^{-k^1 \alpha_1} (2^{-k^2}+2^{-k^3})^{ \alpha_{23}}  \frac{(2^{k^1+k^2-k^3}+ 2^{k^3-k^1-k^2} )^{-\theta}}{|K|^2} h_{I\dot+n_1}^0(x) h_{I\dot+n_2}^{0}(y) h_{I}^0(z) \ud x\ud y \ud z\\
&= 2^{-k^1 \alpha_1} (2^{-k^2}+2^{-k^3})^{ \alpha_{23}}  (2^{k^1+k^2-k^3}+ 2^{k^3-k^1-k^2} )^{-\theta} \frac{|I|^{\frac 32}}{|K|^2} \le \varphi(k)\frac{|I|^{\frac{3}{2}}}{|K|^2}.
\end{align*}

Let us then consider the case, where we have separation in the parameter $3$ but not in the parameter $2$ -- that is,
$|n^2_1| < 2\le |n^3_1|$. Then
\begin{align} \label{eq:splitting}
&\Big(\frac{\prod_{i=1}^2(|x_i - z_i| + |y_i - z_i|)}{|x_3 - z_3| + |y_3 - z_3|}+\frac{|x_3 - z_3| + |y_3 - z_3|}{\prod_{i=1}^2(|x_i - z_i| + |y_i - z_i|)}\Big)^{-\theta} \\
&\qquad\sim  \Big(\frac{|x_2 - z_2| + |y_2 - z_2|}{2^{k^3-k^1}|I^2|}+  \frac{2^{k^3-k^1}|I^2|}{|x_2 - z_2| + |y_2 - z_2|} \Big)^{-\theta}\nonumber\\
&\qquad\lesssim \Big(\frac{|x_2 - z_2|}{2^{k^3-k^1}|I^2|}+  \frac{2^{k^3-k^1}|I^2|}{|x_2 - z_2|} \Big)^{-\theta}+\Big(\frac{|y_2 - z_2|}{2^{k^3-k^1}|I^2|}+  \frac{2^{k^3-k^1}|I^2|}{|y_2 - z_2|} \Big)^{-\theta}, \nonumber
\end{align}
and so using the mixed estimates
\begin{align*}
&\Big|\iiint K(x,y,z) h_{I\dot+n_1}^0(x) h_{I\dot+n_2}^{0}(y) h_{I,Z}(z) \ud x \ud y \ud z\Big|\\
&= \Big|\iiint \Big(K(x,y,z) - K(x,y,(c_{I^1}, z_{2,3}))\Big) h_{I\dot+n_1}^0(x) h_{I\dot+n_2}^{0}(y) h_{I,Z}(z) \ud x \ud y \ud z\Big|\\
&\lesssim \iiint 2^{-k^1 \alpha_1}|K^1|^{-2}|K^3|^{-2} \frac{\Big(\frac{|x_2 - z_2| + |y_2 - z_2|}{2^{k^3-k^1}|I^2|}+  \frac{2^{k^3-k^1}|I^2|}{|x_2 - z_2| + |y_2 - z_2|} \Big)^{-\theta}}{\big(|x_2 - z_2| + |y_2 - z_2|\big)^2} \\&\hspace{7cm}\times h_{I\dot+n_1}^0(x) h_{I\dot+n_2}^{0}(y) h_{I}^0(z) \ud x \ud y \ud z\\
&= 2^{-k^1 \alpha_1}\frac{|I^1|^{\frac{3}{2}}|I^3|^\frac{3}{2}}{|K^1|^2|K^3|^{2}} \iiint  \frac{\Big(\frac{|x_2 - z_2| + |y_2 - z_2|}{2^{k^3-k^1}|I^2|}+  \frac{2^{k^3-k^1}|I^2|}{|x_2 - z_2| + |y_2 - z_2|} \Big)^{-\theta}}{\big(|x_2 - z_2| + |y_2 - z_2|\big)^2}\\
&\hspace{14em}\times h_{I^2\dot+n^2_1}^0(x_2) h_{I^2\dot+n^2_2}^0(y_2)  h_{I^2}^0(z_2) \ud x_2\ud y_2 \ud z_2\\
&\lesssim \varphi(k)\frac{|I|^{\frac{3}{2}}}{|K|^2}.
\end{align*}
We note that the last inequality requires a case study (see also~\cite{HLMV-ZYG}*{Lemma 8.5}) 
and we used the standard estimate
\begin{align}\label{eq:biliKer}
\int_{\R^d} \frac{\ud u}{(r + |u_0 - u|)^{d + \alpha}} \lesssim r^{-\alpha}.
\end{align} 
Symmetrical estimates hold if $|n^2_1| \ge 2 > |n^3_1|$.

\subsubsection*{Adjacent/Separated}
We look at the example case $|n_1^2| \geq 2 > |n_1^3|$ and $|n_1^1| = 1.$ By the size estimate we have
\begin{align*}
  &|\ave{T(h_{I\dot+n_1}^0, h_{I\dot+n_2}^{0}), h_{I,Z}}| \\
  &\lesssim \frac{|I^2|^{3/2}}{|I^{1,3}|^{3/2}|K^2|^2} \iiint \frac{\Bigl(\frac{(|x_1 - z_1| + |y_1 - z_1|)2^{k^2}\ell(I^2)}{|x_3 - z_3| + |y_3 - z_3|} + \frac{|x_3 - z_3| + |y_3 - z_3|}{(|x_1 - z_1| + |y_1 - z_1|)2^{k^2}\ell(I^2)} \Bigr)^{-\theta}}{\big(|x_1 - z_1| + |y_1 - z_1|\big)^2 \big(|x_3 - z_3| + |y_3 - z_3|\big)^2}\\
  &\hspace{10em}\times 1_{I^{1,3} \dot+ n_1^{1,3}}(x_{1,3}) 1_{I^{1,3} \dot+ n_2^{1,3}}(y_{1,3})1_{I^{1,3}}(z_{1,3})\ud x_{1,3}\ud y_{1,3}\ud z_{1,3}.
\end{align*} 
Similarly as \eqref{eq:splitting}, we can split the integral into two terms. Then  by  \eqref{eq:biliKer} we reduce the problem to estimating 
\begin{align*}
&\iiint \frac{\Bigl(\frac{(|x_1 - z_1| + |y_1 - z_1|)2^{k^2}\ell(I^2)}{ |x_3 - z_3|} + \frac{ |x_3 - z_3|}{(|x_1 - z_1| + |y_1 - z_1|)2^{k^2}\ell(I^2)} \Bigr)^{-\theta}}{\big(|x_1 - z_1| + |y_1 - z_1|\big)^2   |x_3 - z_3|}\\
  &\hspace{10em}\times 1_{I^{1,3} \dot+ n_1^{1,3}}(x_{1,3}) 1_{I^{1} \dot+ n_2^{1}}(y_1)1_{I^{1,3}}(z_{1,3})\ud x_{1,3}\ud y_1\ud z_{1,3}\\
  &+ \iiint \frac{\Bigl(\frac{(|x_1 - z_1| + |y_1 - z_1|)2^{k^2}\ell(I^2)}{ |y_3 - z_3|} + \frac{ |y_3 - z_3|}{\big(|x_1 - z_1| + |y_1 - z_1|\big)2^{k^2}\ell(I^2)} \Bigr)^{-\theta}}{(|x_1 - z_1| + |y_1 - z_1|)^2   |y_3 - z_3|}\\
  &\hspace{10em}\times 1_{I^{1,3} \dot+ n_1^{1,3}}(x_{1,3}) 1_{I^{1,3} \dot+ n_2^{1,3}}(y_{1,3})1_{I^{1}}(z_1)\ud x_{1,3}\ud y_{1,3}\ud z_1.
\end{align*}
Since they are similar, we only bound the first one. Note that 
\begin{align*}
&\Bigl(\frac{(|x_1 - z_1| + |y_1 - z_1|)2^{k^2}\ell(I^2)}{ |x_3 - z_3|} + \frac{ |x_3 - z_3|}{(|x_1 - z_1| + |y_1 - z_1|)2^{k^2}\ell(I^2)} \Bigr)^{-\theta} \\ &\qquad \qquad \times(|x_1 - z_1| + |y_1 - z_1|)^{-2}\\
& \le \Bigl(\frac{(|x_1 - z_1| + |y_1 - z_1|)2^{k^2}\ell(I^2)}{ |x_3 - z_3|}  \Bigr)^{-\theta}(|x_1 - z_1| + |y_1 - z_1|)^{-2}\chi_{\{|x_1-z_1|2^{k^2}\ell(I^2)\ge |x_3-z_3|\}}\\
& + \Bigl(\frac{ |x_3 - z_3|} {(|x_1 - z_1| + |y_1 - z_1|)2^{k^2}\ell(I^2)} \Bigr)^{-\theta}(|x_1 - z_1| + |y_1 - z_1|)^{-2}\chi_{\{|x_1-z_1|2^{k^2}\ell(I^2)< |x_3-z_3|\}}.
\end{align*}
Then apply \eqref{eq:biliKer} to the integral over $y_1$, then by following the linear case \cite{HLMV-ZYG}*{Lemma 8.11} we get that the above integral is bounded by $|I^{1,3}|   k^2 2^{-k^2\theta}.$ Thus, we get  
\begin{align*}
  &|\ave{T(h_{I\dot+n_1}^0, h_{I\dot+n_2}^{0}), h_{I,Z}}| \lesssim  \frac{|I^2|^{3/2}}{|I^{1,3}|^{1/2}|K^2|^2} k^2 2^{-k^2\theta}\lesssim k^2 \varphi(k) \frac{|I|^\frac{3}{2}}{|K|^2} .
\end{align*}

\subsubsection*{Adjacent/Adjacent}
We again have no major changes to the linear case but
in order to use the estimate 
\begin{equation}\label{eq:maximalOrigin}
\int_\R \frac{\Big(\frac{t}{|u|}+\frac{|u|}{t}\Big)^{-\theta}}{t|u|}|f(u)| \ud u \lesssim t^{-1} Mf(0) 
\end{equation}
we need to first use \eqref{eq:biliKer} repeatedly. For example, consider $|n_1^1| = 1$ and $|n^2_1| = 1, |n_1^3| \le 1.$
By the size estimate of the kernel, we need to control 
\[ \frac{\Big(\frac{\prod_{i=1}^2(|x_i - z_i| + |y_i - z_i|)}{|x_3 - z_3| + |y_3 - z_3|}+\frac{|x_3 - z_3| + |y_3 - z_3|}{\prod_{i=1}^2(|x_i - z_i| + |y_i - z_i|)}\Big)^{-\theta}}{\prod_{i=1}^3\big(|x_i - z_i| + |y_i - z_i|\big)^2} h_{I\dot+n_1}^0(x) h_{I\dot+n_2}^{0}(y) h_{I}^0(z). 
\] As before, we split this into two terms, one of them is 
\[
\frac{\Big(\frac{\prod_{i=1}^2(|x_i - z_i| + |y_i - z_i|)}{ \abs{x_3 - z_3}}+\frac{ \abs{x_3 - z_3}}{\prod_{i=1}^2(|x_i - z_i| + |y_i - z_i|)}\Big)^{-\theta}}{\prod_{i=1}^3\big(|x_i - z_i| + |y_i - z_i|\big)^2} h_{I\dot+n_1}^0(x) h_{I\dot+n_2}^{0}(y) h_{I}^0(z). 
\]
We then apply \eqref{eq:biliKer} to the integral over $y_3$, and then use the previous trick repeatedly. That is, we write 
\begin{align*}
&\Big(\frac{\prod_{i=1}^2(|x_i - z_i| + |y_i - z_i|)}{ \abs{x_3 - z_3}}+\frac{ \abs{x_3 - z_3}}{\prod_{i=1}^2(|x_i - z_i| + |y_i - z_i|)}\Big)^{-\theta}\\
&\qquad \le \Big(\frac{\prod_{i=1}^2(|x_i - z_i| + |y_i - z_i|)}{ \abs{x_3 - z_3}}\Big)^{-\theta}\chi_{\{|x_1-z_1| (|x_2 - z_2| + |y_2 - z_2|) \ge |x_3-z_3|\}}\\
&\qquad \quad+ \Big(\frac{ \abs{x_3 - z_3}}{\prod_{i=1}^2(|x_i - z_i| + |y_i - z_i|)}\Big)^{-\theta}\chi_{\{|x_1-z_1|(|x_2 - z_2| + |y_2 - z_2|) < |x_3-z_3|\}}
\end{align*}
and apply~\eqref{eq:biliKer} to the integral over $y_1$. Then, after a similar argument on $y_2$, we finally arrive at 
\begin{align*}
& \frac 1{|I|^{\frac 12}}\iint \frac{\Big(\frac{\prod_{i=1}^2 |x_i - z_i|}{\abs{x_3 - z_3}}+\frac{\abs{x_3 - z_3}}{\prod_{i=1}^2 |x_i - z_i|}\Big)^{-\theta}}{\prod_{i=1}^3 |x_i - z_i|} h_{I\dot+n_1}^0(x)  h_{I}^0(z) \ud x \ud z \\
&\lesssim \frac{1}{|I|^{\frac{1}{2}}} \lesssim \frac{|I|^\frac{3}{2}}{|K|^2}.
\end{align*}

\subsubsection*{Adjacent/Identical}
We consider the case $|n_1^1| = 1$ and $n_j^2 = n_j^3 = 0, j= 1,2.$
We write 
\begin{align*}
  \sum_{Q_1^{2,3},Q_2^{2,3}, Q_3^{2,3} \in \ch(I^{2,3})} \ave{T(h_{I\dot+n_1}^0 1_{Q_1^{2,3}}, h_{I\dot+n_2}^{0} 1_{Q_2^{2,3}}), h_{I,Z} 1_{Q_3^{2,3}}}.
\end{align*}
It is enough to consider $Q_1^{2,3} = Q_2^{2,3} = Q_3^{2,3}$ since otherwise we have adjacent intervals, and we are back in the Adjacent/Adjacent case. 
Hence, the partial kernel representation \ref{subsec:partial} yields that
\begin{align*}
  &\Big|\pm |I^{2,3}|^{-\frac{3}{2}}\iiint K_{Q_1^{2,3}} h_{I^1\dot+n_1^1}^0  h_{I^1\dot+n_2^1}^{0} h_{I^1} \Big| \\
  &\lesssim \frac{|I^{2,3}|^{\frac 32}}{|K^{2,3}|^2} \iiint \frac{1}{(|x_1 - z_1| + |y_1-z_1|)^2} h_{I^1\dot+n_1^1}^0(x_1)  h_{I^1\dot+n_2^1}^{0}(y_1) h_{I^1}(z_1) \ud x_1\ud y_1 \ud z_1. 
\end{align*}
Then, first using \eqref{eq:biliKer} and then standard integration methods we get the following inequality
\begin{align*}
  &\iiint \frac{1}{(|x_1 - z_1| + |y_1-z_1|)^2} h_{I^1\dot+n_1^1}^0(x_1)  h_{I^1\dot+n_2^1}^{0}(y_1) h_{I^1}(z_1) \ud x_1 \ud y_1 \ud z_1 \\
  &\lesssim \frac{1}{|I^1|^{\frac 12} } \iint \frac{1}{|x_1 - z_1|} h_{I^1\dot+n_1^1}^0(x_1) h_{I^1}(z_1) \ud x_1 \ud z_1 \\
  &\lesssim \frac{1}{|I^1|^{\frac 12}} \sim \frac{|I^1|^{\frac 32} }{|K^1|^{2}}  
\end{align*}
as desired.
\subsubsection*{Identical/Identical}
Just like in above we split the pairing to
\begin{align*}
  \sum_{Q_1^{1},Q_2^{1}, Q_3^{1} \in \ch(I^{1})} &\sum_{Q_1^{2,3},Q_2^{2,3}, Q_3^{2,3} \in \ch(I^{2,3})} \\
  &\ave{T(h_{I\dot+n_1}^0 (1_{Q_1^1} \otimes 1_{Q_1^{2,3}}), h_{I\dot+n_2}^{0}(1_{Q_2^1} \otimes 1_{Q_2^{2,3}})), h_{I,Z} (1_{Q_3^1} \otimes1_{Q_3^{2,3}})}.
  \end{align*}
  The cases when $Q_i^1 \neq Q_j^1$ for some $i,j=1,2,3, i\neq j$ are essentially included in the cases of the two previous subsections. Hence, we consider $Q_1^1 = Q_2^1 = Q_3^1.$ Then there are two cases left, that is, either $Q_i^{2,3} \neq Q_j^{2,3}$ for some $i,j=1,2,3, i\neq j$, or $Q_1^{2,3} = Q_2^{2,3} = Q_3^{2,3}.$ Beginning from the latter one, similarly as in \cite{HLMV-ZYG}, by splitting $Q_1^3$ into sub-intervals we get
  \[
    |\ave{T(1_{Q_1^1} \otimes 1_{Q_1^{2,3}}, 1_{Q_1^1} \otimes 1_{Q_1^{2,3}}), 1_{Q_1^1} \otimes 1_{Q_1^{2,3}}}| \lesssim |Q_1^1||Q_1^{2,3}|
  \]
  by the weak boundedness property \ref{sub:weak} and the Identical/Adjacent case. Hence, we get the desired bound 
  \begin{align*}
    |\ave{T(h_{I\dot+n_1}^0 (1_{Q_1^1} \otimes 1_{Q_1^{2,3}}), h_{I\dot+n_2}^{0}(1_{Q_1^1} \otimes 1_{Q_1^{2,3}})), h_{I,Z} (1_{Q_1^1} \otimes1_{Q_1^{2,3}})}| \lesssim \frac{|Q_1|}{|I|^{\frac{3}{2}}} \leq \frac{|I|^{\frac{3}{2}}}{|K|^2}.
  \end{align*}

  We handle the remaining case $Q_i^{2,3} \neq Q_j^{2,3}$ for some $i,j=1,2,3, i\neq j.$ By the partial kernel representation and its size estimate we get
  \begin{align*}
    &\Big| \pm |I|^{-\frac{3}{2}}\iiint K_{Q_1^{1}} 1_{Q_1^{2,3}}1_{Q_2^{2,3}} 1_{Q_3^{2,3}} \Big|\\
  &\lesssim \frac{1}{|I^{1}|^{\frac{1}{2}}}\frac 1{|I^{2,3}|^{\frac 32}} \iiint\Big(\frac{ |I^1|(|x_2 - z_2| + |y_2 - z_2|)}{|x_3 - z_3| + |y_3 - z_3|}+\frac{|x_3 - z_3| + |y_3 - z_3|}{|I^1| (|x_2 - z_2| + |y_2 - z_2|)}\Big)^{-\theta}\\
  &\hspace{4em} \times \prod_{i=2}^3\frac{1}{\big(|x_i - z_i| + |y_i - z_i|\big)^2}1_{Q_1^{2,3}}1_{Q_2^{2,3}} 1_{Q_3^{2,3}} \ud x_{2,3} \ud y_{2,3} \ud z_{2,3}.
\end{align*}
Then using similar arguments as in the Adjacent/Adjacent case and \eqref{eq:maximalOrigin} gives us the desired bound.

\section{Structural decomposition of Zygmund shifts}\label{sec:structural}
In this section we decompose the bilinear Zygmund shifts (see Section \ref{subsec:shifts}) as a sum of operators with simpler
cancellation properties. The decomposition is not optimal (in the sense that weighted estimates with Zygmund weights 
cannot be obtained via this) -- however, it is sufficient for unweighted boundedness in the full range that we later obtain via tri-parameter theory.
Recall that $k = (k^1, k^2,k^3)$ is the complexity of the bilinear Zygmund shift. 

\begin{defn}\label{def:zygshift}
Bilinear operators of the form 
  \begin{equation}\label{eq:zygshift}
    S_{(l_1,l_2,l_3)}(f_1, f_2)= \sum_{L \in \calD_\lambda} \sum_{I_j^{(\ell_j)}=L} a_{L, (I_j)}\langle f_1, h_{I_1^1}\otimes h_{I_1^{2,3}}^0\rangle  \langle f_2, h_{I_2^1}^0\otimes h_{I_2^{2,3}}\rangle  h_{I_3},
     \end{equation}
     where $\lambda = 2^n, n\in \Z, |n|\le 3\max(k^i)$ and  \[|a_{L,(I_j)}|\le \frac{|I_j|^{\frac 32}}{|L|^{2}},\] are tri-parameter bilinear shifts of Zygmund nature if at least one rectangle $I_{i_1}^1 \times I_{i_2}^{2,3}, i_1 = 1,3, i_2 = 2,3$ is a Zygmund rectangle and 
     \begin{enumerate}
     \item $\ell_j^i\le k^i$ for all $i, j=1,2,3$;
     \item $(\ell_j^3- \ell_j^2)_+ \le (k^3-k^2)_+$ for all $j=1,2,3$.
     \end{enumerate}
     Moreover, any adjoint $$
     S_{(l_1,l_2,l_3)}^{j_1^*,j_{2,3}^*},\qquad j_1,j_{2,3}\in \{0,1,2\},
     $$
     is also considered to be a tri-parameter bilinear shift of Zygmund nature. Here, the adjoint $j_{2,3}^{*}$ means that, for example, in case $j_{2,3} = 1$ functions $h_{I_1^{2,3}}^0$ and $h_{I_3^{2,3}}$ switch places. 
  \end{defn}
Note that these operators share a `weaker' Zygmund structure. Ideally, we would want to have $I_3 \in \calD_Z$ and $I_1^1 \times I_2^{2,3} \in \calD_Z.$ 

\begin{prop}\label{prop:sumofshifts}
Let $Q_k, k= (k^1,k^2,k^3),$ be a bilinear Zygmund shift operator as defined in Section \ref{subsec:shifts}. Then
\[Q_k = C \sum_{u=1}^c \sum_{l^1 = 0}^{k^1 -1} \sum_{l^{2,3} = 0}^{k^{2,3} - 1}    S^{u},\]
where $S^{u}$ is a bilinear operator as in Definition \ref{def:zygshift} with complexity depending on $l$ and $k$,
and 
\[
\sum_{l^{2,3} = 0}^{k^{2,3} - 1}:= \begin{cases}
    \sum\limits_{0\le l^2=l^3 \le k^2-1}   + \sum\limits_{\substack{l^2=k^2\\ k^2 \le l^3\le k^3-1}}, &\text{if } k^3 \ge k^2 \\
    \sum\limits_{0\le l^2=l^3 \le k^3-1}   + \sum\limits_{\substack{k^3\le l^2\le k^2-1\\ l^3=k^3}}, &\text{if } k^3 < k^2.
  \end{cases}
\]
\end{prop}
\begin{proof}
  The argument is similar in spirit to the purely bi-parameter decomposition in \cite{AMV}.
  For notational convenience, we consider a shift $Q_k$ of the particular form
  \begin{align*}
    &\ave{Q_{k}(f_1,f_2),f_3} \\
  &= \sum_{K \in \calD_{2^{-k^1-k^2+k^3}}}\sum_{\substack{I_1,I_2,I_3 \in \calD_{Z}\\ I_j^{(k)} = K}} a_{K,(I_j)} \Big[A^{3,3}_{I_1,I_2,I_3} - A^{3,3}_{I_{3}^1 \times I_1^{2,3},I_{3}^{1} \times I_2^{2,3}, I_{3}}\\
  &\hspace{10em} -A^{3,3}_{I_{1}^1 \times I_{3}^{2,3},I_{2}^{1} \times I_{3}^{2,3}, I_{3}} +  A^{3,3}_{I_{3},I_{3}, I_{3} } \Big] \\
  &=  \sum_{K \in \calD_{2^{-k^1-k^2+k^3}}}\sum_{\substack{I_1,I_2,I_3 \in \calD_{Z}\\ I_j^{(k)} = K}} a_{K,(I_j)}\ave{f_3,  h_{I_3}} \Big[ \langle f_1, h_{I_1}^0 \rangle \langle f_2, h_{I_2}^0 \rangle - \langle f_1, h_{I^1_3}^0 h_{I^{2,3}_1}^0 \rangle \langle f_2, h_{I^1_3}^0 h_{I^{2,3}_2}^0 \rangle \\
  &\hspace{5em}- \langle f_1, h_{I^1_1}^0 h_{I^{2,3}_3}^0 \rangle \langle f_2, h_{I^1_2}^0 h_{I^{2,3}_3}^0 \rangle 
  + \langle f_1, h_{I_3}^0  \rangle \langle f_2, h_{I_3}^0   \rangle \Big].
\end{align*}
There is no essential difference in the general case. Let us also use the usual abbreviation $\calD_{2^{-k^1-k^2+k^3}} = \calD_\lambda$.

We define
$$
b_{K, (I_j)}
=|I_1| a_{K, (I_j)}
$$
and
$$
B_{I_1, I_2, I_{3}}^{3, 3}
= \langle f_1 \rangle_{I_1} \ave{f_2}_{I_2} \langle f_{3}, h_{I_3} \rangle.
$$
We can write the shift $Q_k$ using these by replacing $a$ with $b$ and $A$ with $B$.

Recall the notation
\begin{align*}
  &\Delta_{K^1}^{l^1} f = \sum_{\substack{L^1 \in \calD^1\\ (L^1)^{(l^1)} = K^1}} \Delta_{L^1} f,\qquad
  P_{K^1}^{k^1} f= \sum_{l^1=0}^{k^1-1}\Delta_{K^1}^{l^1} f, \\
  &E_{K^1} f = \ave{f}_{K^1} 1_{K^1}, \qquad E_{K^1}^{k^1} f =  \sum_{\substack{L^1 \in \calD^1\\ (L^1)^{(k^1)} = K^1}}\ave{f}_{L^1} 1_{L^1}.
\end{align*}
Let us define
\begin{equation}
  P_{K^{2,3}}^{k^{2,3}} f:= \sum_{l^{2,3} = 0}^{k^{2,3}-1} \Delta_{K^{2,3}}^{(l^2,l^3)} f := \begin{cases}
    \sum\limits_{l^2 = 0}^{k^2-1} \Delta_{K^{2,3}}^{l^2} f + \sum\limits_{l^3 = k^2}^{k^3-1} E_{K^2}^{k^2}\Delta_{K^3}^{l^3} f, &\text{if } k^3 \ge k^2 \\
    \sum\limits_{l^3 = 0}^{k^3-1} \Delta_{K^{2,3}}^{l^3} f + \sum\limits_{l^2 = k^3}^{k^2-1} \Delta_{K^2}^{l^2}E_{K^3}^{k^3} f, &\text{if } k^3 < k^2,
  \end{cases}
\end{equation}
where we have the standard one-parameter definition \[\Delta_{K^{2,3}}^{l^i} f = \sum_{\substack{L^{2,3} \in \calD^{2,3}\\ (L^2)^{(l^i)} \times (L^3)^{(l^i)} = K^2 \times K^3 }} \Delta_{L^{2,3}}f.\]
We also use a similar shorthand for the expanded martingale blocks
\[\sum_{l^{2,3} = 0}^{k^{2,3}-1} \Delta_{K^{2,3}}^{(l^2,l^3)} f = \sum_{l^{2,3} = 0}^{k^{2,3}-1} \sum_{(L^{2,3})^{(l^{2,3})} = K^{2,3}} \ave{f,h_{L^{2,3}}} h_{L^{2,3}},\]
where we allow, for example, that $h_{L^{2,3}} = h_{L^2}^0 \otimes h_{L^3}$ when $k^3 > k^2$ and $l^2 = k^2.$ 

Using this notation we define the following. For a cube $I$ and integers $l,j_0 \in \{1,2, \dots \}$ we define
\begin{equation}\label{eq:Dnot}
  D_{I,l}(j,j_0)=
  \begin{cases}
    E_I, \quad &\text{if } j \in \{1, \dots, j_0-1\}, \\ 
    P_{I}^{l}, \quad &\text{if } j=j_0, \\
    \Id, \quad &\text{if } j \in \{j_0+1,j_0+2, \dots\}, 
  \end{cases}
\end{equation}
where $\Id$ denotes the identity operator, and if we have a rectangle $I^{2,3}$ and a tuple $l^{2,3}$ we use the modified $P_{I^{2,3}}^{l^{2,3}}.$ 

Let $I_1, I_2, I_3$ be as in the summation of $Q_k$. We use the above notation in parameter one $D_{I^1,l^1}(j,j_0)$ and for the other two parameters we use $D_{I^{2,3},l^{2,3}}(j,j_0)$.
Thus, expanding to the martingale blocks leads us to 
\begin{equation*}
  \begin{split}
    &B_{I_1, I_2, I_{3}}^{3, 3}\\
    &=\sum_{m_1,m_2=1}^{3}
    \prod_{j=1}^{2} \langle D^1_{K^1,k^1}(j,m_1)D^{2,3}_{K^{2,3},k^{2,3}}(j,m_2)f_j \rangle_{I_j} \langle f_{3}, h_{I_{3}} \rangle.
  \end{split}
\end{equation*}
Hence, we may write 
$$
\sum_{K \in \calD_\lambda}\sum_{\substack{I_1,I_2,I_3 \in \calD_{Z}\\ I_j^{(k)} = K}}B_{I_1, I_2, I_{3}}^{3, 3}
=:\sum_{m_1,m_2=1}^{3}\Sigma_{m_1,m_2}^1.
$$
Also, we have that 
\begin{equation*}
  B_{I_{3}^1 \times I_1^{2,3}, I_{3}^1 \times I_{2}^{2,3},  I_{3}^1 \times I_3^{2,3} }^{3,3}
  =\sum_{m_2=1}^{3}
  \prod_{j=1}^{2} \langle D^{2,3}_{K^{2,3},k^{2,3}}(j,m_2)f_j \rangle_{I_3^1\times I_j^{2,3}} \langle f_{3}, h_{I_{3}} \rangle
\end{equation*}
and
\begin{equation*}
  B_{I_{1}^1 \times I_3^{2,3}, I_{2}^1 \times I_{3}^{2,3},  I_{3}^1 \times I_3^{2,3} }^{3,3}
  =\sum_{m_1=1}^{3}
  \prod_{j=1}^{2} \langle D^1_{K^1,k_1}(j,m_1) f_j \rangle_{I_j^1\times I_3^{2,3}} \langle f_{3}, h_{I_{3}} \rangle,
\end{equation*}
which gives that
$$
\sum_{K \in \calD_\lambda}\sum_{\substack{I_1,I_2,I_3 \in \calD_{Z}\\ I_j^{(k)} = K}}
B_{I_{3}^1 \times I_1^{2,3}, I_{3}^1 \times I_{2}^{2,3},  I_{3}^1 \times I_3^{2,3} }^{3,3}
=: \sum_{m_2=1}^{3} \Sigma_{m_2}^2
$$
and 
$$
\sum_{K \in \calD_\lambda}\sum_{\substack{I_1,I_2,I_3 \in \calD_{Z}\\ I_j^{(k)} = K}}
B_{I_{1}^1 \times I_3^{2,3}, I_{2}^1 \times I_{3}^{2,3},  I_{3}^1 \times I_3^{2,3} }^{3,3}
=: \sum_{m_1=1}^{3} \Sigma^3_{m_1}.
$$
Finally, we just set
$$
\sum_{K \in \calD_\lambda}\sum_{\substack{I_1,I_2,I_3 \in \calD_{Z}\\ I_j^{(k)} = K}}
B_{I_{3}, I_{3},  I_{3}}^{3,3}=: \Sigma^4.
$$

Thus, we have the following decomposition
\begin{equation*}
  \begin{split}
    \langle Q_{k}(f_1,f_2), f_{3}\rangle
    &= \sum_{m_1,m_2=1}^2 \Sigma^1_{m_1,m_2}
    + \sum_{m_2=1}^2 (\Sigma^1_{3,m_2}-\Sigma^2_{m_2})\\
    &+\sum_{m_1=1}^2 (\Sigma^1_{m_1,3}-\Sigma^3_{m_1})
    + (\Sigma^1_{3,3}-\Sigma^2_{3}-\Sigma^3_{3}+\Sigma^4).
  \end{split}
\end{equation*}

First, we take one $\Sigma^1_{m_1,m_2}$ with $m_1,m_2 \in \{1,2\}$. For notational convenience,
we choose the case $m_1=m_2= 2$. Recall that
$$
\Sigma^1_{2,2}
=\sum_{K \in \calD_\lambda}\sum_{\substack{I_1,I_2,I_3 \in \calD_{Z}\\ I_j^{(k)} = K}}
b_{K,(I_j)} \langle f_1 \rangle_K \langle P_{K^1}^{k^1} P_{K^{2,3}}^{k^{2,3}} f_2 \rangle_{I_2} 
\langle f_{3}, h_{I_{3}} \rangle.
$$
We expand
$$
\langle P_{K^1}^{k^1} P_{K^{2,3}}^{k^{2,3}} f_2 \rangle_{I_2}
= \sum_{l^1=0}^{k^1-1} \sum_{l^{2,3}=0}^{k^{2,3}-1} \sum_{\substack{(L^1)^{(l^1)}=K^1\\ (L^{2,3})^{(l^{2,3})} = K^{2,3}}}
\langle f_2 , h_{L^1} \otimes h_{L^{2,3}} \rangle \langle h_{L^1} \otimes h_{L^{2,3}} \rangle_{I_2}
$$
and note that $L$ is not necessarily a Zygmund rectangle. 
It holds that
\begin{equation*}
  \begin{split}
    \Sigma^1_{2,2}
    = \sum_{l^1=0}^{k^1-1} \sum_{l^{2,3}=0}^{k^{2,3}-1} \sum_{K \in \calD_\lambda}
    \sum_{L^{(l^1,l^2,l^3)}=K}
    &\sum_{\substack{I_3 \in \calD_{Z}\\ I_3^{(k)} = K}}
    \Big(\sum_{\substack{I_1 \\ I_1^{(k)} = K }}\sum_{\substack{I_2 \subset L \\ I_2^{(k)}=K}} 
    \frac{b_{K,(I_j)} \langle h_{L} \rangle_{I_2} }
    {|K|^{\frac 12}}\Big) \\
    & \langle f_1, h^0_K \rangle 
    \langle f_2 , h_{L} \rangle
    \langle f_{3}, h_{I_{3}} \rangle.
  \end{split}
\end{equation*}
Now, since we can easily check that
$$
\Big|\sum_{\substack{I_1 \\ I_1^{(k)} = K}}\sum_{\substack{ I_2 \subset L \\ I_2^{(k)}=K}} 
\frac{b_{K,(I_j)} \langle h_{L} \rangle_{I_2} }
{|K|^{\frac 12}}\Big|
\le \frac{|K|^{\frac 12} |L|^{\frac 12} |I_{3}|^{\frac 12}}{|K|^2},
$$
we get a sum of operators we wanted 
$$
\Sigma^1_{2,2}
=\sum_{l^1=0}^{k^1-1} \sum_{l^{2,3}=0}^{k^{2,3}-1}
\langle S_{(0,(l^1,l^2,l^3), k)}(f_1, f_2),f_{3} \rangle,
$$
where $S_{(0,(l^1,l^2,l^3), k)}$ is a type of the shift \eqref{eq:zygshift}.
The general case $\Sigma^1_{m_1,m_2}$ is analogous.

We turn to the terms $\Sigma^1_{3,m_2}-\Sigma^2_{m_2}$.
Let us take, for example, the case $m_2 = 1$. After expanding $P^{k^{2,3}}_{K^{2,3}}$ in the first slot,
$\Sigma^1_{3,1}-\Sigma^2_{1}$ can be written as
\begin{equation*}
  \begin{split}
    \sum_{l^{2,3}=0}^{k^{2,3}-1} \sum_{K \in \calD_\lambda} \sum_{(L^{2,3})^{(l^{2,3})}= K^{2,3}}  
    \sum_{\substack{I_1, I_2,I_3 \\ I_j^{(k)} = K }} 
    b_{K,(I_j)} \langle h_{L^{2,3}} \rangle_{I^{2,3}_{1}}
    \Big[ 
      \Big\langle f_{1}, \frac{1_{K^1}}{|K_1|} \otimes h_{L^{2,3}} \Big\rangle \langle f_2 \rangle_{K^1 \times I^{2,3}_2}  \\ 
     - 
      \Big\langle f_{1}, \frac{1_{I^1_{3}}}{|I^1_{3}|} \otimes h_{L^{2,3}} \Big\rangle
      \langle f_2 \rangle_{I^1_{3} \times I^{2,3}_2} \Big] \langle f_{3}, h_{I_{3}} \rangle.
    \end{split}
  \end{equation*}
  For the moment, we fix one $l^{2,3}$ and write $g_1 = \ave{f_1, h_{L^2}}$ and $g_2 = \ave{f_2}_{I_2^{2,3}}$.
  We write inside the brackets
  \begin{equation*}
    \prod_{j=1}^2 \langle g_j \rangle_{K^1}-\prod_{j=1}^2 \langle g_j \rangle_{I^1_{3}}
    =-\sum_{l^1=0}^{k^1-1}\Big(\prod_{j=1}^2 \langle g_j \rangle_{(I^1_{3})^{(l^1)}}-\prod_{j=1}^2 \langle g_j \rangle_{(I^1_{3})^{(l^1 + 1)}}\Big)
  \end{equation*}
and then expand $\prod_{j=1}^2 \langle g_j \rangle_{(I^1_{3})^{(l^1)}}-\prod_{j=1}^2 \langle g_j \rangle_{(I^1_{3})^{(l^1+1)}}$
  as
  \begin{equation*}
    \langle \Delta_{(I^1_{3})^{(l^1+1)}} g_{1}\rangle_{I^1_{3}} \langle g_2 \rangle_{(I^1_{3})^{(l^1)}}  +  \langle g_1 \rangle_{(I^1_{3})^{(l^1+1)}} \langle \Delta_{(I^1_{3})^{(l^1+1)}} g_{2}\rangle_{I^1_{3}}.
  \end{equation*}
  We get
  \begin{align*}
    &\prod_{j=1}^2 \langle g_j \rangle_{K^1}-\prod_{j=1}^2 \langle g_j \rangle_{I^1_{3}} \\
    &=
    -\sum_{l^1=0}^{k^1-1} \Big(\langle \Delta_{(I^1_{3})^{(l^1+1)}} g_{1}\rangle_{I^1_{3}} \langle g_2 \rangle_{(I^1_{3})^{(l^1)}}  +  \langle g_1 \rangle_{(I^1_{3})^{(l^1+1)}} \langle \Delta_{(I^1_{3})^{(l^1+1)}} g_{2}\rangle_{I^1_{3}}\Big),
  \end{align*}
  where we can expand
  \[\langle \Delta_{(I^1_{3})^{(l^1+1)}} g_{j}\rangle_{I^1_{3}} = \ave{g_j, h_{(I^1_3)^{(l^1+1)}}} \ave{h_{(I^1_3)^{(l^1+1)}}}_{I_3^1}.\]
  
  For fixed $l^1$ and $l^{2,3}$ the expansion leads to the term
  \begin{equation*}
    \begin{split}
      \sum_{K \in \calD_\lambda} \sum_{(L^{2,3})^{(l^{2,3})}= K^{2,3}}  
    \sum_{\substack{I_1, I_2,I_3 \\ I_j^{(k)} = K }} 
    b_{K,(I_j)} \langle h_{(I^1_{3})^{(l^1+1)}} \otimes h_{L^{2,3}} \rangle_{I^{1}_{3} \times I^{2,3}_{1}}\\ 
      \Big\langle f_{1}, h_{(I^1_3)^{(l^1+1)}} \otimes h_{L^{2,3}} \Big\rangle \langle f_2 \rangle_{(I^1_3)^{(l^1)} \times I^{2,3}_2}  \langle f_{3}, h_{I_{3}} \rangle,
    \end{split}
  \end{equation*}
  and to the symmetrical one, where the cancellation $h_{(I^1_3)^{(l^1 + 1)}}$ is paired with the second function and $f_1$ is averaged over $(I_3^1)^{(l^1+1)}.$ 
  Again, we want to reorganize the summations and verify the correct normalization for the shifts of the form \eqref{eq:zygshift}.
  In the first parameter we will now take $(I^1_{3})^{(l^1+1)}$ as the new top cube, that is, 
  \begin{equation}\label{eq:shift1}
    \begin{split}
      &\sum_{K^1}\sum_{(L^1)^{(k^1-l^1)}=K^1} 
      \sum_{K^{2,3}\in \calD_{2^{-l^1-k^2+k^3} \ell(L^1)}}\sum_{(I_{3}^1)^{(l^1)}=L^1} \\
      &\sum_{(L^{2,3})^{(l^{2,3})}=K^{2,3}}
      \sum_{\substack{I^{2,3}_2,I^{2,3}_3 \\(I^{i}_j)^{(k^{i})}=K^{i}}} 
      c_{K^1,L^1,I^1_{3}, K^{2,3}, L^{2,3}, I^{2,3}_{2}, I^{2,3}_{3}}  
      \Big\langle f_{1}, h_{(L^1)^{(1)}} \otimes h_{L^{2,3}} 
      \Big\rangle \langle f_2 \rangle_{ L^1\times I^{2,3}_2}  \langle f_{3}, h_{I_{3}} \rangle,
    \end{split}
\end{equation}
where 
\begin{equation*}
  \begin{split}
    &c_{K^1,L^1,I^1_{3}, K^{2,3}, L^{2,3}, I^{2,3}_{2}, I^{2,3}_{3}}  \\
    &= \sum_{\substack{I^1_1, I^1_{2} \\ (I^1_j)^{(k^1)}=K^1}}
    \sum_{\substack{ I^{2,3}_{1} \subset L^{2,3} \\ (I^{i}_{1})^{(k^{i})}=K^{i} }} 
    b_{K, (I_j)}\langle h_{(L^1)^{(1)} \times L^{2,3}} \rangle_{I^1_{3}\times I^{2,3}_{1}}.
  \end{split}
\end{equation*}
Moreover, we have 
\begin{equation*}
  \begin{split}
    |c_{K^1,L^1,I^1_{3}, K^{2,3}, L^{2,3}, I^{2,3}_{2}, I^{2,3}_{3}}| 
    & \le \frac{|(L^1)^{(1)}|^{\frac{3}{2}}|I_3^1|^\frac{1}{2}}{|(L^1)^{(1)}|^2} \times \frac{|L^{2,3}|^{\frac 12}|I^{2,3}_2| |I_3^{2,3}|^{\frac 12}}{|K^{2,3}|^2}.
  \end{split}
\end{equation*}
Notice that this is the right normalization for \eqref{eq:zygshift}, 
since $f_2$ is related to $L^1$ and $|(L^1)^{(1)}| = 2 |L^1|,$ and we can change the averages into pairings against non-cancellative Haar functions.

We conclude that for some $C \ge 1$ we have
$$
C^{-1} \eqref{eq:shift1}
=\langle S_{((0,l^{2,3}), (1,k^{2,3}), (l^1+1,k^{2,3}))}(f_1,f_2),f_{3} \rangle,
$$ 
where $S_{((0,l^{2,3}), (1,k^{2,3}), (l^1+1,k^{2,3}))}$ is an operator of the desired type and of complexity $$(0,l^{2,3}), (1,k^{2,3}), (l^1+1,k^{2,3}).$$ The other term and the other case of $\Sigma_{3,2}^1 - \Sigma_{2}^2$ are analogous.

The cases $\Sigma_{m_1,3}^1 - \Sigma_{m_1}^3$ are handled almost identically, however, we need to treat 
\[\prod_{j=1}^2 \ave{g_j}_{K^{2,3}} -  \prod_{j=1}^2 \ave{g_j}_{I_{3}^{2,3}}\]
slightly differently. We expand the rectangles $I^{2,3}_3$ in the one-parameter fashion until we reach the smaller of the cubes $K^2,K^3.$ Then we continue with one-parameter expansion with only one of the cubes until we reach the bigger of the cubes $K^2,K^3.$
For example, if $k^3>k^2$, we expand as 
\begin{align*}
  &\prod_{j=1}^2 \ave{g_j}_{K^{2,3}} -  \prod_{j=1}^2 \ave{g_j}_{I_{3}^{2,3}} \\
  &=- \sum_{l^2 = 0}^{k^2-1} \Bigl[\ave{\Delta_{(I_3^{2,3})^{(l^2 + 1,l^2 + 1)}}g_1}_{(I^{2,3}_3)^{(l^{2},l^2)}} \ave{g_2}_{(I^{2,3}_3)^{(l^2,l^2)}} \\
  &\hspace{10em}+ \ave{g_1}_{(I^{2,3}_3)^{(l^{2}+1,l^2 + 1)}} \ave{\Delta_{(I_3^{2,3})^{(l^{2}+1, l^2+1)}}g_2}_{(I^{2,3}_3)^{(l^{2}, l^2)}}\Bigr]\\
  &\quad- \sum_{l^3 = k^2}^{k^3-1} \Bigl[\ave{E_{K^2}\Delta_{ (I_3^{3})^{( l^{3}+1)}}g_1}_{K^2\times (I^{3}_3)^{(l^{3})}} \ave{g_2}_{K^2\times (I^{3}_3)^{(l^{3})}} \\
  &\hspace{10em}+ \ave{g_1}_{K^2\times (I^{3}_3)^{(l^{3} + 1)}} \ave{E_{K^2}\Delta_{(I^{3})^{(l^{3}+1)}}g_2}_{K^2\times (I^{3}_3)^{(l^{3})}}\Bigr],
\end{align*}The case $k^3\le k^2$ can be expanded similarly. 
Similarly as in the previous cases, we can now write the terms in the particular form \eqref{eq:zygshift}. For example, related to the latter term, 
\begin{align*}
  \sum_{K \in \calD_\lambda} &\sum_{\substack{L^{2,3}\in\calD^{2,3}_{\lambda_{l,k}\ell(K^1)}\\L^2 = K^2 \\(L^{3})^{(k^{3}-l^{3})}=K^{3}}} \sum_{(L^{1})^{(l^1)}=K^{1}}
 \sum_{\substack{(I^{1}_3)^{(k^{1})}=K^{1}}}  
     \sum_{\substack{(I_{3}^{2})^{(k^{2})} =K^{2}\\ (I_3^3)^{(l^3)} = L^3}} \\
      &c_{K,L,I_{3}} 
      \Big\langle f_{1}, \frac{1_{K^1}}{|K^1|} \otimes h_{(L^{2,3})^{(0,1)}} 
      \Big\rangle \Big\langle f_2, h_{L^1} \otimes \frac{1_{L^{2,3}}}{|L^{2,3}|} \Big\rangle \langle f_{3}, h_{I_{3}} \rangle,
\end{align*}
where $l^3 \in \{k^2 , \ldots, k^3-1\},\lambda_{l,k} = 2^{-k^1-k^2 + l^3}$ and 
\begin{align*}
  |c_{K,L,I_{3}}|  &= \Big|\sum_{\substack{I_1, I_2\\  (I_j)^{(k)} = K\\ I_2^1 \subset L^1}} a_{K,(I_j)} |I_1|  \ave{h_{L^{1}} \otimes h_{K^{2} \times (L^3)^{(1)}}}_{I_2^{1} \times K^2 \times L^3}\Big|\\
  &\le \sum_{\substack{I_1, I_2\\  (I_j)^{(k)} = K\\ I_2^1 \subset L^1}} \frac{|I_3|^{\frac 12} |I_1||I_2|}{|K|^2} |K^2|^{-\frac 12} |(L^3)^{(1)}|^{-\frac{1}{2}}|L^1|^{-\frac 12} \\
  &= \frac{|L^1|^{\frac 12}}{|K^1|}\frac{|I_3|^{\frac 12}|K^{2}\times (L^3)^{(1)}|^{\frac 32}}{|K^{2}\times (L^3)^{(1)}|^2}.
\end{align*}
This normalization is an absolute constant away from the correct one since we consider that $K^2 \times (L^3)^{(1)}$ is the top rectangle in parameters 2 and 3.


Finally, we consider $\Sigma^1_{3,3}-\Sigma^2_{3}-\Sigma^3_{3}+\Sigma^4$ that equals to
\begin{equation}
  \begin{split}
\sum_{K \in \calD_\lambda}&\sum_{\substack{I_1,I_2,I_3 \in \calD_{Z}\\ I_j^{(k)} = K}}
b_{K,(I_j)} \\
    &\Big[ \prod_{j=1}^2\langle f_j\rangle_{K}
    -\prod_{j=1}^2 \langle f_j \rangle_{I^1_{3} \times K^{2,3}}
    -\prod_{j=1}^2 \langle f_j \rangle_{K^1 \times I^{2,3}_{3}}
    + \prod_{j=1}^2 \langle f_j \rangle_{I_{3}} \Big] \langle f_{3}, h_{I_{3}} \rangle.
  \end{split}
\end{equation} 

As we already showed, we can expand
\begin{align*}
  &\prod_{j=1}^2\langle f_j\rangle_{K}
    -\prod_{j=1}^2 \langle f_j \rangle_{I^1_{3} \times K^{2,3}} \\
    &= -\sum_{l^1 = 0}^{k^1-1} \Big(\langle \Delta_{(I^1_{3})^{(l^1+1)}} g_{1}\rangle_{I^1_{3}} \langle g_2 \rangle_{(I^1_{3})^{(l^1)}}  +  \langle g_1 \rangle_{(I^1_{3})^{(l^1+1)}} \langle \Delta_{(I^1_{3})^{(l^1+1)}} g_{2}\rangle_{I^1_{3}}\Big),
\end{align*}
where $g_j = \ave{f_j}_{K^{2,3}},$ and similarly for \[
  \prod_{j=1}^n \langle f_j \rangle_{I_{3}} -\prod_{j=1}^2 \langle f_j \rangle_{K^1 \times I^{2,3}_{3}}\] we get same expansion with the positive sign and $g_j = \ave{f_j}_{I^{2,3}_3}.$

  Then we sum the two expansions together and expand in the parameters 2 and 3. That is, we will expand 
    \begin{align*}
    &\sum_{l^1 = 0}^{k^1-1} \ave{h_{(I_{3}^1)^{(l^1 + 1)}}}_{(I_{3}^1)^{(l^1)}} \Big\langle  f_1,h_{(I^1_{3})^{(l^1+1)}} \otimes \frac{1_{K^{2,3}}}{|K^{2,3}|} \Big\rangle \langle f_2 \rangle_{(I^1_{3})^{(l^1)}\times K^{2,3}}  \\
    &\hspace{3em}-  \Big\langle  f_1, h_{(I^1_{3})^{(l^1+1)}} \otimes \frac{1_{I_3^{2,3}}}{|I_3^{2,3}|} \Big\rangle \langle f_2 \rangle_{(I^1_{3})^{(l^1)}\times I_3^{2,3}}. 
      \end{align*}
  Thus, we get, for example when $k^2 < k^3$, that  \begin{align*}
   &\sum_{l^1=0}^{k^{1}-1} \sum_{l^{2} = 0}^{k^{2}-1} \sum_{K \in \calD_\lambda} \sum_{\substack{L^1\in \calD^1 \\ (L^1)^{(k^1-l^1)} = K^1}} \sum_{\substack{L^{2,3} \in \calD^{2,3}_{2^{-l^1}\ell(L^1)} \\ (L^{2,3})^{(k^{2} - l^{2},k^3-l^2)}=K^{2,3} }} \sum_{\substack{I_3 \in \calD_Z\\ (I_3)^{(l^1,l^{2},l^2)} = L}}\\
    &\qquad\times c_{K,L, I_3} \Ave{f_1, h_{(L^1)^{(1)}} \otimes h_{(L^{2,3})^{(1,1)}}^0} \Ave{f_2,h_{L^1}^0 \otimes h_{(L^{2,3})^{(1,1)}}} \ave{f_3,h_{I^3}}\\
    &+\sum_{l^1=0}^{k^{1}-1} \sum_{l^{3} = k^2}^{k^{3}-1} \sum_{K \in \calD_\lambda} \sum_{\substack{L^1\in \calD^1 \\ (L^1)^{(k^1-l^1)} = K^1}} \sum_{\substack{L^{2,3} \in \calD_{2^{-l^1-k^2 + l^3} \ell(L^1)} \\ L^2 = K^2 \\ (L^{3})^{(k^{3} - l^{3})}=K^{3} }} \sum_{\substack{I_3 \in \calD_Z\\ (I_3)^{(l^1,k^2,l^{3})} = L}}\\
    &\qquad\times c_{K,L, I_3} \Ave{f_1, h_{(L^1)^{(1)}} \otimes h_{(L^{2,3})^{(0,1)}}^0} \Ave{f_2,h_{L^1}^0 \otimes h_{(L^{2,3})^{(0,1)}}} \ave{f_3,h_{I^3}}.
  \end{align*}
  Here 
\begin{align*}
  |c_{K,L, I_3} | &=\Big| \sum_{\substack{I_1,I_2 \in \calD_Z\\ I_j^{k} = K}} a_{K,(I_j)} |I_1| |L^1|^{-\frac{1}{2}}|(L^{2,3})^{(1)}|^{- \frac 12} \ave{h_{(L^1)^{(1)}} \otimes h_{(L^{2,3})^{(1)}}}_{L^1 \times L^{2,3}} \Big|\\
  &\le \frac{|I_3|^{\frac 12}|(L^{1})^{(1)}|^{\frac 32}}{|(L)^{(1)}|^{2}}|L^1|^{-\frac{1}{2}}|(L^{2,3})^{(1)}|^{- \frac 12}\sim \frac{|I_3|^{\frac 12}|(L^{1})^{(1)}|^\frac{1}{2} |L^{1}|^{\frac 12}|(L^{2,3})^{(1)}|^{\frac 32}}{|(L^1)^{(1)}|^{2} |(L^{2,3})^{(1)}|^2}.
\end{align*}  
We abused notation slightly by $(L^{2,3})^{(1)}$ meaning both $(L^{2,3})^{(1,1)}$ and $(L^{2,3})^{(0,1)}$.
The other terms are handled analogously.
\end{proof}

\section{Boundedness of Zygmund shifts}\label{sec:boundedness}
In this section we prove the boundedness of Zygmund shifts. 
We first prove the following.  A collection $\mathscr S$ is called $\gamma$-sparse if there are
pairwise disjoint subsets $E(S)\subset S$, $S \in \mathscr S$, with $\abs{E(S)} \ge \gamma\abs{S}$. Often the precise
value of $\gamma$ is not important and we just talk about sparse collections.
\begin{prop}\label{prop:sparse}
Let $\lambda=2^k$ for some $k\in \Z$ and 
\[
\Lambda (f_1, f_2, f_3)= \sum_{K\in \calD^{2,3}_\lambda} \sum_{(I_j)^{(\ell_j)}= K} \frac{\prod_{j=1}^3 |I_j|^{\frac 12}}{|K|^2} |\langle f_1, h_{I_1}^0\rangle|\cdot| \langle f_2, h_{I_2}\rangle |\cdot|\langle f_3, h_{I_3}\rangle|.
\] 
Then there exists a sparse collection $\mathcal S \subset \calD^{2,3}_\lambda$ such that 
\[
 \Lambda (f_1, f_2, f_3) \lesssim \max \{k^2, k^3\}\sum_{S\in \calS}|S| \prod_{j=1}^3 \langle |f_j|\rangle_S.
\]
\end{prop}
\begin{proof}
The proof is an easy adaptation of the sparseness argument in \cite{LMOV}*{Section 5}. In fact, we only need to check the validity of 
\[
 \Lambda (f_1, f_2, f_3)  \lesssim \|f_1\|_{L^p}\|f_2\|_{L^q}\|f_3\|_{L^r},
\] where $p,q,r\in (1, \infty)$ and $1/p+1/q+1/r=1$. This can be done by direct computation:
\begin{align*}
 \Lambda (f_1, f_2, f_3)  &\le \int f_1 \sum_{K\in \calD^{2,3}_\lambda} \langle | \Delta_K^{\ell_2} f_2|\rangle_K  \langle | \Delta_K^{\ell_3} f_3|\rangle_K 1_K\\
&\le \|f_1\|_{L^p }\Big\|\Big(\sum_{K\in \calD_{\lambda}^{2,3}}\big[ M_{\calD_{\lambda}^{2,3}} | \Delta_K^{\ell_2} f_2| \big]^2\Big)^{\frac 12}\Big\|_{L^q}\Big\|\Big(\sum_{K\in \calD_{\lambda}^{2,3}}\big[ M_{\calD_{\lambda}^{2,3}} | \Delta_K^{\ell_3} f_3| \big]^2\Big)^{\frac 12}\Big\|_{L^r}\\
&\lesssim  \|f_1\|_{L^p}\|f_2\|_{L^q}\|f_3\|_{L^r}.
\end{align*}
\end{proof}

\begin{prop}\label{prop:model}
  Let $Q_k$, $k = (k^1,k^2,k^3)$, be a bilinear Zygmund shift as in Section \ref{subsec:shifts}, and let $1 < p_1,p_2< \infty$ and $\frac{1}{2} < p < \infty$ with $\frac{1}{p} := \frac{1}{p_1} + \frac{1}{p_2}.$ Let \[w_1,w_2 \in A_{p}(\R\times \R \times \R),\qquad \text{ and }\qquad w := w_1^{\frac{p}{p_1}}w_2^{\frac{p}{p_2}}.\] Then, for every $\eta \in (0,1)$ we have
  \[\|Q_k(f_1,f_2)\|_{L^p(w)} \lesssim \max_{i}\{k^i\}^2 2^{k^1\eta} \|f_1\|_{L^{p_1}(w_1)} \|f_2\|_{L^{p_2}(w_2)}.\]
\end{prop}
\begin{proof}
We prove the weighted boundedness $L^{4}(w_1) \times L^{4}(w_2) \to L^2(w)$, of the tri-parameter bilinear shifts of Zygmund nature \eqref{eq:zygshift}.
We do this with tri-parameter weights $w_i \in A_{4}$.
We then extrapolate the result to the full bilinear range
using the traditional multilinear extrapolation by Grafakos--Martell (and Duoandikoetxea) \cites{GM,DU}.
Our result then follows from Proposition \ref{prop:sumofshifts}.

Note that if we have $I_3\in \calD_Z$ in \eqref{eq:zygshift}, then the related $\lambda$ in Proposition \ref{prop:sparse} is
$$2^{\ell_3^3-\ell_3^2-\ell_3^1}|L^1|.$$ 
(For other cases, for instance if $I_1^1\times I_2^{2,3}\in \calD_Z$, then $\lambda= 2^{\ell_2^3-\ell_2^2-\ell_1^1}|L^1|$). 
Assume $v\in A_{4,\lambda}(\R^2)$; recall that $A_{p,\lambda}(\R^2)$ is defined similarly as $A_{p}(\R^2)$ except that the supremum is taken over rectangles $R=I\times J$ with $|J|=\lambda |I|$. Then 
\begin{align*}
\sum_{S\in \calS}|S| \prod_{j=1}^3 \langle |f_j|\rangle_S  &= \sum_{S\in \calS} \langle |f_1|\rangle_S  \langle |f_2|\rangle_S \langle |f_3| v^{-1}\rangle_S^v v(S).
\end{align*}
Since for any $R\in \mathcal S$,
\begin{align*}
\sum_{\substack{S\subset R\\ S\in \calS}}v(S)&=\sum_{\substack{S\subset R\\ S\in \calS}} \frac{v( S)}{|S|} | S|\lesssim  \sum_{\substack{S\subset R\\ S\in \calS}} \frac{v(  S)}{|  S|}|E_S| \le   \int_{  R} M_{\calD^{2,3}_{\lambda}}(v1_{ R})\lesssim_{[v]_{A_{4,\lambda}(\R^2)}} v( R),
\end{align*}by the Carleson embedding theorem we have 
\begin{equation}\label{eq:e300}
\sum_{S\in \calS}|S| \prod_{j=1}^3 \langle |f_j|\rangle_S \lesssim_{[v]_{A_{4,\lambda}(\R^2)}}  \int_{\R^2} M_{\calD_{\lambda}^{2,3}}|f_1| M_{\calD_{\lambda}^{2,3}}|f_2| M_{\calD_{\lambda}^{2,3}}^{v}(|f_3|v^{-1})v.
\end{equation}

Now, given weights $w_j\in A_{4}(\R^3)$, $j=1,2$, we know that $w=w_1^{1/{2}}w_2^{1/{2}}\in A_{4}(\R^3)$. We have 
\begin{align*}
|\langle S(f_1, f_2), f_3\rangle|&= \sum_{L^1}\sum_{ (I_j^1)^{(\ell_j^1)}= L^1}\frac{\prod_{j=1}^3 |I_j^1|^{\frac 12}}{|L^1|^2} \Lambda (\langle f_1, h_{I_1^1}\rangle, \langle f_2, h_{I_2^1}^0\rangle, \langle f_3, h_{I_3^1}\rangle).
\end{align*}
Note that $\langle w\rangle_{L^1}\in A_{4,\lambda}(\R^2)$  with $[\langle w\rangle_{L^1}]_{A_{4,\lambda}(\R^2)}\le [w]_{A_4}$ for any $\lambda$. Thus, applying \eqref{eq:e300} with $v=\langle w\rangle_{L^1}$ we have 
\begin{align*}
&|\langle S(f_1, f_2), f_3\rangle|\\&\lesssim \max_{i} \{k^i\}  \sum_{L^1}\sum_{ (I_j^1)^{(\ell_j^1)}= L^1}\frac{\prod_{j=1}^3 |I_j^1|^{\frac 12}}{|L^1|^2}\int_{\R^2} M_{\calD_\lambda^{2,3}} \langle f_1, h_{I_1^1}\rangle  M_{\calD_\lambda^{2,3}} \langle f_2, h_{I_2^1}^0\rangle
M_{\calD_\lambda^{2,3}}^v (\langle f_3, h_{I_3^1}\rangle v^{-1})v\\
&= \max_{i} \{k^i\}  \sum_{L^1}\int_{\R^3}\langle M_{\calD }|\Delta_{L^1}^{\ell_1^1}f_1|\rangle_{L^1}\langle M_{\calD }|f_2|\rangle_{L^1}  \sum_{ (I_3^1)^{(\ell_3^1)}= L^1}|I_3^1|^{\frac 12}M_{\calD_\lambda^{2,3}}^{\langle w\rangle_{L^1}} (\langle f_3, h_{I_3^1}\rangle \langle w\rangle_{L^1}^{-1})\frac {1_{L^1}}{|L^1|}w\\
&\le  \max_{i} \{k^i\}  \Big\| \Big(\sum_{L^1}\big[M_{\calD^1}M_{\calD }|\Delta_{L^1}^{\ell_1^1}f_1|\big]^2 \Big)^{\frac 12}\Big\|_{L^4(w_1)}\|M_{\calD^1} M_{\calD }|f_2|\|_{L^4(w_2)}\\
&\hspace{3cm}\times \Big\| \Big(\sum_{L^1} \Big[\sum_{ (I_3^1)^{(\ell_3^1)}= L^1}|I_3^1|^{\frac 12}M_{\calD_\lambda^{2,3}}^{\langle w\rangle_{L^1}} (\langle f_3, h_{I_3^1}\rangle \langle w\rangle_{L^1}^{-1})|L^1|^{-1}\Big]^21_{L^1}\Big)^{\frac 12}\Big\|_{L^2(w)}.
\end{align*}
By the well-know square function and maximal function estimates we have  
\begin{align*}
\Big\| \Big(\sum_{L^1}\big[M_{\calD^1}M_{\calD}|\Delta_{L^1}^{\ell_1^1}f_1|\big]^2 \Big)^{\frac 12}\Big\|_{L^4(w_1)}\lesssim  \|f_1\|_{L^4(w_1)}
\end{align*}and 
\begin{align*}
\|M_{\calD^1} M_{\calD}|f_2|\|_{L^4(w_2)}\lesssim  \|f_2\|_{L^4(w_2)}.
\end{align*}
The estimate of the last term is a bit tricky. By the (one parameter)vector-valued estimates 
of $M_{\calD_\lambda^{2,3}}^{\langle w\rangle_{L^1}}$ (see e.g. \cite{LMV-GEN}*{Proposition 4.3} for a bi-parameter version (the proof easily adapts to the one-parameter case)), we have 
\begin{align*}
&\Big\| \Big(\sum_{L^1} \Big[\sum_{ (I_3^1)^{(\ell_3^1)}= L^1}|I_3^1|^{\frac 12}M_{\calD_\lambda^{2,3}}^{\langle w\rangle_{L^1}} (\langle f_3, h_{I_3^1}\rangle \langle w\rangle_{L^1}^{-1})|L^1|^{-1}\Big]^21_{L^1}\Big)^{\frac 12}\Big\|_{L^2(w)} \\
&\le 2^{\ell_3^1 \eta} \Big\| \Big(\sum_{L^1} \Big[\sum_{ (I_3^1)^{(\ell_3^1)}= L^1}|I_3^1|^{\frac s2}M_{\calD_\lambda^{2,3}}^{\langle w\rangle_{L^1}} (\langle f_3, h_{I_3^1}\rangle \langle w\rangle_{L^1}^{-1})^s|L^1|^{-\frac s2}\Big]^{\frac 2s} \Big)^{\frac 12}\Big\|_{L^2(\langle w\rangle_{L^1})}\\
&\lesssim 2^{\ell_3^1 \eta}\Big\| \Big(\sum_{L^1} \Big[\sum_{ (I_3^1)^{(\ell_3^1)}= L^1}|I_3^1|^{\frac s2} \big|\langle f_3, h_{I_3^1}\rangle \langle w\rangle_{L^1}^{-1}\big|^s|L^1|^{-\frac s2}\Big]^{\frac 2s} \Big)^{\frac 12}\Big\|_{L^2(\langle w\rangle_{L^1})}\\
&\le 2^{\ell_3^1 \eta}\Big\| \Big(\sum_{L^1} \Big[\sum_{ (I_3^1)^{(\ell_3^1)}= L^1}|I_3^1|^{\frac 12} |\langle f_3, h_{I_3^1}\rangle| \langle w\rangle_{L^1}^{-1}|L^1|^{-\frac 12}\Big]^{2} \Big)^{\frac 12}\Big\|_{L^2(\langle w\rangle_{L^1})}\\
&\lesssim 2^{\ell_3^1 \eta}\|f_3\|_{L^2(w^{-1})},
\end{align*}where $s=(1/\eta)'$ and in the last step we have used~\cite{LMV-GEN}*{Proposition 5.8}. Thus,
\[
\| S(f_1, f_2)\|_{L^2(w)}\lesssim \max_{i} \{k^i\}  2^{k^1\eta}\|f_1\|_{L^4(w_1)}\|f_2\|_{L^4(w_2)}. 
\]
  
\end{proof}

Now we are able to conclude the proof of Theorem \ref{thm:bilinearczz}.
\begin{proof}[Proof of Theorem \ref{thm:bilinearczz}]
By the representation formula discussed in Sections \ref{sec:ZygDeco} 
and \ref{sec:ZygDecoRefined}, the coefficient estimates in 
Section \ref{sec:coefficientesti}  
(in particular \eqref{eq:coefficientesti}) we get that 
\begin{align*}
  \ave{T(f_1,f_2),f_3} = &C \E_{\sigma} \sum_{k^1, k^2, k^3 = 2}^\infty \hspace*{-1em} (|k| +1)^{2} \varphi(k)
  \sum_{I \in \calD_{Z}(k)}  \frac{\ave{Q_{(k^1, k^2, k^3)} (f_1,f_2),f_3}}{C(|k| + 1)^2 \varphi(k)}.
\end{align*}
Thus, for $p_1,p_2 \in (1,\infty)$ so that $p \in (1,\infty),$ we conclude by Proposition \ref{prop:model} that 
\begin{align*}
\|T(f_1, f_2)\|_{L^p(w)}&\lesssim  \sum_{k^1, k^2, k^3 = 2}^\infty (|k| +1)^{2}  \varphi(k) \max_{i}\{k^i\}^2 2^{k^1\eta} \|f_1\|_{L^{p_1}(w_1)} \|f_2\|_{L^{p_2}(w_2)}\\
&\lesssim \|f_1\|_{L^{p_1}(w_1)} \|f_2\|_{L^{p_2}(w_2)},
\end{align*}
where we need to take $\eta<\alpha_1$.
Consequently, we can now pass the result to the full bilinear range
using the traditional multilinear extrapolation \cites{GM,DU}.
\end{proof}

\section{Linear commutators in the Zygmund dilation setting}\label{sec:commutator}
In this section we return to the linear theory and complete the following 
commutator estimate left open by previous results. 
This requires new and interesting paraproduct estimates.
For the context, see the explanation below.
\begin{thm}
  Let $b \in L^1_{loc}$ and  $T$ be a linear CZZ operator as in \cite{HLMV-ZYG}. Let $\theta \in (0,1]$ be the kernel exponent measuring the decay in terms of the Zygmund ratio 
  \[D_\theta (x) := \Bigl(\frac{|x_1x_2|}{|x_3|} + \frac{|x_3|}{|x_1x_2|} \Bigr)^{-\theta}.\] Then
  \[\|[b,T]\|_{L^p \to L^p} \lesssim \|b\|_{\bmo_Z} \]
  whenever $p \in (1,\infty).$
\end{thm}
Here the definition of the little BMO is given by  \[
  \|b\|_{\bmo_Z} := \sup_{\calD_Z} \sup_{R \in \calD_Z} \frac{1}{|R|}\int_{R}|b(x) -\ave{b}_R| \ud x < \infty,
  \]
where the supremum is over all different collections of Zygmund rectangles $\calD_Z$ and then over all $R \in \calD_Z.$ 

This theorem was previously considered in \cite{DLOPW-ZYGMUND} using the so-called Cauchy trick.
That method requires weighted bounds with Zygmund weights. But we now know \cite{HLMV-ZYG}
how delicate such weighted bounds are -- 
weighted 
bounds with Zygmund weights do not in general hold if $\theta < 1$. However, the commutator bounds are still true
-- but we need a different proof, presented here. 
It suffices to prove the boundedness of commutators $[b,Q_{k}]$ for any linear shift $Q_k$ of the Zygmund dilation type.

For $\theta = 1$ we could use the Cauchy trick and the weighted bounds from \cite{HLMV-ZYG} -- this would 
give weighted commutator estimates with Zygmund weights.

We begin by recording lemmas that we need for the main proofs of this section.
\begin{lem}\label{lem:equiv}
Let $b$ be a locally integrable function. 
Then the following are equivalent
\begin{enumerate}
\item  $b\in \bmo_{\calD_Z},$
\item $$\max\Big\{\sup_{I^1 \in \calD^1}\|\ave{b}_{I^1,1}\|_{\BMO_{\calD^{2,3}_{\ell(I^1)}}} , \esssup_{(x_2,x_3)\in\R^2} \|b(\cdot,x_2,x_3)\|_{\BMO} \Big\} < \infty,$$\label{it:1}
\item $$\max\Big\{\sup_{I^2 \in \calD^2}\|\ave{b}_{I^2,2}\|_{\BMO_{\calD^{2,3}_{\ell(I^2)}}} , \esssup_{(x_1,x_3)\in\R^{2}}\|b(x_1, \cdot,x_3)\|_{\BMO} \Big\} < \infty.$$\label{it:2}
\end{enumerate}
\end{lem}
For completeness, we give the proof.
\begin{proof}
Let us begin showing that  $\bmo_Z \implies $ \eqref{it:1} (and by symmetry also \eqref{it:2}).
Clearly, for all Zygmund rectangles $I = I^1\times I^2 \times I^3 \in \calD_Z$ we have 
\begin{align}
\|b\|_{\bmo_Z} &\ge \frac{1}{|I|} \int_{I} |b - \ave{b}_I| \ge \frac{1}{|I^{2,3}|} \int_{I^{2,3}} |\ave{b}_{I^1,1} - \ave{b}_I|.\label{eq:e1}
\end{align}
So by uniform boundedness we immediately get
$$
\|\ave{b}_{I^1,1}\|_{\BMO_{\calD^{2,3}_{\ell(I^1)}}} :=\sup_{I^{2,3} \in \calD_{\ell(I^1)}^{2,3}} \frac{1}{|I^{2,3}|} \int_{I^{2,3}} |\ave{b}_{I^1,1} - \ave{\ave{b}_{I^1,1} }_{I^{2,3}}| \leq \|b\|_{\bmo_Z} < \infty.
$$
We move on to proving the second assertion inside  \eqref{it:1}. For fixed $I^1 \in \calD^1$ we define $f_{I^1}(x^2,x^3) := \int_{I^1}| b(x^1,x^2,x^3) - \ave{b}_{I^1}(x^2,x^3) | \ud x^1.$ Then for every $I^{2,3} \in \calD_{\ell(I^1)}^{2,3}$ we have
\begin{align*}
\ave{f_{I^1}}_{I^{2,3}} &\leq  \frac{1}{|I^{2,3}|} \int_{I^{2,3}} \int_{I^1} |b - \ave{b}_{I} | +  \frac{1}{|I^{2,3}|} \int_{I^{2,3}} \int_{I^1} |\ave{b}_{I^1,1} - \ave{b}_I|\leq 2 |I^1|\|b\|_{\bmo_Z},
\end{align*}
where last inequality holds by definition and the above estimate \eqref{eq:e1}.
Now, by the Lebesgue differentiation theorem we get for $(x^2,x^3) \in \R^2 \setminus N(I^1),$ where $N(I^1)$ is a null set depending on $I^1,$ that
$$
f_{I^1}(x^2,x^3) \leq 2 |I^1|\|b\|_{\bmo_Z}.
$$
It is then easy to conclude that 
$$
\|b(\cdot,x_2,x_3)\|_{\BMO}  \leq 2 \|b\|_{\bmo_Z}
$$
for almost every $(x^2,x^3) \in \R^2.$

Conversely,
\begin{align*}
 \int_{I} |b - \ave{b}_I| &\leq  \int_{I} |b - \ave{b}_{I^1,1}|   +  \int_{I} |\ave{b}_{I^1,1} - \ave{b}_I| \\
 &\leq |I^1| \int_{I^{2,3}} \|b(\cdot,x^2,x^3)\|_{\BMO} + |I|\|\ave{b}_{I^1,1}\|_{\BMO_{\ell(I^1)}} \leq |I|(C_1 + C_2),
\end{align*}
where $C^1 := \esssup_{(x^2,x^3)\in \R^{2}}\|b(\cdot,x^2,x^3)\|_{\BMO}$ and $C_2 := \sup_{I^1} \|\ave{b}_{I^1,1}\|_{\BMO_{\ell(I^1)}}.$
\end{proof}

Then the usual duality results imply the following.
\begin{cor}\label{cor:H1little}
If $b \in \bmo_Z$ and $I^1$ is fixed, then
$$
\sum_{I^{2,3} \in \calD^{2,3}_{\ell(I^1)}}\ave{\ave{b}_{I^1},h_{I^{2,3}}} \varphi_{I^{2,3}} \lesssim \|b\|_{\bmo_Z}\Big\| \Big(\sum_{I^{2,3} \in \calD^{2,3}_{\ell(I^1)}} |\varphi_{I^{2,3}}|^2 \frac{1_{I^{2,3}}}{|I^{2,3}|}\Big)^\frac{1}{2}\Big\|_{L^1}.
$$ 
Also, for fixed $(x_2,x_3)$, we have
$$
\sum_{I^{1} \in \calD^1} \ave{b,h_{I^1}}_1 \varphi_{I^{1}} \lesssim \|b\|_{\bmo_Z}\Big\| \Big(\sum_{I^{1} \in \calD^1} |\varphi_{I^{1}}|^2 \frac{1_{I^{1}}}{|I^{1}|}\Big)^\frac{1}{2}\Big\|_{L^1}.
$$
\end{cor}

Using the duality type estimates we can use the square function lower bounds to prove the inclusion of product type spaces.
\begin{defn}\label{defn:BMOprodZ}
Given a lattice of Zygmund rectangles $\calD_Z$ and a sequence of scalars $B = (b_I)_{I \in \calD_Z}$ we define
$$
\|B\|_{\BMO_{\operatorname{prod}}} := \sup_{\Omega} \Bigg(\frac1{|\Omega|} \sum_{\substack{I \in \calD_Z\\ I\subset \Omega}} |b_I|^2\Bigg)^{\frac 12}.
$$
\end{defn}

The inclusion of the little BMO space can be easily seen from the duality estimate 
\begin{equation}\label{BmoprodH1}
\|B\|_{\BMO_{\operatorname{prod}}} \sim \sup\Big\{ \sum_{I \in \calD_Z} |a_I| |b_I| \colon \Big\| \Big( \sum_{I \in \calD_Z} |a_I|^2 \frac{1_I}{|I|} \Big)^{\frac{1}{2}} \Big\|_{L^1} \le 1 \Big\}. 
\end{equation}

\subsection{Paraproduct expansions}
 Here the correct expansions style is the Zygmund martingale expansion similar to \cite{HLMV-ZYG}*{Equation (5.22)}. This gives
 \begin{align}\label{eq:dec}
 bf &= \sum_{I \in \calD_Z}  \Big[ \Delta_{I,Z}b \Delta_{I,Z}f + \Delta_{I, Z}b \Delta_{I^1} E_{I^{2,3}} f + \Delta_{I^1} E_{I^{2,3}} b\Delta_{I, Z} f \\
   &\quad+\Delta_{I, Z}b E_{I^1}  \Delta_{I^{2,3}} f + \Delta_{I, Z}bE_{I^1} E_{I^{2,3}} f +\Delta_{I^1} E_{I^{2,3}}bE_{I^1}  \Delta_{I^{2,3}} f \nonumber\\
    &\quad+  E_{I^1} \Delta_{I^{2,3}} b \Delta_{I, Z} f +E_{I^1} \Delta_{I^{2,3}}b\Delta_{I^1} E_{I^{2,3}}f+ E_{I^1} E_{I^{2,3}}b \Delta_{I, Z} f  \Big]\nonumber\\
 &=: \sum_{i,j= 1}^3 a_{i,j}(b,f),\nonumber
 \end{align}
 where, for example, $a_{1,1} = \sum_{I\in \calD_Z} \Delta_{I,Z}b \Delta_{I,Z}f$ and
 $$
 a_{1,2} = \sum_{I\in \calD_Z} \Delta_{I,Z}b \Delta_{I^1} E_{I^{2,3}}f,
 $$
 i.e., interpret so that rows correspond to the first index $i$ and columns correspond with the second index $j.$


\begin{lem}
If $b \in \bmo_Z,$ then the paraproducts $a_{i,j}$ such that $(i,j) \neq (3,3)$   are bounded.
That is, 
$$
\|a_{i,j}(b,f)\|_{L^p} \lesssim \|b\|_{\bmo_Z}\|f\|_{L^p}, \qquad 1 < p < \infty.
$$
\end{lem}

\begin{proof}
\emph{Case 1: product type $i\neq 3\neq j.$} We begin with the paraproducts where it would suffice to have a product BMO type assumption
(but recall that little BMO is a subset).
The symmetry $\Pi = a_{1,1}$ is essentially trivial. By \eqref{BmoprodH1} we have
\begin{align*}
|\langle \Pi f, g\rangle| &\lesssim \Big\| \Big( \sum_{I \in \calD_Z}  |\langle f, h_{I,Z} \rangle|^2 \langle |g| \rangle_I^2 \frac{1_I}{|I|} \Big)^{\frac{1}{2}} \Big\|_{L^1} \\
&\le \Big\| \Big( \sum_{I \in \calD_Z}  \langle |\Delta_{I,Z} f| \rangle_I^2 1_I \Big)^{\frac{1}{2}} \Big\|_{L^p} 
\| M_Z g \|_{L^{p'}}  \\
&\lesssim \Big\| \Big( \sum_{I \in \calD_Z}  M_{Z}(\Delta_{I,Z} f)^2  \Big)^{\frac{1}{2}} \Big\|_{L^p} \|g\|_{L^{p'}} \\
&\lesssim \| S_Z f\|_{L^p} \|g\|_{L^{p'}} \lesssim \| f\|_{L^p} \|g\|_{L^{p'}}.
\end{align*}

The `twisted' case $\Pi = a_{1,2}$ (and the symmetrical $a_{2,1}$) is trickier.
Indeed, to decouple $f$ and $g$ we cannot blindly take maximal functions only in some parameters -- this would break the Zygmund structure. 
In any case, we begin with the application of \eqref{BmoprodH1} to get
$$
|\langle \Pi f, g\rangle| \lesssim \Big\| \Big(\sum_{I \in \calD_Z} 
\Big| \Big \langle f, \frac{1_{I^1}}{|I^1|} \otimes h_{I^2 \times I^3} \Big \rangle \Big \langle g, h_{I^1} \otimes \frac{1_{I^2 \times I^3}}{|I^2 \times I^3|} \Big \rangle \Big|^2
\frac{1_I}{|I|} \Big)^{\frac{1}{2}} \Big\|_{L^1}.
$$

The above is an $L^1$ norm, while $L^2$ would be nice. This is where $A_{\infty}$ extrapolation comes in. We fix $\nu \in A_{\infty, Z},$ and move to estimate
$$
\Big\| \Big(\sum_{I \in \calD_Z} 
\Big| \Big \langle f, \frac{1_{I^1}}{|I^1|} \otimes h_{I^2 \times I^3} \Big \rangle \Big \langle g, h_{I^1} \otimes \frac{1_{I^2 \times I^3}}{|I^2 \times I^3|} \Big \rangle \Big|^2
\frac{1_I}{|I|} \Big)^{\frac{1}{2}} \Big\|_{L^2(\nu)}.
$$

We will soon show that
\begin{equation}\label{H:eq1}
\begin{split}
\Big\| \Big(\sum_{I \in \calD_Z} 
\Big| \Big \langle f, \frac{1_{I^1}}{|I^1|} \otimes h_{I^2 \times I^3} \Big \rangle \Big \langle g&, h_{I^1} \otimes \frac{1_{I^2 \times I^3}}{|I^2 \times I^3|} \Big \rangle \Big|^2
\frac{1_I}{|I|} \Big)^{\frac{1}{2}} \Big\|_{L^2(\nu)} \\
&\lesssim \Big\| M_{Z}f \Big(  \sum_{I^1\in \calD^1} M_{Z}(\Delta_{I^1}g)^2 \Big)^{1/2} \Big\|_{L^2(\nu)}.
\end{split}
\end{equation}
The $A_{\infty}$ extrapolation, Theorem \ref{thm:ainftyextra}, then implies that this inequality holds also in $L^p(\nu)$, $p \in (0,\infty)$, $\nu \in A_{\infty, Z}$. We take $p = 1$ and $\nu \equiv 1$
to get that
\begin{align*}
|\langle \Pi f, g\rangle| &\lesssim  \Big\| M_{Z}f \Big(  \sum_{I^1\in \calD^1} M_{Z}(\Delta_{I^1}g)^2 \Big)^{1/2} \Big\|_{L^1} \\
&\le \| M_{Z}f \|_{L^p} \Big\| \Big(  \sum_{I^1\in \calD^1} M_{Z}(\Delta_{I^1}g)^2 \Big)^{1/2} \Big\|_{L^{p'}} \lesssim \| f\|_{L^p} \|g\|_{L^{p'}}.
\end{align*}

It remains to prove \eqref{H:eq1}. We write
\begin{align*}
&\Big\| \Big(\sum_{I \in \calD_Z} 
\Big| \Big \langle f, \frac{1_{I^1}}{|I^1|} \otimes h_{I^2 \times I^3} \Big \rangle \Big \langle g, h_{I^1} \otimes \frac{1_{I^2 \times I^3}}{|I^2 \times I^3|} \Big \rangle \Big|^2
\frac{1_I}{|I|} \Big)^{\frac{1}{2}} \Big\|_{L^2(\nu)}^2\\
&= \sum_{I^1\in \calD^1}\sum_{I^2\times I^3\in \calD^{2,3}_{\ell(I^1)}} \Big| \Big\langle f, \frac{1_{I^1}}{|I^1|}\otimes h_{I^2\times I^3}\Big\rangle\Big|^2 \Big| \Big \langle g, h_{I^1}\otimes \frac{1_{I^2\times I^3}}{|I^2\times I^3|}\Big \rangle \Big|^2 \langle \nu \rangle_{I}.
\end{align*}
Fix some $I^1 \in \calD^1$. Let $I^2_0 \times I^3_0 \in \calD^{2,3}_{\ell(I^1)}$ and suppose $\varphi_1$, $\varphi_2$ and $\varphi_3$
are locally integrable functions in $\R^2$. Then, there exists a sparse collection 
$\calS=\calS(I^2_0 \times I^3_0, \varphi_1, \varphi_2, \varphi_3)\subset \calD^{2,3}_{\ell(I^1)}(I^2_0 \times I^3_0)$ so that
$$
\sum_{\substack{I^2\times I^3\in \calD^{2,3}_{\ell(I^1)} \\ I^2 \times I^3 \subset I_0^2 \times I_0^3}} 
|\langle \varphi_1,  h_{I^2\times I^3}\rangle|^2 
|\langle \varphi_2 \rangle_{I^2\times I^3} |^2 \langle \varphi_3 \rangle_{I^2\times I^3}
\lesssim \sum_{Q \in \calS} \langle |\varphi_1| \rangle_{Q} ^2\langle |\varphi_2| \rangle_{Q} ^2 \langle |\varphi_3| \rangle_{Q}|Q|.
$$
We use this with the functions $\varphi_1=\langle f \rangle_{I^1} $, 
$\varphi_2 = \langle g, h_{I^1} \rangle$ and $\varphi_3= \langle \nu \rangle_{I^1}$ to have that for some sparse collection 
$\calS=\calS(I^1, I^2_0 \times I^3_0,f,g,\nu) \subset \calD^{2,3}_{\ell(I^1)}$ there holds that 
\begin{equation*}
\begin{split}
\sum_{\substack{I^2\times I^3\in \calD^{2,3}_{\ell(I^1)} \\ I^2 \times I^3 \subset I_0^2 \times I_0^3}}  
&\Big| \Big\langle f, \frac{1_{I^1}}{|I^1|}\otimes h_{I^2\times I^3}\Big\rangle\Big|^2 \Big| \Big \langle g, h_{I^1}\otimes \frac{1_{I^2\times I^3}}{|I^2\times I^3|}\Big \rangle \Big|^2 \langle \nu \rangle_{I} \\
& \lesssim \sum_{Q \in \calS} \langle |\langle f \rangle_{I^1}  | \rangle_{Q} ^2
\langle |\langle g, h_{I^1} \rangle| \rangle_{Q} ^2  \langle \nu \rangle_{I^1} (Q) \\
&\le \sum_{Q \in \calS} \Big(\Big\langle \big(M^{2,3}_{\ell(I^1)} \langle f \rangle_{I^1}\big)\big(M^{2,3}_{\ell(I^1)} \langle g, h_{I^1}\rangle\big) \Big\rangle_Q^{\langle \nu \rangle_{I^1}}\Big)^2 
\langle \nu \rangle_{I^1} (Q) \\
&\lesssim \int_{\R^2} \big(M^{2,3}_{\ell(I^1)} \langle f \rangle_{I^1}\big)^2\big(M^{2,3}_{\ell(I^1)} \langle g, h_{I^1}\rangle\big)^2 
\langle \nu \rangle_{I^1}, 
\end{split}
\end{equation*}
where in the last step we used the fact that $\langle \nu \rangle_{I^1} \in A_{\infty, \ell(I^1)}(\R^2)$ and the Carleson embedding theorem.

Since the last estimate holds uniformly for every $I^2_0 \times I^3_0 \in \calD^{2,3}_{\ell(I^1)}$, we get that
\begin{equation*}
\begin{split}
\sum_{I^1\in \calD^1}\sum_{I^2\times I^3\in \calD^{2,3}_{\ell(I^1)}}& \Big| \Big\langle f, \frac{1_{I^1}}{|I^1|}\otimes h_{I^2\times I^3}\Big\rangle\Big|^2 \Big| \Big \langle g, h_{I^1}\otimes \frac{1_{I^2\times I^3}}{|I^2\times I^3|}\Big \rangle \Big|^2 \langle \nu \rangle_{I} \\
& \lesssim \sum_{I^1\in \calD^1}
\int_{\R^2} \big(M^{2,3}_{\ell(I^1)} \langle f \rangle_{I^1}\big)^2\big(M^{2,3}_{\ell(I^1)} \langle g, h_{I^1}\rangle\big)^2 
\langle \nu \rangle_{I^1} \\
&  \le \sum_{I^1\in \calD^1}
\int_{\R^3} \big(M^{2,3}_{\ell(I^1)} \langle f \rangle_{I^1}\big)^2\big(M^{2,3}_{\ell(I^1)} \langle |\Delta_{I^1}g|\rangle_{I^1}\big)^2 
1_{I^1} \nu \\ 
&\le \int_{\R^2} [M_Z f]^2 \sum_{I^1\in \calD^1} M_Z(\Delta_{I^1} g)^2 \nu.
\end{split}
\end{equation*}
Thus, \eqref{H:eq1} is proved.

\emph{Case 2: little BMO paraproducts ($i=3, j=1,2$ or $i=1,2,j=3$).}
Actually, now we only have ``trivial'' type cases with different twist. Symmetries $a_{1,3}$ and $a_{3,1}$ are similar as well as $a_{2,3}$ and $a_{3,2}.$ Let us choose for example $\Pi = a_{1,3}$ first. By Corollary \ref{cor:H1little} we have
\begin{align*}
|\langle \Pi(b, f), g\rangle| &\lesssim \Big\| \Big( \sum_{I^1 \in \calD^1} \Big( \sum_{I^{2,3}\in \calD^{2,3}_{\ell(I^1)}}|\langle f, h_{I,Z} \rangle||\langle g,h_{I^1}h_{I^1}\otimes h_{I^{2,3}} \rangle|\frac{1_{I^{2,3}}}{ |I^{2,3}|}\Big)^2 \frac{1_{I^1}}{|I^1|} \Big)^{\frac{1}{2}} \Big\|_{L^1}.
\end{align*}
Now we again can use similar sparse method as above and for fixed $I^1$ prove 
\begin{align*}
\int \sum_{I^{2,3}\in \calD_{\ell(I^1)}}|\langle f, h_{I,Z} \rangle||\langle g,h_{I^1}h_{I^1}\otimes h_{I^{2,3}} \rangle|\frac{1_{I^{2,3}}}{ |I^{2,3}|} \ave{\nu}_{I^1} \\\lesssim \int M_{\ell(I^1)}^{2,3} (\ave{|\Delta_{I^1} f|}_{I^1}) M_{\ell(I^1)}^{2,3} \ave{|g|}_{I^1}  1_{I^1}\nu.
\end{align*}
The above estimate together with vector-valued version of Theorem \ref{thm:ainftyextra} (proven in \cite{CUMP} for general Muckenhoupt basis) yields
\begin{align*}
&\Big\| \Big( \sum_{I^1 \in \calD^1} \Big( \sum_{I^{2,3}\in \calD_{\ell(I^1)}}|\langle f, h_{I,Z} \rangle||\langle g,h_{I^1}h_{I^1}\otimes h_{I^{2,3}} \rangle|\frac{1_{I^{2,3}}}{ |I^{2,3}|}\Big)^2 \frac{1_{I^1}}{|I^1|} \Big)^{\frac{1}{2}} \Big\|_{L^1} \\
&\lesssim \Big\| \Big( \sum_{I^1 \in \calD^1}  M_{Z} (\Delta_{I^1} f)^2   \frac{1_{I^1}}{|I^1|} \Big)^{\frac{1}{2}} M_Z g \Big\|_{L^1} \\
&\leq  \Big\| \Big(  \sum_{I^1\in \calD^1} M_{Z}(\Delta_{I^1}f)^2 \Big)^{1/2} \Big\|_{L^{p}} \|M_Z g\|_{L^{p'}} \lesssim \| f\|_{L^p} \|g\|_{L^{p'}}.
\end{align*}

Moving to the symmetry $\Pi = a_{3,2}$ we first get
\begin{align*}
&|\langle \Pi(b, f), g\rangle| \\
&= \Big|\sum_{I \in \calD_Z} \ave{\ave{b}_{I^1},h_{I^{2,3}}} \Ave{f, h_{I^1} \otimes \frac{1_{I^{2,3}}}{|I^{2,3}|}} \ave{g,h_{I,Z}} \Big|\\
&\lesssim\|b\|_{\bmo_Z}  \Big\| \sum_{I^1 \in \calD^1}\Big(   \sum_{I^{2,3}\in \calD_{\ell(I^1)}}|\langle f,h_{I^1} \otimes \frac{1_{I^{2,3}}}{ |I^{2,3}|} \rangle|^2|\langle g,h_{I,Z}\rangle|^2\frac{1_{I^{2,3}}}{ |I^{2,3}|} \Big)^{\frac{1}{2}}\frac{1_{I^1}}{|I^1|}\Big\|_{L^1},
\end{align*}
where we use the other estimate in Corollary \ref{cor:H1little}.
Like above, we continue as follows
\begin{align*}
&\Big\| \sum_{I^1 \in \calD^1}\Big(   \sum_{I^{2,3}\in \calD_{\ell(I^1)}}|\langle f,h_{I^1} \otimes \frac{1_{I^{2,3}}}{ |I^{2,3}|} \rangle|^2|\langle g,h_{I,Z}\rangle|^2\frac{1_{I^{2,3}}}{ |I^{2,3}|}  \Big)^{\frac{1}{2}}\frac{1_{I^1}}{|I^1|}\Big\|_{L^1} \\
&\lesssim \Big\| \sum_{I^1 \in \calD^1} M_{\ell(I^1)}^{2,3} \langle |\Delta_{I^1}f| \rangle_{I^1} M_{\ell(I^1)}^{2,3}\langle |\Delta_{I^1} g|\rangle_{I^1}1_{I^1}\Big\|_{L^1}\\
&\le \Big\| \Big(  \sum_{I^1\in \calD^1} M_{Z}(\Delta_{I^1}f)^2 \Big)^{1/2} \Big\|_{L^{p}} \Big\| \Big(  \sum_{I^1\in \calD^1} M_{Z}(\Delta_{I^1}g)^2 \Big)^{1/2} \Big\|_{L^{p'}} \\
&\lesssim \| f\|_{L^p} \|g\|_{L^{p'}}.
\end{align*}

\end{proof}

In above proof we needed the $A_\infty$ extrapolation with Zygmund $A_\infty$ weights.
In fact, we give a very simple proof of $A_{\infty}$ extrapolation \cite{CUMP} in general.
\begin{thm}\label{thm:ainftyextra}
Let $(f,g)$ be a pair of non-negative functions. Assume that there is some $0<p_0<\infty$ such that for all $w\in A_{\infty, Z}$ there holds 
\[
\int f^{p_0} w \le C([w]_{A_{\infty, Z}})\int g^{p_0} w,
\]
where $C$ is an increasing function.
Then for all $0<p<\infty$ and all  $w\in A_{\infty, Z}$ there holds 
\[
\int f^{p} w \le C([w]_{A_{\infty, Z}})\int g^{p} w.
\]
\end{thm}
\begin{proof}
We have for all $1<r<\infty$ and all $w \in A_{r,Z}$ that
\[
\int (f^{p_0/r})^r w \le C([w]_{A_{r, Z}})\int (g^{p_0/r})^r w.
\]
Thus, by the classical extrapolation with $A_{p,Z}$ weights we have 
\begin{equation}\label{E:eq9}
\int (f^{p_0/r})^s w \le C([w]_{A_{s, Z}})\int (g^{p_0/r})^s w
\end{equation}
for all $1<s<\infty$ and $w\in A_{s,Z}$.  

Finally, let $0<p<\infty$ and $w\in A_{\infty, Z}$. Then, there exists some $1<s_0<\infty$ such that 
$w\in A_{s_0,Z}$. Choose some $1<r<\infty$ and $s_0\le s<\infty$ such that 
\[
sp_0/r=p.
\]
For example, we can take 
\[
s=\frac{s_0p}{p_0}\Big(\frac{p_0}p+1\Big)=s_0\Big(\frac{p}{p_0}+1\Big),\quad r=s_0\Big(\frac{p_0}p+1\Big).
\] 
Since $A_{s_0,Z} \subset A_{s,Z}$, we can use \eqref{E:eq9} with the exponents $s$ and $r$ to get the claim.
\end{proof}

\subsection{Zygmund shift commutators}

Let $k = (k^1, k^2)$, $k^i \in \{0,1,2,\ldots\}$, be fixed. A Zygmund shift $Q = Q_k$ 
  of complexity $k$, see \cite{HLMV-ZYG}, has the form
  \begin{align*}
    &\ave{Q_{k}f,g} \\
    &= \sum_{K \in \calD_{2^{-k^1-k^2+k^3}}}\sum_{\substack{I,J \in \calD_{Z}\\ I^{(k)} = K = J^{(k)}}} a_{IJK} \ave{f, h_{I^1} \otimes H_{I^{2,3},J^{2,3}}} \ave{g, H_{I^1,J^1} \otimes h_{J^{2,3}}} 
  \end{align*}
  or
    \begin{align*}
    &\ave{Q_{k}f,g} \\
    &= \sum_{K \in \calD_{2^{-k^1-k^2+k^3}}}\sum_{\substack{I,J \in \calD_{Z}\\ I^{(k)} = K = J^{(k)}}} a_{IJK} \ave{f, h_{I^1} \otimes h_{I^{2,3}}} \ave{g, H_{I^1,J^1} \otimes H_{I^{2,3},J^{2,3}} },
  \end{align*}
  where $H_{I,J}$ 
  \begin{enumerate}
  \item is supported on $I\cup J$ and constant on children:
\[
H_{I,J} = \sum_{L \in \ch(I)\cup \ch(J)} b_L 1_L
\]
  \item is $L^2$ normalized: $|H_{I,J}|\leq |I|^{-\frac{1}{2}},$ and
  \item has zero average: $\int H_{I,J} = 0.$ 
  \end{enumerate}
  We will be focusing on the mixed type form since it is the most interesting one. Usually the other type is much easier and the method is easily recovered from this case.

  \begin{prop}
    Let $Q_k$ be a Zygmund shift of complexity $k=(k^1,k^2,k^3).$ Let $1<p<\infty$ and $b \in \bmo_Z.$ Then we have
    \[\|[b,Q_k] f\|_{L^p} \lesssim \max(k^1,k^2,k^3) (|k|+1)^2 \|b\|_{\bmo_Z} \|f\|_{L^p}. \]
  \end{prop}
  
  \begin{proof}
  We consider the  commutator
  $
  [b,Q_k]f\colon bQ_k f - Q_k(bf) 
  $
  that in the dual form equals to 
  \begin{align*}
  \sum_{K \in \calD_{2^{-k^1-k^2+k^3}}}\sum_{\substack{I,J \in \calD_{Z}\\ I^{(k)} = K = J^{(k)}}} a_{IJK} \Big[ \ave{bf, h_{I^1} \otimes H_{I^{2,3},J^{2,3}}} \ave{g, H_{I^1,J^1} \otimes h_{J^{2,3}}} \\
  -\ave{f, h_{I^1} \otimes H_{I^{2,3},J^{2,3}}} \ave{bg, H_{I^1,J^1} \otimes h_{J^{2,3}}}\Big].
  \end{align*}
  Now, expanding both $bf$ and $bg$ with the expansion \eqref{eq:dec} we get the terms
  \begin{align*}
\ave{Q_k(a_{i,j}(b,f)), g} \qquad \text{ and }\qquad  \ave{Q_kf, a_{i,j}(b,g)} 
  \end{align*}
  whenever $(i,j)\neq (3,3).$ These terms are directly bounded separately, in particular, we have $Q_k\colon L^{p} \to L^{p}$ and $a_{i,j}\colon L^{p} \to L^p.$
  Hence, we are left with bounding
    \begin{align*}
  &\sum_{K \in \calD_{\lambda}}\sum_{\substack{I,J \in \calD_{Z}\\ I^{(k)} = K = J^{(k)}}} a_{IJK} \Big[\sum_{L \in \calD_Z} \ave{b}_{L} \ave{\Delta_{L,Z}f, h_{I^1} \otimes H_{I^{2,3},J^{2,3}}} \ave{g, H_{I^1,J^1} \otimes h_{J^{2,3}}} \\
    &\qquad-\sum_{L \in \calD_Z}\ave{b}_L\ave{f, h_{I^1} \otimes H_{I^{2,3},J^{2,3}}} \ave{\Delta_{L,Z} g, H_{I^1,J^1} \otimes h_{J^{2,3}}}\Big]\\
  &=   \sum_{K \in \calD_{\lambda}}\sum_{\substack{I,J \in \calD_{Z}\\ I^{(k)} = K = J^{(k)}}} a_{IJK} \\
  &\qquad\times\Big[ \sum_{\substack{L \in \calD_Z \\ \ell(L^1)=2^{-k^1}\ell(K^1) \\ \ell(K^2) \le 2^{k^2}\ell(L^2) \le 2^{\max(k^2,k^3)} \ell(K^2) }}  \ave{b}_{L} \ave{\Delta_{L,Z}f, h_{I^1} \otimes H_{I^{2,3},J^{2,3}}} \ave{g, H_{I^1,J^1} \otimes h_{J^{2,3}}} \\
  &\qquad-  \sum_{\substack{Q \in \calD_Z \\ Q^1 \subset K^1, \ \ell(Q^1) \ge \ell(I^1) \\ 2^{-k^1}\ell(K^2) \le 2^{k^2}\ell(Q^2) \le \ell(K^2)}} \ave{b}_Q\ave{f, h_{I^1} \otimes H_{I^{2,3},J^{2,3}}} \ave{\Delta_{Q,Z} g, H_{I^1,J^1} \otimes h_{J^{2,3}}}\Big],
  \end{align*}
  where we have abbreviated $2^{-k^1-k^2+k^3}$ by $\lambda.$
  Now, we write
  $$
  \ave{f, h_{I^1} \otimes H_{I^{2,3},J^{2,3}}}= \sum_{\substack{L \in \calD_Z \\ \ell(L^1)=2^{-k^1}\ell(K^1) \\ \ell(K^2) \le 2^{k^2}\ell(L^2) \le 2^{\max(k^2,k^3)} \ell(K^2) }} 
\langle \Delta_{L,Z} f, h_{I^1} \otimes H_{I^{2,3}, J^{2,3}} \rangle
  $$
  and 
  $$
\ave{g, H_{I^1,J^1} \otimes h_{J^{2,3}}}  = \sum_{\substack{Q \in \calD_Z \\ Q^1 \subset K^1, \ \ell(Q^1) \ge \ell(I^1) \\ 2^{-k^1}\ell(K^2) \le 2^{k^2}\ell(Q^2) \le \ell(K^2)}}
\ave{\Delta_{Q,Z} g, H_{I^1,J^1} \otimes h_{J^{2,3}}}
  $$
  for the unexpanded terms. Thus, we end up with
  \begin{align*}
  &\sum_{K \in \calD_{\lambda}}\sum_{\substack{I,J \in \calD_{Z}\\ I^{(k)} = K = J^{(k)}}} a_{IJK}\sum_{\substack{L \in \calD_Z \\ \ell(L^1)=2^{-k^1}\ell(K^1) \\ \ell(K^2) \le 2^{k^2}\ell(L^2) \le 2^{\max(k^2,k^3)} \ell(K^2) }}\sum_{\substack{Q \in \calD_Z \\ Q^1 \subset K^1, \ \ell(Q^1) \ge \ell(I^1) \\ 2^{-k^1}\ell(K^2) \le2^{k^2} \ell(Q^2) \le \ell(K^2)}} \\
  &\qquad\times\Big[   (\ave{b}_{L}-\ave{b}_Q) \ave{\Delta_{L,Z}f, h_{I^1} \otimes H_{I^{2,3},J^{2,3}}}  \ave{\Delta_{Q,Z} g, H_{I^1,J^1} \otimes h_{J^{2,3}}}\Big].
  \end{align*}

We write explicitly the complexity levels for $Q$ and $L.$
That is, in the above summations we have $(L^2)^{(l^2)} = (K^2)^{(\max(0,k^3-k^2))}$ for some $l^2 \in \{0, \ldots ,\max(k^2,k^3)\},$ $(Q^1)^{(q_1)} = K^1,$ for some $q_1 \in \{0,\ldots ,k^1\},$ and $(Q^2)^{(q_2)} = K^2$ for some $q_2 \in \{k^2, \ldots, k^2 + k^1\}.$ We get 
  \begin{align*}
  &\sum_{K \in \calD_{\lambda}}\sum_{\substack{I,J \in \calD_{Z}\\ I^{(k)} = K = J^{(k)}}} a_{IJK}\sum_{l^2 = 0}^{\max(k^2,k^3)} \sum_{\substack{q_1 \in \{0,\ldots ,k^1\} \\ q_2 \in \{k^2, \ldots, k^2 + k^1\}}} \sum_{\substack{L \in \calD_Z \\ \ell(L^1)=2^{-k^1}\ell(K^1) \\ (L^2)^{(l^2)} = (K^2)^{(\max(0,k^3-k^2))}}} \sum_{\substack{Q \in \calD_Z \\ (Q^1)^{(q_1)} = K^1\\ (Q^2)^{(q_2)} = K^2}} \\
  &\qquad\times\Big[   (\ave{b}_{L}-\ave{b}_Q) \ave{\Delta_{L,Z}f, h_{I^1} \otimes H_{I^{2,3},J^{2,3}}}  \ave{\Delta_{Q,Z} g, H_{I^1,J^1} \otimes h_{J^{2,3}}}\Big].
  \end{align*}
Here we need to notice that $R = R^1 \times R^2 \times R^3 \supset K,L,Q,$ where \[R = K^{(k^1,\max(0,k^3-k^2), k^1+\max(k^2-k^3,0))} \text{ and } R \in \calD_Z.\] This is a common ``Zygmund ancestor'' for all of these rectangles.

Let us expand in the difference 
  \(\ave{b}_{L}-\ave{b}_Q\)
 in the following way
\begin{align*}
  \ave{b}_L &= \ave{b}_L -  \ave{b}_{L^{(0,1,1)}}  \\
            &  + \ave{b}_{L^{(0,1,1)}}  - \ave{b}_{L^{(0,2,2)}} \\
            &\qquad\vdots\\
            &+  \ave{b}_{L^{(0, l^2 - 1, l^2-1)}} - \ave{b}_{L^{(0,l^2,l^2)}} + \ave{b}_{L^{(0,l^2,l^2)}}\\
            &= \sum_{r^2 = 0}^{l^2-1}  \big(\ave{b}_{L^{(0,r^2,r^2)}}  - \ave{b}_{L^{(0,r^2+1,r^2+1)}}\big) + \ave{b}_{L^{(0,l^2,l^2)}}.
\end{align*}
Notice that since $\ell(L^1)\ell(L^2) = \ell(L^3),$ we have $\ell(L^1) \ell((L^2)^{(r^2)}) = \ell((L^3)^{(r^2)}),$ i.e.\  rectangles $(L^2)^{(r^2)} \times (L^3)^{(r^2)} \in \calD_{\ell(L^1)}$ which is desirable since we want to use the characterization $(2)$ in Lemma~\ref{lem:equiv}.
We continue with the last term
\begin{align*}
  \ave{b}_{L^{(0,l^2,l^2)}} &= \ave{b}_{L^{(0,l^2,l^2)}} - \ave{b}_{L^{(1,l^2,1 + l^2)}} \\
  &+ \ave{b}_{L^{(1,l^2,1 + l^2)}} - \ave{b}_{L^{(2,l^2,2 + l^2)}} \\
  &\qquad \vdots \\
  & \ave{b}_{L^{(k^1-1,l^2,k^1-1 + l^2)}} - \ave{b}_{L^{(k^1, l^2,k^1 + l^2)}} + \ave{b}_{G} \\
  &= \sum_{r^1 = 0}^{k^1 - 1} \big( \ave{b}_{L^{(r^1, l^2,r^1 + l^2)}} - \ave{b}_{L^{(r^1 +1, l^2,r^1 + 1 + l^2)}}\big) + \ave{b}_{R}. 
\end{align*}
Recall that $(L^2)^{(l^2)} = (K^2)^{(\max(0,k^3-k^2))} =: R^2$ and observe that since $\ell((L^3)^{(k^1 + l^2)}) = \ell((L^2)^{(l^2)})\ell((L^1)^{(k^1)}) = \ell(R^2)\ell(K^1)$ we get $(L^3)^{(k^1 + l^2)} = R^3$.
Thus, we end up with a sum of terms of the forms
\begin{align}\label{eq:bdiff}
  \ave{b}_{L^{(0,r^2,r^2)}}  - \ave{b}_{L^{(0,r^2+1,r^2+1)}}\quad
   \text{ and }\quad  
   \ave{b}_{L^{(r^1, l^2,r^1 + l^2)}} - \ave{b}_{L^{(r^1 +1, l^2,r^1 + 1 + l^2)}},
\end{align}
and we have for fixed $r^1$ and $r^2$
\[|\eqref{eq:bdiff}| \lesssim \|b\|_{\bmo_Z}\]
by Lemma~\ref{lem:equiv}.

By the same argument as above we get 
\begin{align*}
  \ave{b}_Q =& \sum_{\rho^2 = 0}^{\max(0,k^3-k^2) + q^2 - 1} \ave{b}_{Q^{(0,\rho^2,\rho^2)}} - \ave{b}_{Q^{(0,\rho^2 +1,\rho^2+1)}}\\
  &+ \sum_{\rho^1 = 0}^{q^1} \ave{b}_{Q^{(\rho^1,\wt q^2 ,\rho^1 +\wt q^2)}} - \ave{b}_{Q^{(\rho^1 + 1,\wt q^2 ,\rho^1 +1 +\wt q^2)}}\\
  &+ \ave{b}_{R},
\end{align*} 
where $\wt q^2 = \max(0,k^3-k^2) + q^2,$  \[(Q^2)^{(\wt q^2)} = (K^2)^{(\max(0,k^3-k^2))} \qquad \text{and} \qquad (Q^3)^{(q^1 + \wt q^2)} = (K^3)^{(k^1 + \max(k^2-k^3,0))}.\] Notice that the last term corresponds to the last term in the previous expansion, and hence, their difference equals to zero.
Again, here we have 
\[ |\ave{b}_{Q^{(0,\rho^2,\rho^2)}} - \ave{b}_{Q^{(0,\rho^2 +1,\rho^2+1)}} + \ave{b}_{Q^{(\rho^1,\wt q^2 ,\rho^1 +\wt q^2)}} - \ave{b}_{Q^{(\rho^1 + 1,\wt q^2 ,\rho^1 +1 +\wt q^2)}}| \lesssim \|b\|_{\bmo_Z}\]
for fixed $\rho^1$ and $\rho^2.$

Now, we can split the  commutator into the two terms
$$
\calW_{K,k}^b f= 1_{K} \sum_{\substack{L \in \calD_Z \\ \ell(L^1)=2^{-k^1}\ell(K^1) \\ \ell(K^2) \le 2^{k^2}\ell(L^2) \le 2^{\max(k^2,k^3)} \ell(K^2) }} b_{L,K} \Delta_{L,Z}f,
$$
where
$$
|b_{L,K}| \lesssim \max(k^1, k^2,k^3) \|b\|_{\bmo_Z},
$$
and 
$$
\calV_{K,k}^{b} g= \sum_{\substack{Q \in \calD_Z \\ Q^1 \subset K^1, \ \ell(Q^1) \ge \ell(I^1) \\ 2^{-k^1}\ell(K^2) \le 2^{k^2}\ell(Q^2) \le \ell(K^2)}} b_{Q,K}
\Delta_{Q,Z} g,
$$
where
$$
|b_{Q,K}| \lesssim \max(k^1, k^2,k^3) \|b\|_{\bmo_Z}.
$$

Thus, the last term of the commutator is the sum of
  \begin{align*}   \sum_{K \in \calD_{\lambda}}\sum_{\substack{I,J \in \calD_{Z}\\ I^{(k)} = K = J^{(k)}}} a_{IJK}  \ave{\calW_{K,k}^{b}f, h_{I^1} \otimes H_{I^{2,3},J^{2,3}}}  \ave{\calV_{K,k}g, H_{I^1,J^1} \otimes h_{J^{2,3}}} 
  \end{align*}
  and
    \begin{align*}   \sum_{K \in \calD_{\lambda}}\sum_{\substack{I,J \in \calD_{Z}\\ I^{(k)} = K = J^{(k)}}} a_{IJK}  \ave{\calW_{K,k}f, h_{I^1} \otimes H_{I^{2,3},J^{2,3}}}  \ave{\calV_{K,k}^{b}g, H_{I^1,J^1} \otimes h_{J^{2,3}}}.
  \end{align*}
  The boundedness follows via standard methods (adapt proofs of~\cite{HLMV-ZYG}*{Theorem 6.2 and Lemma 5.20}.)
\end{proof}

\appendix
\section{Bilinear Fefferman-Pipher multipliers}\label{app:multipliers}
In this section we consider bilinear variants of multipliers studied by Fefferman-Pipher~\cite{FEPI}. These
considerations motivate the kernel estimates in Section~\ref{sec:ZSIO}. After the presented calculations,
the reader can easily check how everything fits with Section~\ref{sec:ZSIO}. In fact, we will see that the bilinear Fefferman-Pipher multipliers produce  kernels which satisfy the the kernel estimates in Section~\ref{sec:ZSIO} with 
\[
\theta=2, \quad \alpha_1=1, \quad \alpha_{2,3}=1, \quad 
\] and an extra logarithm factor. 
In the partial kernel estimates $\wt \theta=1$ and there is also a harmless logarithm factor.

We consider the following multi-parameter dilation on $\R^6$ -- define
 \[\rho_{s,t}(x,y) = (sx_1,tx_2,stx_3,sy_1,ty_2,sty_3), \qquad s,t>0,\]
and set
\[A^1 := \{(\xi,\eta) \in \R^6\colon \tfrac 12 < |(\xi_1, \eta_1)| \le 1, \tfrac 12 < |(\xi_2,\xi_3,\eta_2,\eta_3)| \le 1 \}.\]
In this section we consider the parameter groups $\{1\}$ and $\{2,3\}$ only. The grouping $\{\{2\},\{1,3\}\}$ is similar, for example, we would set
\[A^2 := \{(\xi,\eta) \in \R^6\colon \tfrac 12 < |(\xi_2, \eta_2)| \le 1, \tfrac 12 < |(\xi_1,\xi_3,\eta_1,\eta_3)| \le 1 \}.\]

For Schwartz functions $f_1,f_2$ we define the bilinear multiplier operator
\[T_{m,1} (f_1,f_2)(x) = \int_{\R^3} \int_{\R^3} m(\xi,\eta) \widehat{f_1}(\xi)\widehat{f_2}(\eta) e^{2\pi i x\cdot (\xi + \eta) }\ud \xi \ud \eta,\]
where the symbol $m \in C^N$ is assumed to satisfy
\[\Norm{m}{\mathcal M^1_{Z}}:=
\sup_{\substack{\norm{\alpha}{\infty}\leq N\\ |\beta|_\infty \le N}}\sup_{s,t>0} \sup_{(\xi,\eta)\in A^1}\abs{\partial_\xi^\alpha \partial_{\eta}^\beta(m\circ\rho_{s,t})(\xi, \eta)}<\infty.\]
Thus, if $(\xi,\eta) \in A^1,$ then by definition
\begin{align}
  \abs{ (\partial_\xi^\alpha \partial_{\eta}^\beta m)(s\xi_1, t\xi_2, st\xi_3,s\eta_1,t\eta_2,st\eta_3)}&\le  \Norm{m}{\mathcal M^1_{Z}} s^{-\alpha_1-\beta_1} t^{-\alpha_2- \beta_2} (st)^{-\alpha_3-\beta_3}\label{eq:diffMulti}\\
  &=\Norm{m}{\mathcal M^1_{Z}} s^{-(\alpha_1+\beta_1)+(\alpha_2 +\beta_2)} (st)^{-(\alpha_2+\beta_2)-(\alpha_3+\beta_3)}.\nonumber
\end{align}

Now, for $(\zeta_1, \sigma_1)\neq 0$ and $(\zeta_2,\zeta_3,\sigma_2,\sigma_3)\neq 0$ denote
 \begin{align*}
  s= |(\zeta_1,\sigma_1)|,\qquad st= |(s \zeta_2,\zeta_3, s\sigma_2,\sigma_3)|,\\ (\xi_1, \xi_2, \xi_3) = \Big(\frac{\zeta_1}{s},\frac{\zeta_2}{t},\frac{\zeta_3}{st}\Big),\qquad (\eta_1,\eta_2,\eta_3)=\Big(\frac{\sigma_1}{s},\frac{\sigma_2}{t},\frac{\sigma_3}{st}\Big).
\end{align*}
 Thus, $(\xi, \eta)\in A^1$ and 
 \begin{align}\label{eq:blinmulti}
 |\partial^\alpha_{\zeta}\partial^\beta_{\sigma} m(\zeta,\sigma)| &\lesssim \|m\|_{\calM_Z^1} (|\zeta_1|+|\sigma_1|)^{-(\alpha_1+\beta_1)+(\alpha_2 +\beta_2)}\\
 &\hspace{1.5cm}\times \big(|((|\zeta_1|+|\sigma_1|)\zeta_2,\zeta_3)| + |((|\zeta_1|+|\sigma_1|)\sigma_2,\sigma_3)|\big)^{-(\alpha_2+\beta_2)-(\alpha_3+\beta_3)}.\nonumber
 \end{align}

 We write, with two standard partitions of unity $\phi_1$ on $\R^2\setminus \{0\}$ and $\phi_{2,3}$ on $\R^4 \setminus \{0\},$ that
 \[1 = \sum_{j,k\in \Z} \phi_1(2^{-j}\xi_1, 2^{-j}\eta_1) \phi_{2,3} (2^{-k}\xi_2,2^{-j-k}\xi_3,2^{-k}\eta_2,2^{-j-k}\eta_3).\]
 Via this identity we obtain
 \begin{align*}
  m &= \sum_{j,k} (\phi_1 \otimes \phi_{2,3} \circ \rho_{2^{-j},2^{-k}})\cdot m\\
&= \sum_{j,k} (\phi_1 \otimes \phi_{2,3} \cdot (m \circ \rho_{2^{j},2^{k}}))\circ \rho_{2^{-j},2^{-k}} 
=: m_{j,k}.
 \end{align*}
Since $\phi_1$ and $\phi_{2,3}$ are supported in $\bar B(0, 2)\setminus B(0,\frac 12)$ in $\R^2$ and $\R^4$, respectively, we know that
\begin{align*}
\supp m_{j,k}& \subset \\
&\rho_{2^j, 2^k}\Big\{(\xi, \eta): (\xi_1, \eta_1)\in \bar B_{\R^2}(0, 2)\setminus B_{\R^2}(0,\tfrac 12), (\xi_{2,3}, \eta_{2,3})\in   \bar B_{\R^4}(0, 2)\setminus B_{\R^4}(0,\tfrac 12)\Big\}.
\end{align*}
Using this we get
\[\|\partial^\alpha \partial^\beta m_{j,k}\|_{L^\infty} \lesssim 2^{-(j,k,j+k)\cdot (\alpha+\beta)} \quad \text{and} \quad \|\partial^\alpha \partial^\beta m_{j,k}\|_{L^1} \lesssim 2^{(j,k,j+k)\cdot (\mathbf{2}-(\alpha+\beta))},\] 
where $\mathbf{2} = (2,2,2).$

Let $K_{j,k}(y,z)=\check{m}_{j,k}$ and $K(y,z)=\sum_{j,k}K_{j,k}(y, z)$ -- then $K(x-y,x-z)$ is the corresponding kernel. 
Using similar analysis as in \cite{HLMV-ZYG} we have
\begin{align*}
  \|y^\alpha z^{\tilde \alpha}\partial^\beta_y \partial^\gamma_z K_{j,k}\|_{L^\infty}
  &\lesssim \|  \partial^{\alpha}_\xi \partial_\eta^{\tilde \alpha} (\xi^\beta \eta^\gamma m_{j,k})\|_{L^1} \\
  &\le \sum_{\substack{l \le \alpha \\ \tilde l \le \tilde \alpha}} \binom{\alpha}{l}\binom{\tilde \alpha}{\tilde l} \|\partial^{l} (\xi^\beta) \partial^{\tilde l}(\eta^\gamma) \cdot \partial^{\alpha - l}\partial^{\tilde \alpha - \tilde l}m_{j,k})\|_{L^1} \\
  &\lesssim  2^{(j,k,j+k)\cdot( \mathbf{2}+(\beta+\gamma) - (\alpha+\tilde \alpha))}  
  \end{align*}
for multi-indices $\alpha,\tilde \alpha,\beta,\gamma.$ 
Hence, we get 
\[
|y^{\beta + \mathbf{1}}z^{\gamma + \mathbf{1}} \partial^\beta_y \partial^\gamma_z K_{j,k}(y,z)| \lesssim  2^{(j,k,j+k)\cdot( \mathbf{2}+(\beta+\gamma) - (\alpha+\tilde \alpha))} |y^{ \beta + \mathbf{1}- \alpha}|\cdot |z^{ \gamma + \mathbf{1}- \tilde\alpha}|.
\]
Taking $\alpha_i,\tilde\alpha_i \in \{0,N\}$
we obtain
\begin{align*}
  &|y^{\beta + \mathbf{1}}z^{\gamma + \mathbf{1}} \partial^\beta_y \partial^\gamma_z K(y,z)| \\
  &\lesssim \sum_{j} \min\{(2^j  |y_1|)^{  \beta_1 + 1},(2^j |y_1|)^{  \beta_1 + 1 -N}\} \min\{(2^j  |z_1|)^{  \gamma_1 + 1},(2^j |z_1|)^{  \gamma_1 + 1 -N}\} \\
  &\quad\times\sum_k \min\{(2^k  |y_2|)^{  \beta_2 + 1},(2^k |y_2|)^{  \beta_2 + 1 -N}\} \min\{(2^k  |z_2|)^{  \gamma_2 + 1},(2^k |z_2|)^{  \gamma_2 + 1 -N}\} \\
  &\qquad\times \min\{(2^{j+k}  |y_3|)^{  \beta_3 + 1},(2^{j+k} |y_3|)^{  \beta_3 + 1 -N}\} \min\{(2^{j+k}  |z_3|)^{  \gamma_3 + 1},(2^{j+k} |z_3|)^{  \gamma_3 + 1 -N}\} . 
\end{align*}

We can estimate the inner sum either by
\begin{align*}
  &\sum_{k \colon 2^k < 1/(|y_2| + |z_2|)} (2^k  |y_2|)^{  \beta_2 + 1}(2^k  |z_2|)^{  \gamma_2 + 1}(2^{j+k}  |y_3|)^{  \beta_3 + 1} (2^{j+k}  |z_3|)^{  \gamma_3 + 1}\\
  &+ \sum_{k \colon 2^k \ge 1/(|y_2| + |z_2|)\ge 1/(2|y_2|)} (2^k  |y_2|)^{  \beta_2 + 1-N}(2^k  |z_2|)^{  \gamma_2 + 1}(2^{j+k}  |y_3|)^{  \beta_3 + 1} (2^{j+k}  |z_3|)^{  \gamma_3 + 1} \\
  &+ \sum_{k \colon 2^k \ge 1/(|y_2| + |z_2|)> 1/(2|z_2|)} (2^k  |y_2|)^{  \beta_2 + 1}(2^k  |z_2|)^{  \gamma_2 + 1-N}(2^{j+k}  |y_3|)^{  \beta_3 + 1} (2^{j+k}  |z_3|)^{  \gamma_3 + 1} \\
  &\lesssim  \frac {|y_2|^{\beta_2+1}}{(|y_2|+|z_2|)^{\beta_2+1}} \cdot \frac {|z_2|^{\gamma_2+1}}{(|y_2|+|z_2|)^{\gamma_2+1}}\cdot   \frac {(2^j|y_3|)^{\beta_3+1}}{(|y_2|+|z_2|)^{\beta_3+1}}\cdot   \frac {(2^j|z_3|)^{\gamma_3+1}}{(|y_2|+|z_2|)^{\gamma_3+1}}=:I_1
  \end{align*}
  or by 
  \begin{align*}
  &\sum_{k \colon 2^k < 2^{-j}/(|y_3| + |z_3|)} (2^k  |y_2|)^{  \beta_2 + 1}(2^k  |z_2|)^{  \gamma_2 + 1}(2^{j+k}  |y_3|)^{  \beta_3 + 1} (2^{j+k}  |z_3|)^{  \gamma_3 + 1}\\
  &+ \sum_{k \colon 2^k \ge  2^{-j}/(|y_3| + |z_3|)\ge 2^{-j}/(2|y_3| )} (2^k  |y_2|)^{  \beta_2 + 1}(2^k  |z_2|)^{  \gamma_2 + 1}(2^{j+k}  |y_3|)^{  \beta_3 + 1-N} (2^{j+k}  |z_3|)^{  \gamma_3 + 1}\\
  & +\sum_{k \colon 2^k \ge  2^{-j}/(|y_3| + |z_3|)> 2^{-j}/(2|z_3| )} (2^k  |y_2|)^{  \beta_2 + 1}(2^k  |z_2|)^{  \gamma_2 + 1}(2^{j+k}  |y_3|)^{  \beta_3 + 1} (2^{j+k}  |z_3|)^{  \gamma_3 + 1-N}\\
  &\lesssim \frac {|y_2|^{\beta_2+1}}{[2^j(|y_3|+|z_3|)]^{\beta_2+1}} \cdot \frac {|z_2|^{\gamma_2+1}}{[2^j(|y_3|+|z_3|)]^{\gamma_2+1}}\cdot   \frac {|y_3|^{\beta_3+1}}{(|y_3|+|z_3|)^{\beta_3+1}}\cdot   \frac {|z_3|^{\gamma_3+1}}{(|y_3|+|z_3|)^{\gamma_3+1}}=:I_2,
  \end{align*}
  in both cases provided that $\beta_2+\beta_3+\gamma_2+\gamma_3<N-4$. 

The outer sum can then be estimated either by 
\begin{align*}
 &\sum_{j \colon 2^j <1/(|y_1| + |z_1|)} (2^j |y_1|)^{\beta_1+1} (2^j |z_1|)^{\gamma_1+1} I_1  \\
 &+ \sum_{j \colon 2^j \ge 1/(|y_1| + |z_1|)\ge 1/{(2|y_1|)}} (2^j |y_1|)^{\beta_1+1-N} (2^j |z_1|)^{\gamma_1+1} I_1\\
 &+ \sum_{j \colon 2^j \ge 1/(|y_1| + |z_1|)>1/{(2|z_1|)}} (2^j |y_1|)^{\beta_1+1} (2^j |z_1|)^{\gamma_1+1-N} I_1\\
 &\lesssim \frac{|y_1|^{\beta_1+1} |z_1|^{\gamma_1+1}}{(|y_1|+|z_1|)^{\beta_1+\gamma_1+2}}\frac{|y_2|^{\beta_2+1} |z_2|^{\gamma_2+1}}{(|y_2|+|z_2|)^{\beta_2+\gamma_2+2}}\frac{|y_3|^{\beta_3+1}|z_3|^{\gamma_3+1}}{[(|y_1|+|z_1|)(|y_2|+|z_2|)]^{\beta_3+\gamma_3+2}}
\end{align*}
or, if $(|y_1|+|z_1|)(|y_2|+|z_2|)\le |y_3|+|z_3|$, by 
\begin{align*}
 &\sum_{j \colon 2^j <(|y_2|+|z_2|)/(|y_3| + |z_3|)} (2^j |y_1|)^{\beta_1+1} (2^j |z_1|)^{\gamma_1+1} I_1\\
 &+ \sum_{j \colon  \frac{|y_2|+|z_2|}{|y_3|+|z_3|}\le 2^j \le \frac 1{|y_1| + |z_1|} } (2^j |y_1|)^{\beta_1+1} (2^j |z_1|)^{\gamma_1+1}  I_2\\
 &+ \sum_{j \colon 2^j > 1/(|y_1| + |z_1|)>1/{(2|z_1|)}} (2^j |y_1|)^{\beta_1+1} (2^j |z_1|)^{\gamma_1+1-N}  I_2\\
 &+  \sum_{j \colon 2^j > 1/(|y_1| + |z_1|)>1/{(2|y_1|)}} (2^j |y_1|)^{\beta_1+1-N} (2^j |z_1|)^{\gamma_1+1}  I_2 =: I+II+III+IV.
\end{align*}
It is straightforward that 
\begin{align*}
I&\sim \frac{|y_1|^{\beta_1+1} |z_1|^{\gamma_1+1}}{(|y_1|+|z_1|)^{\beta_1+\gamma_1+2}}\frac{|y_2|^{\beta_2+1} |z_2|^{\gamma_2+1}}{(|y_2|+|z_2|)^{\beta_2+\gamma_2+2}}\frac{|y_3|^{\beta_3+1}|z_3|^{\gamma_3+1}}{ (|y_3|+|z_3|)^{\beta_3+\gamma_3+2}}\\
&\hspace{3cm}\times \Big(\frac{(|y_1|+|z_1|)(|y_2|+|z_2|)}{|y_3|+|z_3|} \Big)^{\beta_1+\gamma_1+2};\\
III&\sim IV \sim  \frac{|y_1|^{\beta_1+1} |z_1|^{\gamma_1+1}}{(|y_1|+|z_1|)^{\beta_1+\gamma_1+2}}\frac{|y_2|^{\beta_2+1} |z_2|^{\gamma_2+1}}{(|y_2|+|z_2|)^{\beta_2+\gamma_2+2}}\frac{|y_3|^{\beta_3+1}|z_3|^{\gamma_3+1}}{ (|y_3|+|z_3|)^{\beta_3+\gamma_3+2}}\\
&\hspace{3cm}\times \Big(\frac{(|y_1|+|z_1|)(|y_2|+|z_2|)}{|y_3|+|z_3|} \Big)^{\beta_2+\gamma_2+2}.
\end{align*}Lastly, we have 
\begin{align*}
II &\sim \frac{|y_1|^{\beta_1+1} |z_1|^{\gamma_1+1}}{(|y_1|+|z_1|)^{\beta_1+\gamma_1+2}}\frac{|y_2|^{\beta_2+1} |z_2|^{\gamma_2+1}}{(|y_2|+|z_2|)^{\beta_2+\gamma_2+2}}\frac{|y_3|^{\beta_3+1}|z_3|^{\gamma_3+1}}{ (|y_3|+|z_3|)^{\beta_3+\gamma_3+2}}\\
&\hspace{3cm}\times \Big(\frac{(|y_1|+|z_1|)(|y_2|+|z_2|)}{|y_3|+|z_3|} \Big)^{\min\{\beta_1+\gamma_1, \beta_2+\gamma_2\}+2}L_{\beta_1, \beta_2, \gamma_1, \gamma_2}(y,z),
\end{align*}
where 
\[
L_{\beta_1, \beta_2, \gamma_1, \gamma_2}(y,z):=1+\log_+\Big(\frac{|y_3|+|z_3|}{(|y_1|+|z_1|)(|y_2|+|z_2|)} \Big)
\]
when $\beta_1+\gamma_1=\beta_2+\gamma_2$ and $L_{\beta_1, \beta_2, \gamma_1, \gamma_2}(y,z)=1$ otherwise. In conclusion, we get 
\begin{align*}
|  \partial^\beta_y \partial^\gamma_z K(y,z)|&\lesssim \frac 1{[(|y_1|+|z_1|)(|y_2|+|z_2|)+|y_3|+|z_3|]^{\beta_3+\gamma_3+4}}\\
&\qquad \times \frac 1{(|y_1|+|z_1|)^{\beta_1+\gamma_1}(|y_2|+|z_2|)^{\beta_2+\gamma_2} }\\
&\qquad \times \min\Big\{1,\Big(\frac{(|y_1|+|z_1|)(|y_2|+|z_2|)}{|y_3|+|z_3|} \Big)^{\min\{\beta_1+\gamma_1, \beta_2+\gamma_2\}} \Big\}L_{\beta_1, \beta_2, \gamma_1, \gamma_2}(y,z).
\end{align*}

\subsection*{Partial kernel estimates}
Let $m\in \mathcal M_Z^1$. We define truncations of $m$ by setting
\[
m_J:= \sum_{|j|\le J^1, |k|\le J^2}m_{j,k},\quad J=(J^1, J^2)\in \mathbb N^2.
\] 
\begin{lem}
Suppose that $m\in \mathcal M_Z^1$. Let $m_J$ be defined as above and let $K_J=\check m_J$. Then for $(y_2, z_2)\neq 0 \neq (y_3, z_3)$ we have the estimate 
\begin{align*}
&\Big| \iiint_{I^1\times I^1 \times I^1}\partial^{\beta_2}_{y_2}\partial^{\beta_3}_{y_3}\partial^{\gamma_2}_{z_2}\partial^{\gamma_3}_{z_3}K_J (x_1-y_1, y_2, y_3, x_1-z_1, z_2, z_3)\ud y_1 \ud z_1 \ud x_1 \Big|\\
&\lesssim \frac 1{(|y_2|+|z_2|)^{\beta_2+\gamma_2}}\cdot \frac 1{(|y_3|+|z_3|)^{\beta_3+\gamma_3}}|I^1| (\frac{|I^1|(|y_2|+|z_2|)}{|y_3|+|z_3|} + \frac{|y_3|+|z_3|}{|I^1|(|y_2|+|z_2|)} )^{-1}\\
&\qquad \times \frac {1}{\prod_{i=2}^3 (|y_i|+|z_i|)^2 }\cdot \big(1+\log_+ \frac{|y_3|+|z_3|}{|I^1|(|y_2|+|z_2|)}\big),
\end{align*}
where $I^1$ is an interval and $\beta_2+\beta_3+\gamma_2+\gamma_3\le 1$.
\end{lem}
\begin{proof}
Since $m_J(0, \xi_2, \xi_3, 0, \eta_2, \eta_3)=0$, using the Fourier transform we know that 
\begin{equation}\label{eq:1204}
\iint_{\R^2 } \partial^{\beta_2}_{y_2}\partial^{\beta_3}_{y_3}\partial^{\gamma_2}_{z_2}\partial^{\gamma_3}_{z_3}K_J (y_1, y_2, y_3, z_1, z_2, z_3)\ud y_1 \ud z_1=0.
\end{equation}
Suppose first that $|I^1|(|y_2|+|z_2|)\ge |y_3|+|z_3|$ -- by \eqref{eq:1204} we may equivalently estimate the integral over $I^1\times (\R^2 \setminus (I^1\times I^1))$ instead of $I^1\times I^1 \times I^1$. By the kernel estimates we have 
\begin{align*}
&\Big| \iiint_{I^1\times (\R^2 \setminus (I^1 \times I^1))}\partial^{\beta_2}_{y_2}\partial^{\beta_3}_{y_3}\partial^{\gamma_2}_{z_2}\partial^{\gamma_3}_{z_3}K_J (x_1-y_1, y_2, y_3, x_1-z_1, z_2, z_3)\ud y_1 \ud z_1 \ud x_1 \Big|\\
&\lesssim \int_{I^1}\iint_{\R^2 \setminus (I^1 \times I^1)}\frac 1{(|y_2|+|z_2|)^{ \beta_2+\gamma_2}} \\
&\hspace{1cm}\times \frac {1+\log_+ \frac{|y_3|+|z_3|}{(|x_1-y_1|+|x_1-z_1|)(|y_2|+|z_2|)}}{[(|x_1-y_1|+|x_1-z_1|)(|y_2|+|z_2|)+|y_3|+|z_3|]^{\beta_3+\gamma_3+4}}\ud y_1 \ud z_1 \ud x_1.
\end{align*}
Note that we have either $y_1\in \R\setminus I^1$ or $z_1\in \R\setminus I^1$, and we may without loss of generality assume $y_1\in \R\setminus I^1$. Then the integral is dominated by 
\begin{align*}
&\int_{I^1}\iint_{(\R \setminus  I^1)\times \R}\frac 1{(|y_2|+|z_2|)^{ \beta_2+\gamma_2}} \\
&\hspace{1cm}\times \frac {1+\log_+ \frac{|y_3|+|z_3|}{|x_1-y_1|(|y_2|+|z_2|)}}{[(|x_1-y_1|+|x_1-z_1|)(|y_2|+|z_2|)+|y_3|+|z_3|]^{\beta_3+\gamma_3+4}}\ud y_1 \ud z_1 \ud x_1\\
&\hspace{1cm}\lesssim \frac 1{(|y_2|+|z_2|)^{ \beta_2+\gamma_2+\beta_3+\gamma_3+4}}
\int_{I^1} \int_{\R \setminus I^1} \frac  {1+\log_+ \frac{|y_3|+|z_3|}{|x_1-y_1|(|y_2|+|z_2|)}} {\big(|x_1-y_1|+ \frac{|y_3|+|z_3|}{|y_2|+|z_2|}\big)^{\beta_3+\gamma_3+3}}
\ud y_1 \ud x_1.
\end{align*}
Let $t:=  \frac{|y_3|+|z_3|}{|y_2|+|z_2|}$. By a change of variables we reduce to 
\begin{align*}
& \frac {t^{-\beta_3-\gamma_3-1}}{(|y_2|+|z_2|)^{ \beta_2+\gamma_2+\beta_3+\gamma_3+4}}\iint_{t^{-1} I^1\times (\R \setminus t^{-1}I^1) } \frac  {1+\log_+ \frac{1}{|x_1-y_1|}} {\big(|x_1-y_1|+ 1\big)^{\beta_3+\gamma_3+3}}
\ud y_1 \ud x_1\\
&\lesssim  \frac {t^{-\beta_3-\gamma_3-1}}{(|y_2|+|z_2|)^{ \beta_2+\gamma_2+\beta_3+\gamma_3+4}}
\int_{t^{-1}I^1}\frac 1{\big(d(x_1, \partial(t^{-1}I^1))+1\big)^{\beta_3+\gamma_3+2}}\ud x_1\\
&\lesssim  \frac {t^{-\beta_3-\gamma_3-1}}{(|y_2|+|z_2|)^{ \beta_2+\gamma_2+\beta_3+\gamma_3+4}}\\
&= \frac 1{(|y_2|+|z_2|)^{\beta_2+\gamma_2+3}}\frac 1{(|y_3|+|z_3|)^{\beta_3+\gamma_3+1}}\\
&\sim  \frac 1{(|y_2|+|z_2|)^{\beta_2+\gamma_2}}\cdot \frac 1{(|y_3|+|z_3|)^{\beta_3+\gamma_3}}|I^1| (\frac{|I^1|(|y_2|+|z_2|)}{|y_3|+|z_3|} + \frac{|y_3|+|z_3|}{|I^1|(|y_2|+|z_2|)} )^{-1}\\
&\qquad \times \frac {1}{\prod_{i=2}^3 (|y_i|+|z_i|)^2 }.
\end{align*}

Assume then that $|I^1|(|y_2|+|z_2|)< |y_3|+|z_3|$. This time we integrate over $I^1\times I^1\times I^1$. Proceeding as above we reduce to the integral  
\begin{align*}
&\iiint\limits_{t^{-1}I^1 \times t^{-1}I^1\times t^{-1}I^1}\frac {t^{-\beta_3-\gamma_3-1}}{(|y_2|+|z_2|)^{ \beta_2+\gamma_2+\beta_3+\gamma_3+4 }} \frac {1+\log_+ \frac{1}{(|x_1-y_1|+|x_1-z_1|)}}{[(|x_1-y_1|+|x_1-z_1|) +1]^{\beta_3+\gamma_3+4}}\ud y_1 \ud z_1 \ud x_1 \\
&\le \iint\limits_{  t^{-1}I^1\times t^{-1}I^1}\frac {t^{-\beta_3-\gamma_3-1}}{(|y_2|+|z_2|)^{ \beta_2+\gamma_2+\beta_3+\gamma_3+4 }} \frac {1+\log_+ \frac{1}{ |x_1-y_1| }}{( |x_1-y_1| +1)^{\beta_3+\gamma_3+3}}\ud y_1  \ud x_1 \\
&\sim \frac {t^{-\beta_3-\gamma_3-1}}{(|y_2|+|z_2|)^{ \beta_2+\gamma_2+\beta_3+\gamma_3+4 }}\iint\limits_{  t^{-1}I^1\times t^{-1}I^1}  \big( 1+\log_+ \frac{1}{ |x_1-y_1| }\big)\ud y_1  \ud x_1\\
&\lesssim \frac {t^{-\beta_3-\gamma_3-1}}{(|y_2|+|z_2|)^{ \beta_2+\gamma_2+\beta_3+\gamma_3+4 }} (t^{-1}|I^1|)^2 (1+\log_+ (t|I^1|^{-1}))\\
&= \frac 1{(|y_2|+|z_2|)^{\beta_2+\gamma_2}}\cdot \frac 1{(|y_3|+|z_3|)^{\beta_3+\gamma_3}}|I^1| (\frac{|I^1|(|y_2|+|z_2|)}{|y_3|+|z_3|} + \frac{|y_3|+|z_3|}{|I^1|(|y_2|+|z_2|)} )^{-1}\\
&\qquad \times \frac {1}{\prod_{i=2}^3 (|y_i|+|z_i|)^2 }\cdot \big(1+\log_+ \frac{|y_3|+|z_3|}{|I^1|(|y_2|+|z_2|)}\big).
\end{align*}
Thus, we are done.
\end{proof}
With \eqref{eq:blinmulti} at hand, similarly as in the linear case we can derive the following. 
\begin{lem}\label{lem:par}
Let $m \in \calM^1_Z$ and 
denote by $T_{m}$ the corresponding Fourier multiplier operator.

Let $f_1,g_1  \in L^4(\R)$, $f_{2,3}, g_{2,3} \in L^4(\R^2)$ and $h_1\in L^2(\R)$, $h_{2,3}\in L^2(\R^2)$.
Then
$$
\langle T_{m} (f_1 \otimes f_{2,3}, g_1 \otimes g_{2,3}), h_1\otimes h_{2,3} \rangle
= \ave{T_{m_{f_{2,3}, g_{2,3}, h_{2,3}}} (f_1,g_1), h_1},
$$ 
where $m_{ f_{2,3}, g_{2,3}, h_{2,3}}$ is a standard bilinear Coifman-Meyer multiplier in $\R$ satisfying the estimates
\begin{align*}
  |(\ud/\ud \xi_1)^{\alpha}(\ud/\ud \eta_1)^{\beta} &m_{f_{2,3}, g_{2,3}, h_{2,3}}(\xi_1,\eta_1)| \\
&\lesssim \| m \|_{\calM_Z^1}\|f_{2,3}\|_{L^4} \| g_{2,3}\|_{L^4} \|h_{2,3}\|_{L^2} (|\xi_1|+|\eta_1|)^{-\alpha-\beta}.
\end{align*}
Thus,  $T_{m_{f_{2,3}, g_{2,3}, h_{2,3}}}$ is a convolution form bilinear Calder\'on-Zygmund operator.
In particular, there exists a standard bilinear Calder\'on-Zygmund kernel $K_{m, f_{2,3}, g_{2,3}, h_{2,3}}$ 
such that  \[\|K_{m, f_{2,3}, g_{2,3}, h_{2,3}}\|_{CZ_1(\R^2)} \lesssim \|f_{2,3}\|_{L^4} \| g_{2,3}\|_{L^4}\|h_{2,3}\|_{L^2}.\] Moreover, if
$\supp f_1 \cap \supp g_1 \cap \supp h_1=\emptyset$, then
\begin{align*}
  \langle T_{m} (f_1 \otimes f_{2,3}, g_1 &\otimes g_{2,3}), h_1\otimes h_{2,3} \rangle \\
&= \iiint K_{m, f_{2,3}, g_{2,3}, h_{2,3}}(x_1,y_1, z_1) f_1(y_1) g_1(z_1) h_1(x_1)\ud y_1 \ud z_1\ud x_1.
\end{align*}
\end{lem}
\subsection*{Conclusion}
Notice that Lemma \ref{lem:par} immediately implies that we have the weak boundedness property.
Therefore, the bilinear multipliers satisfy  Definition \ref{defn:czz}.

\bibliography{refH}

\begin{bibdiv}
\begin{biblist}

\bib{AMV}{article}{
      author={Airta, Emil},
      author={Martikainen, Henri},
      author={Vuorinen, Emil},
       title={Product space singular integrals with mild kernel regularity},
        date={2022},
     journal={J. Geom. Anal.},
      volume={32},
       pages={article number 24},
}

\bib{CRW}{article}{
      author={Coifman, R.~R.},
      author={Rochberg, R.},
      author={Weiss, Guido},
       title={Factorization theorems for {H}ardy spaces in several variables},
        date={1976},
        ISSN={0003-486X},
     journal={Ann. of Math. (2)},
      volume={103},
      number={3},
       pages={611\ndash 635},
         url={https://doi.org/10.2307/1970954},
      review={\MR{412721}},
}

\bib{CUMP}{article}{
      author={Cruz-Uribe, David},
      author={Martell, Jos\'e~Mar\'ia},
      author={P\'{e}rez, Carlos},
       title={Extrapolation from {$A_\infty$} weights and applications},
        date={2004},
        ISSN={0022-1236},
     journal={J. Funct. Anal.},
      volume={213},
      number={2},
       pages={412\ndash 439},
         url={https://doi.org/10.1016/j.jfa.2003.09.002},
      review={\MR{2078632}},
}

\bib{DU}{article}{
      author={Duoandikoetxea, J.},
       title={Extrapolation of weights revisited: new proofs and sharp bounds},
        date={2011},
        ISSN={0022-1236},
     journal={J. Funct. Anal.},
      volume={260},
      number={6},
       pages={1886\ndash 1901},
         url={https://doi.org/10.1016/j.jfa.2010.12.015},
}

\bib{DLOPW-ZYGMUND}{article}{
      author={Duong, Xuan~Thinh},
      author={Li, Ji},
      author={Ou, Yumeng},
      author={Pipher, Jill},
      author={Wick, Brett},
       title={Weighted estimates of singular integrals and commutators in the {Z}ygmund dilation setting},
        date={2019},
     journal={preprint, arXiv:1905.00999},
}

\bib{FEPI}{article}{
      author={Fefferman, R.},
      author={Pipher, J.},
       title={Multiparameter operators and sharp weighted inequalities},
        date={1997},
        ISSN={0002-9327},
     journal={Amer. J. Math.},
      volume={119},
      number={2},
       pages={337\ndash 369},
      review={\MR{1439553}},
}

\bib{GM}{article}{
      author={Grafakos, Loukas},
      author={Martell, Jos\'{e}~Mar\'{\i}a},
       title={Extrapolation of weighted norm inequalities for multivariable operators and applications},
        date={2004},
        ISSN={1050-6926},
     journal={J. Geom. Anal.},
      volume={14},
      number={1},
       pages={19\ndash 46},
         url={https://doi.org/10.1007/BF02921864},
      review={\MR{2030573}},
}

\bib{GO}{article}{
      author={Grafakos, Loukas},
      author={Oh, Seungly},
       title={The {K}ato-{P}once inequality},
        date={2014},
        ISSN={0360-5302},
     journal={Comm. Partial Differential Equations},
      volume={39},
      number={6},
       pages={1128\ndash 1157},
         url={https://doi.org/10.1080/03605302.2013.822885},
      review={\MR{3200091}},
}

\bib{GT}{article}{
      author={Grafakos, Loukas},
      author={Torres, Rodolfo~H.},
       title={Multilinear {C}alder\'{o}n-{Z}ygmund theory},
        date={2002},
        ISSN={0001-8708},
     journal={Adv. Math.},
      volume={165},
      number={1},
       pages={124\ndash 164},
         url={https://doi.org/10.1006/aima.2001.2028},
      review={\MR{1880324}},
}

\bib{GH}{article}{
      author={Grau de~la Herr\'{a}n, Ana},
      author={Hyt\"{o}nen, Tuomas},
       title={Dyadic representation and boundedness of nonhomogeneous {C}alder\'{o}n-{Z}ygmund operators with mild kernel regularity},
        date={2018},
        ISSN={0026-2285},
     journal={Michigan Math. J.},
      volume={67},
      number={4},
       pages={757\ndash 786},
         url={https://doi.org/10.1307/mmj/1531447374},
      review={\MR{3877436}},
}

\bib{HLLT}{article}{
      author={Han, Yongsheng},
      author={Li, Ji},
      author={Lin, Chin-Cheng},
      author={Tan, Chaoqiang},
       title={Singular {I}ntegrals {A}ssociated with {Z}ygmund {D}ilations},
        date={2019},
     journal={J. Geom. Anal.},
      volume={29},
       pages={2410\ndash 2455},
}

\bib{HPW}{article}{
      author={Holmes, Irina},
      author={Petermichl, Stefanie},
      author={Wick, Brett~D.},
       title={Weighted little bmo and two-weight inequalities for {J}ourn\'{e} commutators},
        date={2018},
        ISSN={2157-5045},
     journal={Anal. PDE},
      volume={11},
      number={7},
       pages={1693\ndash 1740},
         url={https://doi.org/10.2140/apde.2018.11.1693},
      review={\MR{3810470}},
}

\bib{Hy}{article}{
      author={Hyt\"{o}nen, Tuomas},
       title={The sharp weighted bound for general {C}alder\'{o}n-{Z}ygmund operators},
        date={2012},
        ISSN={0003-486X},
     journal={Ann. of Math. (2)},
      volume={175},
      number={3},
       pages={1473\ndash 1506},
         url={https://doi.org/10.4007/annals.2012.175.3.9},
      review={\MR{2912709}},
}

\bib{HLMV-ZYG}{article}{
      author={Hyt\"onen, Tuomas},
      author={Li, Kangwei},
      author={Martikainen, Henri},
      author={Vuorinen, Emil},
       title={Multiresolution analysis and {Z}ygmund dilations},
        date={2022},
     journal={preprint, arXiv:2203.15777},
}

\bib{KP}{article}{
      author={Kato, Tosio},
      author={Ponce, Gustavo},
       title={Commutator estimates and the {E}uler and {N}avier-{S}tokes equations},
        date={1988},
        ISSN={0010-3640},
     journal={Comm. Pure Appl. Math.},
      volume={41},
      number={7},
       pages={891\ndash 907},
         url={https://doi.org/10.1002/cpa.3160410704},
      review={\MR{951744}},
}

\bib{LOPTT}{article}{
      author={Lerner, Andrei},
      author={Ombrosi, Sheldy},
      author={P\'erez, Carlos},
      author={Torres, Rodolfo},
      author={Trujillo-Gonz\'alez, Rodrigo},
       title={New maximal functions and multiple weights for the multilinear {C}alder\'on--{Z}ygmund theory},
        date={2009},
     journal={Adv. Math.},
      volume={220},
      number={4},
       pages={1222\ndash 1264},
}

\bib{LMOV}{article}{
      author={Li, Kangwei},
      author={Martikainen, Henri},
      author={Ou, Yumeng},
      author={Vuorinen, Emil},
       title={Bilinear representation theorem},
        date={2019},
        ISSN={0002-9947},
     journal={Trans. Amer. Math. Soc.},
      volume={371},
      number={6},
       pages={4193\ndash 4214},
         url={https://doi.org/10.1090/tran/7505},
      review={\MR{3917220}},
}

\bib{LMV1}{article}{
      author={Li, Kangwei},
      author={Martikainen, Henri},
      author={Vuorinen, Emil},
       title={Bilinear {C}alder\'{o}n-{Z}ygmund theory on product spaces},
        date={2020},
        ISSN={0021-7824},
     journal={J. Math. Pures Appl. (9)},
      volume={138},
       pages={356\ndash 412},
         url={https://doi.org/10.1016/j.matpur.2019.10.007},
      review={\MR{4098772}},
}

\bib{LMV-GEN}{article}{
      author={Li, Kangwei},
      author={Martikainen, Henri},
      author={Vuorinen, Emil},
       title={Genuinely multilinear weighted estimates for singular integrals in product spaces},
        date={2021},
     journal={Adv. Math.},
      volume={393},
       pages={108099},
         url={https://doi.org/10.1016/j.aim.2021.108099},
}

\bib{MA_ADV}{article}{
      author={Martikainen, Henri},
       title={Representation of bi-parameter singular integrals by dyadic operators},
        date={2012},
        ISSN={0001-8708},
     journal={Adv. Math.},
      volume={229},
      number={3},
       pages={1734\ndash 1761},
         url={https://doi.org/10.1016/j.aim.2011.12.019},
      review={\MR{2871155}},
}

\bib{MRS}{article}{
      author={M\"{u}ller, Detlef},
      author={Ricci, Fulvio},
      author={Stein, Elias~M.},
       title={Marcinkiewicz multipliers and multi-parameter structure on {H}eisenberg (-type) groups. {I}},
        date={1995},
        ISSN={0020-9910},
     journal={Invent. Math.},
      volume={119},
      number={2},
       pages={199\ndash 233},
         url={https://doi.org/10.1007/BF01245180},
      review={\MR{1312498}},
}

\bib{NW}{article}{
      author={Nagel, Alexander},
      author={Wainger, Stephen},
       title={{$L^2$} boundedness of {H}ilbert transforms along surfaces and convolution operators homogeneous with respect to a multiple parameter group},
        date={1977},
     journal={Amer. J. Math.},
      volume={99},
       pages={761\ndash 785},
}

\bib{NTV}{article}{
      author={Nazarov, Fedor},
      author={Treil, Sergei},
      author={Volberg, Alexander},
       title={The {$Tb$}-theorem on non-homogeneous spaces},
        date={2003},
        ISSN={0001-5962},
     journal={Acta Math.},
      volume={190},
      number={2},
       pages={151\ndash 239},
         url={https://doi.org/10.1007/BF02392690},
      review={\MR{1998349}},
}

\bib{RS}{article}{
      author={Ricci, F.},
      author={Stein, E.~M.},
       title={Multiparameter singular integrals and maximal functions},
        date={1992},
        ISSN={0373-0956},
     journal={Ann. Inst. Fourier (Grenoble)},
      volume={42},
      number={3},
       pages={637\ndash 670},
      review={\MR{1182643}},
}

\end{biblist}
\end{bibdiv}

\end{document}